%% file: comp_p_units_arxiv.tex
\documentclass[12pt,leqno,oneside]{article} 

\usepackage{amsfonts}
\usepackage{latexsym}
\usepackage{amsmath}
\usepackage{amssymb}
\usepackage{textcomp}
\usepackage{theorem}
\usepackage[dvips]{graphicx,color}
\usepackage{makeidx}
\usepackage[all,2cell,dvips]{xy}
\numberwithin{equation}{section}

{\theorembodyfont{\sl} \newtheorem{Def}{\indent Definition}[section]
{\theorembodyfont{\sl} \newtheorem{Lemma}{\indent Lemma}[section]
{\theorembodyfont{\sl} \newtheorem{Th}{\indent Theorem }[section]
{\theorembodyfont{\sl} \newtheorem{Cor}{\indent Corollary}[section]
{\theorembodyfont{\sl} \newtheorem{Prop}{\indent Proposition}[section]
{\theorembodyfont{\rmfamily} \newtheorem{Rem}{\indent Remark}[section]}
{\theorembodyfont{\sl}\newtheorem{Conj}{\indent Conjecture}[section]
{\theorembodyfont{\rmfamily} \newtheorem{Assumption}{\indent Assumption}[section]}

\newenvironment{proof}[1][Proof]{\begin{trivlist}
\item[\hskip \labelsep {\bfseries #1}]}{\end{trivlist}}

\newcommand{\ZZ}{\mathbb{Z}}
\newcommand{\QQ}{\mathbb{Q}}

\newcommand{\RR}{\mathbb{R}}
\newcommand{\CC}{\mathbb{C}}

\newcommand{\PP}{\mathbb{P}}
\newcommand{\XX}{\mathbb{X}}

\newcommand{\s}{\hspace{3pt}}

\newcommand{\laa}{\langle}
\newcommand{\raa}{\rangle}
\newcommand{\bs}{\backslash}

\newcommand{\fin}{\hspace{4 pt} $\square$}

\newcommand{\ca}{\mathcal}
\newcommand{\mk}{\mathfrak}
\newcommand{\wt}{\widetilde}
\newcommand{\wh}{\widehat}
\newcommand{\ord}{\textrm{ord}}

\newcommand{\Norm}{\mathbf{N}}

\newcommand{\indicator}{1\hspace{-4pt}1}
\newcommand{\M}[4]
{\left(
\begin{array}{cc}
#1 & #2 \\
#3 & #4
\end{array} 
\right)}
\newcommand{\V}[2]
{\left(
\begin{array}{c}
#1\\
#2
\end{array} \right)}
\newcommand{\R}[2]{\left(
\begin{array}{cc}
#1 & #2\\
\end{array} 
\right)}

\newcommand{\mint}[2]{\times \hspace{-0.16in} \int_{#1}^{#2}} %% multiplicative integral
\DeclareMathOperator{\disc}{disc}
\DeclareMathOperator{\Gal}{Gal}
\DeclareMathOperator{\End}{End}
\DeclareMathOperator{\cond}{cond}
\DeclareMathOperator{\Div}{Div}
\DeclareMathOperator{\Deg}{deg}
\DeclareMathOperator{\sign}{sign}

\begin{document}
\title{Relationships between $p$-unit constructions for real quadratic fields}
\author{Hugo Chapdelaine}
\date{Avril, 2010}
\maketitle
\begin{abstract}
Let $K$ be a real quadratic field and let $p$ be a prime number which is inert in $K$. Let $K_p$ be the completion
of $K$ at $p$. In a previous paper, we constructed a $p$-adic invariant $u_C\in K_p^{\times}$,
and we proved a $p$-adic Kronecker limit formula relating $u_C$ to the first derivative at $s=0$
of a certain $p$-adic zeta function. By analogy with the $p$-adic Gross-Stark conjectures, 
we conjectured that $u_C$ is a $p$-unit in a suitable narrow ray class field of $K$.
Recently, Dasgupta has proposed an exact 
$p$-adic formula for the Gross-Stark units of an arbitrary totally real number
field. In our special setting, i.e., where one deals with a real quadratic number field, 
his construction
produces a $p$-adic invariant $u_D\in K_p^{\times}$. 
In this paper we show precise relationships between the $p$-adic invariants 
$u_C$ and $u_D$. In order to do so, we extend Dasgupta's construction of $u_D$
to a broader setting.
\end{abstract}
\tableofcontents
\section{Introduction}
Let $K$ be a real quadratic number field of discriminant $d_K$.
Let $p$ be a prime number which is inert in $K$ and denote by $K_p$ the completion 
of $K$ at $p$. Let $\ca{H}_p=\PP^1(\CC_p)\bs\PP^1(\QQ_p)$ be the
$p$-adic upper half plane endowed with the structure of a rigid analytic space. 
Note that $\ca{H}_p\cap K=K\bs\QQ$ since $p$ was assumed to be inert in $K$.
We denote the maximal order of $K$ by $\ca{O}_K=\ZZ+\omega\ZZ$ where $\omega=\frac{d_K+\sqrt{d_K}}{2}$ 
and its unique order of conductor $n\geq 1$ by 
$\ca{O}_n=\ZZ+n\omega\ZZ$. For $\tau\in\ca{H}_p\cap K$ we let $\Lambda_{\tau}=\ZZ+\tau\ZZ$ and 
$\ca{O}_{\tau}=\End_K(\Lambda_{\tau})=\{\lambda\in K:\lambda\Lambda_{\tau}\subseteq\Lambda_{\tau}\}$. 
Let $N$ be a positive square-free integer coprime to $d_K$. The square-free condition
holds only for the introduction since its simplifies the presentation. Now
make the following crucial assumption: 
there exists an ideal $\mk{N}$ such that $\ca{O}_K/\mk{N}\simeq\ZZ/N\ZZ$. 
The last condition is equivalent to saying that $N=\prod_{i=1}^r l_i$ where
each $l_i$ is a prime number which splits in $K$, the $l_i$'s being distinct. The latter assumption is
often called the \textit{Heegner hypothesis} by analogy with the construction of Heegner points
on modular elliptic curves (see for example Chap. 3 of \cite{Dar2}). 
Now choose a positive integer $n$ which is coprime to $pNd_K$. Let 
$\alpha(z)$ be a modular unit for the congruence group $\Gamma_0(N)$ which has no zeros
nor poles on the set $\Gamma_0(N)\{\infty\}$ where $\infty$ corresponds to the cusp
$\frac{1}{0}$ (for an example of such modular unit see equation \eqref{beate} in Appendix \ref{goeland}). 
Let $\tau\in\ca{H}_p\cap K$ be such that $\ca{O}_{\tau}=\ca{O}_{N\tau}
=\ca{O}_n$. In \cite{Dar-Das}, Darmon and Dasgupta constructed a $p$-adic invariant $u_{DD}(\alpha,\tau)\in K_p^{\times}$.
They conjectured that $u_{DD}(\alpha,\tau)$ is a global $p$-unit in $K(\ca{O}_n\infty)$, 
the \textit{narrow ring class field} of $K$ of
conductor $n$. From now on we assume that the modular unit $\alpha(z)$ is fixed so we omit it 
from the notation and therefore we write $u_{DD}(\tau)$ instead of $u_{DD}(\alpha,\tau)$.

Conjecturally, the construction mentioned above produces $p$-units in the \textit{ring class fields}
of $K$. It is then natural to try to extend it to the larger 
setting of \textit{ray class fields} of $K$.
This was accomplished in the PhD thesis of the author, see \cite{Ch_T} and \cite{Ch1}. 
Let us explain very briefly what kind of units we expect to construct in this more general setting. 
Let $f$ and $n$ be fixed positive integers which are prime to $pNd_K$. Let $\{\beta_{r}(z)\}_{r\in\ZZ/f\ZZ}$ be a
family of modular units associated to $\alpha(z)$ (see Appendix \ref{goeland} 
for more details). By definition, $\beta_r(z)$ has no zeros nor poles 
on the set $\Gamma_0(fN)\{\infty\}$ and moreover one has that 
\begin{align}\label{fromulaa}
\prod_{r=0}^{f-1}\beta_r(z)=\alpha(z).
\end{align}
In \cite{Ch1}, we proposed a construction of a $p$-adic invariant 
$u_C(\beta_r,\tau)\in K_p^{\times}$ for certain pairs 
$(r,\tau)\in\ZZ/f\ZZ\times(\ca{H}_p\cap K)$ such that $\ca{O}_{\tau}=\ca{O}_{N\tau}=\ca{O}_n$
(See Section \ref{def_invariants} for the definition of $u_C(r,\tau)$). 
From now on we simply denote $u_C(\beta_r,\tau)$ by $u_C(r,\tau)$. 
The elements $u_C(r,\tau)$ are conjectured to be global $p$-units in $K(\mk{f}_n\infty)^{\laa\sigma_{\wp_n}\raa}$ 
where $\mk{f}_n=f\ca{O}_n$, $\wp_n=p\ca{O}_n$, $K(\mk{f}_n\infty)$ 
is the \textit{narrow extended class field of conductor $\mk{f}_n$ of $K$},
and $K(\mk{f}_n\infty)^{\laa\sigma_{\wp_n}\raa}$ is the subfield of $K(\mk{f}_n\infty)$
which is fixed by the Frobenius $\sigma_{\wp_n}$ at $\wp_n$. 
(For the precise definition of $K(\mk{f}_n\infty)$ see Definition \ref{affine}). 
In the case where $n=1$, so
that $\mk{f}_1=f\ca{O}_K$, the abelian extension $K(\mk{f}_1\infty)$ is nothing else than the
usual narrow ray class field of $K$ of conductor $f$. Note that 
$K(\ca{O}_n\infty)\subseteq K(\mk{f}_n\infty)^{\laa\sigma_{\wp_n}\raa}$ and that
$K(\mk{f}_n\infty)^{\laa\sigma_{\wp_n}\raa}$ corresponds to the largest subfield of $K(\mk{f}_n\infty)$
for which $\wp_n$ splits completely. Our construction of $u_C(r,\tau)$ may be viewed as a natural generalization of
$u_{DD}(\tau)$ in the following sense:
\begin{enumerate}
\item If $f=1$ then $u_{C}(0,\tau)=u_{DD}(\tau)$.
\item Moreover, if one sets $n=f$ then
\begin{align}\label{Norm}
\prod_{r=0}^{f-1}u_C(r,\tau)=u_{DD}(\tau).
\end{align}
\end{enumerate}
Note that $(2)$ is equivalent to $(1)$ in the case where $f=n=1$. The identity \eqref{Norm}
is a direct consequence of the identity \eqref{fromulaa}. See Appendix \ref{goeland} (and also Proposition \ref{ano_way})
where this is explained in greater details.

We want to say here a few words about the method that was used to construct
the $p$-adic invariants $u_C$ and $u_{DD}$. 
One of the key feature of the method is to use
periods of a family of Eisenstein series which varies in the weight. If 
one renormalizes these Eisenstein series in order to clear the 
\lq\lq transcendental period\rq\rq (a certain power of $2\pi i$ which depends on the weight) 
one obtains rational numbers.
One of the breakthrough that appeared in \cite{Dar-Das} was to realize that these
rational periods can be \lq\lq packaged\rq\rq\; in a certain way in order to construct 
a partial modular symbol of $\ZZ$-valued measures on $\XX$. Here $\XX=(\ZZ_p\times\ZZ_p)\bs 
(p\ZZ_p\times p\ZZ_p)$ is the set of primitive vectors of $\ZZ_p^2$. The invariants
 $u_{DD}(\tau)$ and $u_C(r,\tau)$ are then defined as a certain $p$-adic integral of such a
measure on the space $\XX$. 
Many numerical examples which support the algebraicity of 
$u_{DD}(\tau)$ and $u_C(r,\tau)$ can be found in \cite{Das2} and \cite{Ch3}. 
Note that it is desirable to have a theory which allows $N$ to be non-square free since 
in all the numerical examples which appear in the two papers above one has $N=4$.

On the other hand well known conjectures of Gross and Stark
relate special values of abelian $L$-functions to global units. Let $F$ be a totally real number field. 
In the seventies, Stark formulated a series of conjectures relating 
the special values at $s=0$ of Archimedean abelian $L$-functions of $F$ to global units in $F^{ab}$.
Then in the early eighties (see \cite{Gr1}), Gross formulated a $p$-adic 
analogue of Stark's conjecture by replacing Archimedean abelian $L$-functions of $F$
by $p$-adic abelian $L$-functions of $F$, and global units in $F^{ab}$ by global $p$-units in $F^{ab}$. 
A few years later, Gross made a refinement of the previous conjecture. This refined
conjecture will be referred as the \lq\lq Strong Gross conjecture\rq\rq (for a precise 
statement see Conjecture \ref{Gross_strong} below). It turns out that
the Strong Gross Conjecture is closely related to the Darmon-Dasgupta 
construction explained above and therefore we would like to recall it.
In order to do so we need to set up some notation.

Let $\mk{f}$ be an integral ideal of
$F$ and let $M=F(\mk{f}\infty)$ be the narrow ray class field of $F$ of conductor $\mk{f}$. 
Let $\mk{p}$ be a prime of 
$F$ above $p$ that does not ramify in $M$. 
Let $S$ be a finite set of places of $F$ containing all the infinite places
of $F$, the prime $\mk{p}$ and the primes which ramify in $M/F$. Now we 
make following crucial assumption on $S$:
\begin{Assumption}\label{ass_1}
The only place of $S$ (Archimedean or non-Archimedean) that splits completely in $M$ is $\mk{p}$.
\end{Assumption}
In particular, because of this assumption on $S$, $M$ has to be a totally imaginary number field.
For an ideal $\mk{a}$ of $F$ we denote by $\sigma_{\mk{a}}\in \Gal(M/F)$ the 
Frobenius at $\mk{a}$. For any $\sigma\in \Gal(M/F)$ let
\begin{align}\label{rabbit}
\zeta_S(M/F,\sigma,s):=\sum_{\substack{(\mk{a},S)=1\\ \sigma_{\mk{a}}=\sigma}}\frac{1}{\Norm(\mk{a})^s}
=\left(1-\frac{1}{\Norm(\mk{p})^s}\right)\zeta_R(M/F,\sigma,s),
\end{align}
where $R=S\bs\{\mk{p}\}$. The first 
summations (resp. the summation which defines $\zeta_R(\sigma,s)$) is taken over all integral ideals of $F$ coprime 
to $S$ (resp. coprime to $R$). The second equality follows from Assumption \ref{ass_1}
and implies that $ord_{s=0}\zeta_S(\sigma,s)=1$. It was proved by Siegel and Klingen that
the special values at negative integers of 
$\zeta_R(\sigma,s)$ are \textit{rational numbers} (see \cite{Kli} and \cite{Sie3}). This key 
fact is the starting point for the $p$-adic Gross-Stark conjecture.

In order to state the version of Gross's conjecture which is convenient in our context, 
we need to introduce one more
set of places of $F$ in order to \lq\lq regularize\rq\rq the special values at negative integers
of $\zeta_{R}(M/F,\sigma,s)$. Let $T$ be a finite set of places of $F$ which is disjoint from $S$.
The role played by the set $T$ here is similar to the role played by the integer $N$
in the Darmon-Dasgupta construction. For each complex number $s$, consider the group ring element
\begin{align}\label{talus}
\prod_{\eta_i\in T}
(1-(\Norm\eta_i)^{1-s}[\sigma_{\eta_i}])=\sum_{\mk{d}|\mk{N}}(-1)^{\wt{\mk{e}}(\mk{d})}
\Norm(\mk{d})^{1-s}[\sigma_\mk{d}]\in\CC[\Gal(M/F)],
\end{align}
where $\mk{N}=\prod_{\eta_i\in T}\eta_i$ and $\wt{\mk{e}}(\mk{d})=\mbox{{$\#$ of prime divisors of $\mk{d}$}}$. 
For $\sigma\in \Gal(M/F)$, define the partial zeta function
associated to the sets $S$ and $T$ by the group ring equation
\begin{align}\label{protugal}
\zeta_{S,T}(M/F,\sigma,s)=\sum_{\mk{d}|\mk{N}}(-1)^{\wt{\mk{e}}(\mk{d})}\Norm(\mk{d})^{1-s}\zeta_{S}(M/F,\sigma\sigma_{\mk{d}}^{-1},s).
\end{align}

In order to force the integrality at $s=0$,
Gross introduced the following assumption
\begin{Assumption}\label{ass_inter}
The set $T$ contains at least two primes of different residue characteristic or at least one prime
$\eta$ ($\Norm\eta=l$) with absolute ramification of degree at most $l-2$.
\end{Assumption}

In Appendix \ref{appendix1} the interested reader may find some key properties of special 
values of abelian $L$-functions at negative integers attached to $F$ which 
motivate Assumption \ref{ass_inter}.

The Strong Gross Conjecture predicts the existence of a special kind of $\mk{p}$-units in $M$
for which their $\mk{p}$-adic valuation can be related to the value $\zeta_{R,T}(M/F,\sigma,0)\in\ZZ$.
Define the following subgroup of the group of $\mk{p}$-units of $M$:
\begin{align*}
U_{\mk{p}}:=\{x\in M^{\times}:|x|_{\nu}=1\s\textrm{for all $\nu\nmid\mk{p}$}\s\}.
\end{align*} 
Here $\nu$ ranges over all \textit{finite and infinite} places of $M$. 
An element of $U_{\mk{p}}$ will be called a \textit{strong $\mk{p}$-unit}. This terminology is
justified for the following two reasons: First of all, an element $u\in U_{\mk{p}}$ is necessarily a 
global $p$-unit, i.e., $u\in\ca{O}_M[\frac{1}{p}]^{\times}$. But it is more (stronger) than just being a $p$-unit
since for all complex embeddings $\tau:M\rightarrow\CC$ one has that $|u^{\tau}|=1$, i.e.,
the image of $u$ by any complex embedding lies on the unit circle. There is also a linguistic
reason behind this terminology: If one translates from English to German the word
\lq\lq strong\rq\rq, then one gets the word \lq\lq stark\rq\rq\s and it is expected that these
strong $p$-units are related to the Gross-Stark $p$-units.

For any finite abelian extension $L$ of $F$ which contains $M$ and which is unramified outside 
$S$ we let
\begin{align*}
rec_{\mk{p}}^L:F_{\mk{p}}^{\times}\rightarrow\mathbb{A}_{F}^{\times}\rightarrow \Gal(L/F),
\end{align*} 
denote the reciprocity map given by local class field theory. From $M\subseteq M_{\mk{P}}\simeq F_{\mk{p}}$
we may evaluate $rec_{\mk{p}}^L$ on any element of $M^{\times}$; the image will be contained in $\Gal(L/M)$.

We can now give the precise statement of the strong Gross conjecture 
(see Conjectures 7.4 and 7.6 of \cite{Gr2}). Here the use of the adjective strong emphasizes the fact  
Conjecture \ref{Gross_strong} below is a strengthening of Conjecture 3.13 of \cite{Gr1}. 
\begin{Conj}{(Strong Gross Conjecture)}\label{Gross_strong}
Let $S$ and $T$ be two finite sets of places of $F$ which satisfy Assumptions \ref{ass_1} and
\ref{ass_inter} and let $R=S\bs\{\mk{p}\}$. 
For every prime $\mk{P}$ of $M$ above $\mk{p}$, there exists a unique strong $\mk{p}$-unit 
$u_T\in U_{\mk{p}}$ such that $u_T\equiv 1\pmod{T}$ and such that 
for all $\sigma\in \Gal(M/F)$ one has
\begin{align}\label{pierre1}
\zeta_{R,T}(M/F,\sigma,0)=\ord_{\mk{P}}(u_T^{\sigma}).
\end{align}
Moreover, for all finite abelian extension $L$ of $F$ which contains $M$ and which is unramified outside
$S$ one has
\begin{align}\label{pierre2}
rec_{\mk{p}}^L(u_T^{\sigma})=\prod_{\substack{\tau\in \Gal(L/F)\\ \tau|_{H}=\sigma^{-1}}}
\tau^{\zeta_{S,T}(L/F,\tau^{-1},0)}\in \Gal(L/M).
\end{align}
\end{Conj}
We remark that
\begin{align*}
\sum_{\substack{\tau\in \Gal(L/F)\\ \tau|_{H}=\sigma^{-1}}}\zeta_{S,T}(L/F,\tau,0)=\zeta_{S,T}(M/F,\sigma^{-1},0)=0,
\end{align*}
where the second equality follows from the presence of the Euler factor in \eqref{rabbit}.
So indeed the right hand side of \eqref{pierre2} lies in $\Gal(L/M)$. 

The condition $u_T\in U_{\mk{p}}$ together with \eqref{pierre1} specify the valuation of $u_T$ 
at all places of $F$ and therefore determine $u_T$ uniquely up to a root of
unity in $L$. Thanks to Assumption \ref{ass_inter} on the set $T$, we see that 
the additional condition $u_T\equiv 1\pmod{T}$ guarantees the uniqueness of $u_T$. 
We call $u_T$
the \textit{Gross-Stark unit} for the data $(S,T,M,\mk{P})$. 

Recently, in \cite{Das3}, Dasgupta has proposed an exact $\mk{p}$-adic formula for the Gross-Stark unit
$u_T$ which can be viewed as a refinement of the identity \eqref{pierre2} (See Appendix \ref{appendix1} 
where this is explained in greater details). 
One of the key ideas of his approach
is to replace the special values of partial zeta functions of $F$ 
by special values of \textit{Shintani Zeta functions}. 
In order to control the denominator
of these special values at $s=0$
he makes an additional assumption on the set $T$:
\begin{Assumption}\label{ass_2}
Assume that no prime of $S$ has the same residue characteristic as any prime of $T$
and that no two primes in $T$ have the same residue characteristic. 
\end{Assumption}

Dasgupta considers special values at $s=0$ of \lq\lq Shintani zeta functions\rq\rq\s 
which depend on the sets $S$, $T$ and on a choice of a fundamental domain for the action 
of $\ca{O}_K(\mk{f}\infty)^{\times}$ on the totally positive quadrant $\RR_{>0}^n$. Here $\ca{O}_K(\mk{f}\infty)^{\times}$
corresponds to the group of totally positive units of $\ca{O}_K$ which are congruent to one modulo
$\mk{f}$. For every ideal $\mk{a}$ of $F$ which is coprime to $S$ he uses these special values 
to construct a $\ZZ$-valued measure $\mu(\mathcal{D},\mk{a})$ 
on $\ca{O}_{F_\mk{p}}$. He then defines a $p$-adic invariant 
\begin{align}\label{turtle}
u_D(\mk{a},\mk{f})\in F_{\mk{p}}^{\times}
\end{align} 
as a certain multiplicative $p$-adic integral (defined as a limit of Riemann products)
on the space $O:=\ca{O}_{F_\mk{p}}\bs \mk{p}^e\ca{O}_{F_\mk{p}}$ of
the measure $\mu(\mathcal{D},\mk{a})$ against the identity function on $O$. 
Here $e\geq 1$ is the smallest integer such that $\mk{p}^e=(\pi)$
where $\pi$ is a totally positive element of $F$ such that $\pi\equiv 1\pmod{\mk{f}}$.
He then shows that under some technical conditions (which are always satisfied 
when $F$ is a real quadratic field!) that his $p$-adic invariant $u_{D}(\mk{a},\mk{f})$
does not depend on the choice of the Shintani domain $\mathcal{D}$ and that 
it only depends on $\sigma_\mk{a}$ rather than $\mk{a}$ itself. The previous fact is rather
remarkable since the measure $\mu(\ca{D},\mk{a})$ really depends on the Shintani 
domain and the ideal itself. So the process of integrating the measure $\mu(\ca{D},\mk{a})$ may be viewed as
a way of removing the dependence on the fundamental domain and relaxing the dependence on $\mk{a}$. 
Let us choose a prime $\mk{P}$ of $M$ above $\mk{p}$ so that $u_D(\mk{a},\mk{f})$ may be viewed as an element
of $M_{\mk{P}}\simeq F_{\mk{p}}$. Then Dasgupta conjectures
that $u_D(\mk{a},\mk{f})$ is equal to the Gross-Stark unit $u_T^{\sigma_{\mk{a}}}\in M_{\mk{P}}$ 
for the set of data $(S,T,M,\mk{P})$. We will not say more about his construction and we encourage the 
interested reader to look at \cite{Das3} for more details.

Let us go back to our original setting where $F=K=\QQ(\sqrt{d_K})$ is a real quadratic field.
Let $(p,f,N)$ be a triple as in the first paragraph of the introduction and let 
$\wp=p\ca{O}_K$ and $\mk{f}=f\ca{O}_K$. Choose a \textit{splitting 
of $N$}, i.e., a factorization of the form $N\ca{O}_K=\mk{N}\mk{N}^{\sigma}$ where $\mk{N}$ is 
an $\ca{O}_K$-ideal. Such a splitting exists because of the Heegner hypothesis. Let
\begin{enumerate}
\item[$(i)$] $S=\{\nu:\nu|\mk{f}\s\mbox{or}\s\nu|\infty\}\cup\{\wp\}$,
\item[$(ii)$] $T=\{\nu:\nu|\mk{N}\}$,
\end{enumerate}
where $\nu$ ranges over all places of $K$. We thus see that the triple $(p,f,\mk{N})$ 
encodes the same data as the triple $(\wp,S,T)$. Note that the set $S$ and $T$ satisfy the 
previous three assumptions.
Let $L=K(\mk{f}\infty)^{\laa\sigma_{\wp}\raa}$ and let $\mk{a}$ be an integral 
ideal coprime to $S$. Let us fix a prime $\mk{P}$ of $L$ above $\wp$. 
Then Dasgupta's construction
produces a $p$-adic invariant $u_{D}(\mk{a},\mk{f})$ which is conjectured to be equal to $u_T^{\sigma_{\mk{a}}}$ where
$u_T$ is the Gross-Stark unit for the set of data $(S,T,L,\mk{P})$. On the other hand, the author's construction
provides a $p$-adic invariant $u_C(r,\tau)$ which is also conjectured to be a strong $p$-unit in $L$.

We thus have two $p$-adic invariants $u_C$ and $u_D$ which are conjectured to be strong $p$-units
in $L$. Moreover, in each case, one has a conjectural \lq\lq analytic\rq\rq\; description of
how $\Gal(L/K)$ acts on $u_C$ and $u_D$. In each case, this analytic description 
will be referred as the Shimura reciprocity law.
(For the precise definitions of $u_C$ and $u_D$ and their respective Shimura reciprocity law, 
see Section \ref{def_invariants}). The main goal of this paper is to give precise relations 
between the $p$-adic invariants $u_C$ and $u_D$. Our main theorem is the following:
\begin{Th}\label{vag_th}
The invariant $u_C^{\varphi(f)}$ (relative to the order $\ca{O}_K$) 
belongs to the group generated by the $u_D$'s (relative to the order $\ca{O}_K$) where $\varphi$ 
corresponds to the Euler function. The invariant 
$u_D^{12}$ (relative to the order $\ca{O}_K$) belongs to the group generated by the $u_C$'s (relative to 
the order $\ca{O}_K$). 
Moreover, all the previous relations between the $u_D$'s and $u_C$'s 
are compatible with the Shimura reciprocity laws.
\end{Th}

The meaning of \lq\lq relative to the order $\ca{O}_K$\rq\rq\s is important and 
will be explained below (see Definition \ref{blow}). 
We will give in the text precise formulas that describe all these relations (see Theorem \ref{prec_th}
and Corollary \ref{fatigue}).
In order to prove Theorem \ref{vag_th} it is desirable to work in 
greater generality and extend the definition of Dasgupta's invariant $u_{D}(\mk{a},\mk{f})$ to invertible
$\ca{O}$-ideals $\mk{a}$ where $\ca{O}$ is allowed to be an \textit{arbitrary} order of $K$ (not just the maximal one)
of conductor coprime to $N$ and where $\mk{a}$ is no longer assumed to be coprime to $\mk{f}=f\ca{O}$. 
Let us explain the motivation behind this
level of generality since it introduces some technicalities. 
Let $\tau\in K\bs\QQ$ be such that $\ca{O}_{\tau}=\ca{O}_{N\tau}=\ca{O}$ where $\ca{O}$ is a fixed
\textit{ambient order} of conductor coprime to $N$. Note that we do not assume that
that $\cond(\ca{O})$ is corpime to $f$. In fact, in our applications 
it will be important to allow $\cond(\ca{O})$  to be divisible by $f$. To fix the idea, suppose that
$\cond(\ca{O})=n$ so that $\ca{O}=\ca{O}_n$, $(N,n)=1$.
Let $Q_{\tau}(x,y)=Ax^2+Bxy+Cy^2$ be the unique primitive quadratic form associated to $\tau$ 
for which $A=A_{\tau}>0$ and let $\mk{a}=A_{\tau}\Lambda_{\tau}$ be the associated integral invertible $\ca{O}$-ideal.
We note that $N|A$ and that $\disc(Q_{\tau})=B^2-4AC=n^2d_K$. The Atkin-Lehner involution $\tau\mapsto\tau^*=\frac{-1}{fN\tau}$ 
is the key ingredient which allows us to relate the invariant $u_C$ to the invariant $u_D$ (see identity \eqref{ville} below). 
Now a direct computation shows that
\begin{align*}
Q_{\tau^*}(x,y)=
\sign(C)\left(\frac{f^2}{d}CNx^2-B\frac{f}{d}xy+\frac{A}{dN}y^2\right),
\end{align*} 
where $0<d=(A,f)$. In particular, we have $\disc(Q_{\tau^*})=
\left(\frac{f}{d}\right)^2n^2d_K$. Thus, if $d\neq f$, the
Atkin Lehner involution changes the conductor of the order in the sense that 
$\cond(\ca{O}_{\tau^*})=\frac{f}{d}\cond(\ca{O}_{\tau})$. Moreover, in general,
the invertible integral $\ca{O}_{\tau^*}$-ideal 
$\mk{a}^*:=A_{\tau^*}\Lambda_{\tau^*}=C\frac{f^2}{d}N\Lambda_{\tau^*}$
is no longer coprime to $\mk{f}^{*}=f\ca{O}_{\tau^*}$ in the sense that $(\mk{f}^*,\mk{a}^*)
:=\mk{f}^{*}+\mk{a}^*\neq\ca{O}_{\tau^*}$. 
This explains in part the need of a theory which deals with
arbitrary orders (not just the maximal one) and arbitrary ideals (not just the ones coprime to $f$). 
Having extended the definition of $u_D$ to this broader setting it is useful to make the following
definition:
\begin{Def}\label{blow}
Let $\tau\in K\bs\QQ$ and let $\mk{b}=A_{\tau}\Lambda_{\tau}\subseteq K$ be the associated
normalized discrete $\ZZ$-module of rank $2$. Assume furthermore that 
$\End_K(\mk{b})=\ca{O}$ where $\cond(\ca{O},N)=1$. Let $\mk{g}\subseteq\ca{O}$ be 
an $\ca{O}$-ideal (not necessarily invertible). In this case we say that the $p$-adic invariant $u_D(\mk{b},\mk{g})$
(resp. $(u_C(r,\tau))$) is {\bf relative to the order $\ca{O}$}.
\end{Def}

We keep the same notation as the paragraph above Definition \ref{blow}. 
One of the key steps in the proof of Theorem \ref{vag_th} will be to show the following relation
(see Proposition \ref{klasse})
\begin{align}\label{ville}
u_{D}(\mk{a}^*,\mk{f}^*)^{12}=u_C(1,\tau)^e,
\end{align}
where $u_C(1,\tau)$ 
is the $p$-adic invariant associated to a suitable choice of a
family of modular units $\{\beta_r(z)\}_{r\in\ZZ/f\ZZ}$ and $e$ is a certain 
rational number that can be computed explicitly. We note that the construction of the
$p$-adic invariants $u_D$ and $u_C$ are different in nature and therefore the identity
\eqref{ville} is far from trivial. Roughly, the strategy that we 
use  to prove \eqref{ville} consists in relating the underlying $p$-adic measure
of $u_D$ (say $\mu_D$), to the underlying $p$-adic measure of $u_C$ (say $\mu_C$). Both of
$\mu_D$ and $\mu_C$ may be viewed as $\ZZ$-valued measures on the space $\ZZ_p^2$. 
In some suitable sense one can
show that $\mu_D$ is equal to $\mu_C$ (when evaluated on a ball of $\ZZ_p^2$) 
up to some error term which disappears when one integrates over the space $\XX$
(see Proposition \ref{tricky2} for a precise statement involving the error term and the measures
$\mu_D$ and $\mu_C$). However, we would like to emphasize 
here that the $p$-adic invariant
on the left hand side of \eqref{ville} is relative to the order $\ca{O}_{\tau^*}$ and the one on the
right hand side is relative to the order $\ca{O}_{\tau}$. As pointed out earlier, 
in general these two orders differ. It is precisely the discrepancy of these
two orders which makes Theorem \ref{vag_th} deeper than identity \eqref{ville}.

Let us explain in greater details why this discrepancy is 
problematic regarding Theorem \ref{vag_th}. In \cite{Das3},
Dasgupta considered $p$-adic invariants of the form $u_D(\mk{b},\mk{f})$
where $\mk{f}=f\ca{O}_K$ and $\mk{b}=A_{\tau}\Lambda_{\tau}$ where
$\tau\in K\bs\QQ$ is such that $End_K(\Lambda_{\tau})=\ca{O}_K$ and $(A_{\tau},f)=1$. 
Set $\rho=\tau^*$ where $\tau^*=\frac{-1}{fN\tau}$. Note that $\rho^*:=\frac{-1}{fN\rho}=\tau$. 
Using \eqref{ville}, one may deduce that
\begin{align}\label{botte}
u_D(A_{\rho^*}\Lambda_{\rho^*},f\ca{O}_{\rho^*})^{12}=u_D(\mk{b},f\ca{O}_K)^{12}=u_C(1,\rho)^{e},
\end{align}
where $e$ is a certain integer such that $e|\varphi(f)$. The first equality in \eqref{botte} 
follows from the definition of $\rho$ and the second equality follows from \eqref{ville}.
One immediately sees that \eqref{botte} gives a precise relation between Dasgupta's $p$-adic invariant $u_D$
(as considered in \cite{Das3})
and a special case of the $p$-adic invariant $u_C$ considered in \cite{Ch1}. 
However, the left hand side of \eqref{botte} is relative
to the order $\ca{O}_K$ and the one on the right hand side is relative to the order
$\ca{O}_f$. Therefore, identity \eqref{botte} does not quite fulfill the statement
of Theorem \ref{vag_th}. In order to prove Theorem \ref{vag_th}, one needs to work harder
and prove some \lq\lq distribution relations\rq\rq\s which are satisfied by the $p$-adic invariants
$u_D$'s \textit{relative to different orders}. Our Proposition \ref{pointe} gives an explicit example of such 
distribution relations. This proposition is the key tool which is used in the proof of Theorem \ref{vag_th}.

Finally, for the end of the introduction, we would like to introduce some new notions
that emerged naturally from our work. Let $\ca{O}$ be an arbitrary order of $K$ of conductor coprime to $N$.
Let $\mk{b}\subseteq K$ be an arbitrary
$\ca{O}$-invertible ideal, i.e., $\End_K(\mk{b})=\End_K(\mk{g})=\ca{O}$ and let
$\mk{g}\subseteq\ca{O}$ be an arbitrary $\ca{O}$-module which we assume to be 
invertible only for the sake of simplicity.
We say that $\mk{b}$ is \textit{$\mk{g}$-int} if $\mk{b}$ may be written as
$\mk{b}=\mk{c}\mk{d}^{-1}$ where $\mk{c},\mk{d}\subseteq\ca{O}$ are invertible
$\ca{O}$-ideals and where $(\mk{d},\mk{g})=\mk{d}+\mk{g}=\ca{O}$. Note in particular that if 
$\mk{b}\subseteq\ca{O}$ then $\mk{b}$ is automatically $\mk{g}$-int. 
A pair $[\mk{b},\mk{g}]$ is said to be 
\textit{primitive} of conductor $[\ca{O},\lambda\mk{g}]$ 
if there exists $\lambda\in K^{\times}$ such that $\lambda\mk{b},\lambda\mk{g}\subseteq\ca{O}$ 
and such that $(\lambda\mk{b},\lambda\mk{g})=\ca{O}$. We have put brackets 
$[\mk{b},\mk{g}]$ in order to distinguish the 
\lq\lq ordered pair $(\mk{b},\mk{g})$\rq\rq\s from the standard notation $(\mk{b},\mk{g})=\mk{b}+\mk{g}$. 
From now on we drop the $\ca{O}$ in the notation $[\ca{O},\lambda\mk{g}]$ since it was assumed 
from the outset that $\ca{O}$ was a fixed \textit{ambient} order. 
Note that even though
$\lambda$ is not unique, the $\ZZ$-lattice $\lambda\mk{g}$ is well defined. 
If no such $\lambda$ exists then we say that
the pair $[\mk{b},\mk{g}]$ is \textit{non-primitive}. For example, 
let $\mk{a}=A_{\tau}\Lambda_{\tau}$ where $\ca{O}_{\tau}=\ca{O}_K$ and
let $\mk{f}=f\ca{O}_K$. Furthermore, assume that $(\mk{a},\mk{f})=\ca{O}_K$ which is 
equivalent to $(A_{\tau},f)=1$. Then one may check that
$[\mk{a},\mk{f}]$ is primitive of conductor $\mk{f}$. Now consider the pair
$[\mk{a}^*,\mk{f}^*]$ where $\tau^*=\frac{-1}{fN\tau}$, $\mk{a}^*=A_{\tau^*}\Lambda_{\tau^*}$ and
$\mk{f}^*=f\ca{O}_{\tau^*}$. Then one may check that $[\mk{a}^*,\mk{f}^*]$ is a non-primitive pair if
$f>1$. 

Let us explain how this notion of primitive pairs intervenes in the context of the
$p$-adic invariant $u_D$ that was introduced earlier.
Let $[\mk{b},\mk{g}]$ be a pair as in the paragraph above of conductor $\mk{g}$. Then Conjecture
\ref{alg_conj} predicts that $u_D(\mk{b},\mk{g})$ is a strong $p$-unit in $K(\mk{g}\infty)^{\laa\sigma_{\wp}\raa}$.
The author expects, that for such a fixed primitive pair $[\mk{b},\mk{g}]$
and for a \lq\lq generic divisor $\wt{\delta}$\rq\rq\s (see Section \ref{div_tilde} for the definition of
$\wt{\delta}$) that the $p$-adic invariant $u_{D,\wt{\delta}}(\mk{b},\mk{g})$ 
generates the maximal CM subfield of $K(\mk{g}\infty)^{\laa\sigma_{\wp}\raa}$. 
Thus we expect that (for a generic divisor $\wt{\delta}$) 
$u_{D,\wt{\delta}}(\mk{a},\mk{f})$ generates the
maximal CM subfield of $K(\mk{f}\infty)^{\laa\sigma_{\wp}\raa}$, say $L_{CM}$. It is important here to use the word
generic since explicit computations done in 
\cite{Das2} and \cite{Ch3} reveal that it may happen that for some divisor $\wt{\delta}$,
$u_{D,\wt{\delta}}(\mk{a},\mk{f})$ lies in a proper subfield of
$L_{CM}$. On the other hand, the author expects things to be different for a non-primitive pair.
For example, we may consider the non-primitive pair $[\mk{a}^*,\mk{f}^*]$.
A priori, Conjecture \ref{alg_conj}, only predicts that
for a divisor $\wt{\delta}$ one has that $u_{D,\delta}(\mk{a}^*,\mk{f}^*)
\in K(\mk{f}^*\infty)^{\laa\sigma_{\wp}\raa}$. Note that 
$K(\mk{f}\infty)\subsetneqq K(\mk{f}^*\infty)$. However, a more careful analysis of Conjecture \ref{alg_conj} (see
Section \ref{down}), implies that 
$u_{D,\wt{\delta}}(\mk{a}^*,\mk{f}^*)\in K(\mk{f}\infty)^{\laa\sigma_{\wp}\raa}$. This strange phenomenon can be accounted
by the observation that the pair $[\mk{a}^*,\mk{f}^*]$ is non-primitive! The interested reader may also find many numerical 
evidence in \cite{Ch3} supporting the conjectural relation 
$u_{D,\wt{\delta}}(\mk{a}^*,\mk{f}^*)^{12}=u_C(1,\tau)\in K(\mk{f}\infty)^{\laa\sigma_{\wp}\raa}$.

One of the main goal of this paper is to extend Dasgupta's construction in the following broader setting:
to an arbitrary order $\ca{O}\subseteq K$ (not necessarily the maximal one) of conductor coprime to $N$ 
and to an arbitrary pair $[\mk{b},\mk{g}]$ (not necessarily primitive) where
$\mk{g}\subset\ca{O}$ is an $\ca{O}$-ideal and where $\mk{b}$ 
is $\ca{O}$-invertible and $\mk{g}$-int, we construct
a $p$-adic invariant $u_D(\mk{b},\mk{g})\in K_{p}^{\times}$.
\section{Notation and some basic notions about orders}\label{notations}
Let $K$ be a real quadratic number field and let $\ca{O}_K=\ZZ+\omega\ZZ$ be its maximal $\ZZ$-order. 
An arbitrary $\ZZ$-order of $K$ will be denoted by the letter $\ca{O}$. For every positive integer
$n\geq 1$ there exists a unique order of $K$ with \textit{conductor $n$} which we denote by $\ca{O}_n=\ZZ+n\omega\ZZ$. 
Let $\ca{O}$ be a fixed $\ZZ$-order of $K$.
A discrete $\ca{O}$-module $\mk{a}\subseteq K$ will be called an $\ca{O}$-ideal. 
If $\mk{a},\mk{b}\subseteq K$ are $\ZZ$-lattices (always assumed of rank $2$) we denote by 
$(\mk{a},\mk{b}):=\mk{a}+\mk{b}$. We remark here that if $\mk{a},\mk{b}\subseteq K$ are 
$\ca{O}$-ideals then $(\mk{a},\mk{b})$ is the smallest $\ca{O}$-ideal of $K$ which contains
$\mk{a}$ and $\mk{b}$ (note however that the $\ZZ$-module $(\mk{a},\mk{b})$ is completely independent of
the ambient order $\ca{O}$!).
By an \textit{invertible} $\ca{O}$-ideal (or an $\ca{O}$-invertible ideal) we mean an 
$\ca{O}$-ideal $\mk{a}$ such that 
\begin{align*}
\End_K(\mk{a}):=\{\lambda\in K:\lambda\mk{a}\subseteq\mk{a}\}=\ca{O}.
\end{align*}
Note that every $\ZZ$-lattice $\Lambda\subseteq K$ is an invertible $\ca{O}$-ideal
for $\ca{O}=\End_K(\Lambda)$. For an arbitrary $\ZZ$-lattice 
$\Lambda\subseteq K$ we define
\begin{align}\label{tot}
\Lambda^{-1}:=\{\lambda\in K:\lambda\Lambda\subseteq\End_K(\Lambda)\}. 
\end{align}
A $\ZZ$-lattice $\Lambda\subseteq K$ will be called 
\textit{integral} if $\Lambda\subseteq \End_K(\Lambda)$. A $\ZZ$-lattice
$\Lambda$ is said to be \textit{$\ca{O}$-integral} if $\Lambda$ is an $\ca{O}$-ideal such that 
$\Lambda\subseteq\ca{O}$. We note that the notion of integrality is absolute
since it does not depend on the choice of an ambient order while the notion
of $\ca{O}$-integrality is relative since it depends on the choice of an ambient order $\ca{O}$.

We recall some facts about invertible $\ca{O}$-ideals. 
It is well known that an $\ca{O}$-ideal $\mk{a}$ is invertible if and only if for every 
prime ideal $\mk{p}\neq 0$ of $\ca{O}$ one has $\mk{a}_{\mk{p}}$ is a principal
$\ca{O}_{\mk{p}}$-ideal (here $\mk{a}_{\mk{p}}$ denotes the localization $\mk{a}_{\mk{p}}$ 
of $\mk{a}$ at $\mk{p}$). 
For a proof see Proposition 12.4 of \cite{Neu}).
From this we may deduce that a prime $\mk{p}$ of $\ca{O}$ is invertible if and only if 
$\ca{O}_{\mk{p}}$ is a discrete valuation ring. We have
the following criterion for invertible prime ideals of $\ca{O}$:
$$
\mk{p}\nmid \cond(\ca{O})\Longleftrightarrow\mk{p}\s\mbox{is invertible},
$$
where $\cond(\ca{O})$ denotes the conductor of $\ca{O}$. 
For a proof of this fact see Proposition 12.10 of \cite{Neu}. 

If $\Lambda$ and $\Lambda'$ are two lattices in $K$ then one always has that
\begin{align}\label{tramp}
\End_K(\Lambda)\cap \End_K(\Lambda')\subseteq \End_K(\Lambda\cap\Lambda').
\end{align}
In general the inclusion in \eqref{tramp} could be strict. 
If $\mk{a}$ and $\mk{b}$ are $\ca{O}$-ideals the two lattices 
$\mk{a}\cap\mk{b}$ and $(\mk{a},\mk{b})$ are again $\ca{O}$-ideals and one always has that
$$
\ca{O}\subseteq \End_K(\mk{a}\cap\mk{b})\s\s\s\mbox{and}\s\s\s 
\ca{O}\subseteq \End_K((\mk{a},\mk{b})).
$$ 
In general, the two inclusions above
could be strict, even when restricted to invertible $\ca{O}$-ideals. 
However, in the special case where $\mk{a}$ and $\mk{b}$ are invertible $\ca{O}$-ideals 
supported only on invertible prime ideals of $\ca{O}$, then the two inclusions above become equalities.

Let $\mk{a}$ and $\mk{b}$ be $\ca{O}$-integral ideals. Then we have the following short exact
sequence
\begin{align*}
0\rightarrow\mk{a}/\mk{a}\mk{b}\rightarrow\ca{O}/\mk{a}\mk{b}\rightarrow\ca{O}/\mk{a}\rightarrow 0.
\end{align*}
From this we may deduce that
\begin{align}\label{tigre}
[\ca{O}:\mk{a}\mk{b}]=[\mk{a}:\mk{a}\mk{b}][\ca{O}:\mk{a}].
\end{align}
Under the additional assumption that $\mk{a}$ is an invertible $\ca{O}$-ideal, an easy localization argument
shows that $\mk{a}/\mk{a}\mk{b}$ is isomorphic (non-canonically!) to $\ca{O}/\mk{b}$ as an abelian group. 
Combining the previous observation with \eqref{tigre} we may conclude that if either $\mk{a}$ or $\mk{b}$
is $\ca{O}$-invertible then
\begin{align}\label{tigre2}
[\ca{O}:\mk{a}\mk{b}]=[\ca{O}:\mk{a}][\ca{O}:\mk{b}].
\end{align}

We have the following diagram
$$
\xymatrix{
Spec(\ca{O}_K)\ar[rr]^{\pi}\ar[rd]_{p_1} & & Spec(\ca{O}_n)\ar[ld]^{p_2}\\
& Spec(\ZZ) &
}
$$
Since $\ca{O}_K[\frac{1}{n}]=\ca{O}_n[\frac{1}{n}]$ we see that 
the map $\pi$, when restricted to $p_1^{-1}(Spec(\ZZ[\frac{1}{n}]))$, is an isomorphism. 
Let $l$ be a prime divisor of $n$ and let $\mk{I}_l:=l\ZZ+n\omega\ZZ$. Note that
$\mk{I}_l$ is a prime ideal of $\ca{O}_n$ such that $\ca{O}_n/\mk{I}_{l}\simeq\ZZ/l\ZZ$.
The fiber above $l\ZZ$ (always for $l|n$) is given by one of the following two possibilities:
\begin{enumerate}
\item if $l$ is inert or ramified in $K$: $p_2^{-1}(l\ZZ)=\{\mk{I}_l\}$,  
\item if $l$ splits in $K$: 
$p_2^{-1}(l\ZZ)=\{\mk{I}_l,\eta',\eta'^{\sigma}\}$ where 
$l\ca{O}_K=\eta\eta^{\sigma}$ and $\eta'=\eta\cap\ca{O}_n$,
\end{enumerate} 
In general, for a fixed integral $\ca{O}_K$-ideal $\mk{N}$ and an arbitrary order $\ca{O}$ 
(we think of $\ca{O}$ here as varying), one has a natural inclusion
\begin{align}\label{ineq}
\ca{O}/(\mk{N}\cap\ca{O})\hookrightarrow\ca{O}_K/\mk{N},
\end{align}
which may fail to be onto. However, there is a special type of 
$\ca{O}_K$-ideal $\mk{N}$ for which \eqref{ineq} is an isomorphism for 
\textit{all} orders $\ca{O}$ of $K$, namely the case where $\ca{O}_K/\mk{N}\simeq\ZZ/N\ZZ$ 
is a cyclic abelian group. We leave the proof of this elementary fact to the reader. 
In particular, let $\{l_i\}_{i=1}^r$ be a set of prime numbers that split in 
$K$ and let $l_i\ca{O}_K=\eta_i\eta_i^{\sigma}$. Then if one sets $\mk{N}=\prod_{i=1}^r\eta_i^{e_i}$, for some
arbitrary positive integers $e_i$,
one has that $\ca{O}_K/\mk{N}\simeq\ZZ/N\ZZ$ where $N=\prod_{i=1}^r l_i^{e_i}$. Therefore the map in 
\eqref{ineq} is onto. Always in this special case, 
one may furthermore check that for all orders $\ca{O}$, one has that $\End_K(\ca{O}\cap\mk{N})=\ca{O}$, so that
\eqref{tramp} becomes an equality with $\Lambda=\ca{O}$ and $\Lambda'=\mk{N}$.

We let $\mathcal{L}$ be the set of all free $\ZZ$-modules contained in $K$. The
set $\mathcal{L}$ comes equipped with a natural stratification
\begin{align}\label{trianglee}
\mathcal{L}=\coprod_{\ca{O}}\mathcal{L}^\mathcal{\ca{O}},
\end{align}
where $\mathcal{L}^{\ca{O}}=\{L\in\mathcal{L}:\End_K(L)=\mathcal{O}\}$ and 
the disjoint union is taken over all orders of $K$.

We now introduce equivalence relations on $\mathcal{L}$ and on a stratum $\mathcal{L}^{\ca{O}}$. These equivalence
relations may be viewed as natural generalizations of the usual equivalence relation on the set of ideals of
a Dedekind domain which gives rise to the ideal class group. Let $\lambda\in K^{\times}$. The notation $\lambda\gg 0$
is taken to mean that $\lambda$ is a totally positive element, i.e., $\lambda,\lambda^{\sigma}>0$.
\begin{Def}
Let $f\in\ZZ_{>0}$ be a fixed integer and let
$L_1,L_2\in\mathcal{L}$. We say that $L_1\sim_{f} L_2$ if and only if
there exists $\lambda\in L_1^{-1}f+1$, $\lambda\gg 0$, such that $\lambda L_1=L_2$. 
(see \eqref{tot} for the definition of $L_1^{-1}$). Let $\ca{O}$ be a fixed order
and let $\mk{f}$ be an $\ca{O}$-integral ideal. For $L_1,L_2\in\mathcal{L}^{\ca{O}}$ we say that
$L_1\sim_{\mk{f}} L_2$ if and only if
there exists $\lambda\in L_1^{-1}\mk{f}+1$, $\lambda\gg 0$, such that $\lambda L_1=L_2$.
\end{Def}
It is easy to see that $\sim_f$ is reflexive, transitive and symmetric. Therefore 
$\sim_{f}$ gives rise to an equivalence relation on $\mathcal{L}$. A similar statement holds
for $\sim_{\mk{f}}$ if we replace $\mathcal{L}$ by $\mathcal{L}^{\ca{O}}$. We note that
$\sim_f$ preserves the stratification given by \eqref{trianglee} and that if 
$L_1,L_2\in\mathcal{L}^{\ca{O}}$ and $L_1\sim_f L_2$
then $L_1$ is $\ca{O}$-integral if and only if $L_2$ is $\ca{O}$-integral. 

Let $L_1,L_2,\mk{a}\in\mathcal{L}^{\ca{O}}$ 
and assume that $L_1\sim_f L_2$. Then if $\mk{a}$ is $\ca{O}$-integral one can show that
\begin{align}\label{prise}
L_1\mk{a}\sim_f L_2\mk{a}. 
\end{align}
Now suppose that $L_1,L_2,\mk{a},\mk{b}\in\mathcal{L}^{\ca{O}}$, $L_1\sim_f L_2$
and $\mk{a}\sim_f\mk{b}$. If $L_1$ and $\mk{b}$ are $\ca{O}$-integral then applying \eqref{prise}
twice we find that $\mk{a}L_1\sim_f\mk{b}L_1$ and $\mk{b}L_1\sim_f\mk{b}L_2$. Thus by transitivity of
$\sim_f$ we find that
\begin{align}\label{prise2}
\mk{a}L_1\sim_f\mk{b}L_2.
\end{align}
Note that in general, if $L_1,L_2\in\mathcal{L}^{\ca{O}}$ and $L_1\sim_f L_2$, it is not 
necessarily true that $L_1^{-1}\sim_f L_2^{-1}$ (unless $f=1$).

Let $\ca{O}$ be a fixed order. Let $\mk{f}$ be an $\ca{O}$-integral ideal which is not necessarily 
$\ca{O}$-invertible. We define the set
$$
I_{\ca{O}}(\mk{f}):=
\{\mk{b}\subseteq\ca{O}:\textrm{$\mk{b}$ is an invertible integral $\ca{O}$-ideal
coprime to $\mk{f}$, i.e., $(\mk{f},\mk{b})=\ca{O}$}\}.
$$
For $m\in\ZZ_{\geq 1}$ we also let $I_{\ca{O}}(m)=I_{\ca{O}}(m\ca{O})$.
Consider the monoid $I_{\ca{O}}(1)$ where the multiplication is given by the
usual multiplication of ideals. For every $\ca{O}$-integral ideal $\mk{f}$ we
have the  equivalence relation $\sim_{\mk{f}}$ on $I_{\ca{O}}(1)$. 
Note that if $\mk{a}\sim_{\mk{f}}\mk{b}$ then $(\mk{a},\mk{f})=(\mk{b},\mk{f})$. In fact,
if for every $\ca{O}$-integral ideal $\mk{d}\supseteq\mk{f}$ we let
$$
I_{\ca{O}}[\mk{d},\mk{f}]:=
\{\mk{b}\subseteq\ca{O}:\textrm{$\mk{b}$ is an invertible integral $\ca{O}$-ideal
such that $(\mk{f},\mk{b})=\mk{d}$}\},
$$
then the set $I_{\ca{O}}(1)/\sim_{\mk{f}}$ admits the following stratification:
\begin{align*}
I_{\ca{O}}(1)/\sim_{\mk{f}}\;=\;\bigcup_{\mk{d}|\mk{f}}\;I_{\ca{O}}[\mk{d},\mk{f}]/\sim_{\mk{f}}.
\end{align*} 
Using \eqref{prise2}, we see that the monoid structure on $I_{\ca{O}}(1)$ descends
to a (finite) monoid structure on $I_{\ca{O}}(1)/\sim_{\mk{f}}$.
The set of invertible elements of $I_{\ca{O}}(1)/\sim_{\mk{f}}$ is exactly
$I_{\ca{O}}(\mk{f})/\sim_{\mk{f}}$.
\begin{Def}\label{affine}
We set $C_{\ca{O}}(\mk{f}):=I_{\ca{O}}(\mk{f})/\sim_{\mk{f}}$. 
We call $C_{\ca{O}}(\mk{f})$ the \textit{narrow extended ideal class group of $K$ of conductor $[\ca{O},\mk{f}]$}.
In the special case where $\mk{f}$ is $\ca{O}$-invertible we simply say that $C_{\ca{O}}(\mk{f})$ is 
the \textit{narrow extended ideal class group of $K$ of conductor $\mk{f}$}. 
Note that this makes sense since the order $\ca{O}$ is already encoded in $\mk{f}$ 
($\End_K(\mk{f})=\ca{O}$). By class field theory, the ideal class group $C_{\ca{O}}(\mk{f})$ corresponds to an abelian extension 
of $K$ which we denote by $K([\ca{O},\mk{f}]\infty)$ where $\infty=\infty_1\infty_2$ stands for the
product of the two distinct real places of $K$. We call $K([\ca{O},\mk{f}]\infty)$ the 
\textit{narrow extended class field} of conductor $[\ca{O},\mk{f}]$. 
In the case where $\mk{f}$ is $\ca{O}$-invertible we simply
write  $K(\mk{f}\infty)$ rather than the more cumbersome notation $K([\ca{O},\mk{f}]\infty)$. 
\end{Def}
We note that $K([\ca{O},\mk{f}]\infty)$ is an extension of $K$ unramified outside
the finite places of $K$ not dividing $[\ca{O}:\mk{f}]\cdot\cond(\ca{O})$. On the
other hand, the ramification above $\infty_1$ and $\infty_2$ is a more delicate question and
depends on the sign and congruences of a fundamental unit of $K$. 

We define
\begin{enumerate}
\item $P_{\ca{O},1}(\mk{f})=\left\{\frac{\alpha}{\beta}\ca{O}:\alpha,\beta\in\ca{O},
(\alpha,\mk{f})=(\beta,\mk{f})=\ca{O},
\alpha\equiv\beta\pmod{\mk{f}}\right\}$,
\item $P_{\ca{O},1}(\mk{f}\infty)=\left\{\frac{\alpha}{\beta}\ca{O}:\alpha,\beta\in\ca{O},
(\alpha,\mk{f})=(\beta,\mk{f})=\ca{O},
\alpha\equiv\beta\pmod{\mk{f}},\frac{\alpha}{\beta}\gg 0\right\}$.
\item $Q_{\ca{O},1}(\mk{f}):=\left\{\frac{\alpha}{\beta}\in K:\alpha,\beta\in\ca{O},
(\alpha,\mk{f})=(\beta,\mk{f})=\ca{O},
\alpha\equiv\beta\pmod{\mk{f}}\right\}$.
\item $Q_{\ca{O},1}(\mk{f}\infty):=\left\{\frac{\alpha}{\beta}\in K:\alpha,\beta\in\ca{O},
(\alpha,\mk{f})=(\beta,\mk{f})=\ca{O},
\alpha\equiv\beta\pmod{\mk{f}},\frac{\alpha}{\beta}\gg 0\right\}$.
\end{enumerate}
We have a natural map $Q_{\ca{O},1}(\mk{f})\rightarrow P_{\ca{O},1}(\mk{f})$ (resp. 
$Q_{\ca{O},1}(\mk{f}\infty)\rightarrow P_{\ca{O},1}(\mk{f}\infty)$) which is
given by $\lambda\mapsto\lambda\ca{O}$. One may prove that for $\mk{a},\mk{b}\in I_{\ca{O}}(\mk{f})$, 
$\mk{a}\sim_{\mk{f}}\mk{b}$ 
if and only if there exists a $\lambda\ca{O}\in P_{\ca{O},1}(\mk{f}\infty)$ such that $\lambda\mk{a}=\mk{b}$. 
Keeping in mind that $P_{\ca{O},1}(\mk{f}\infty)\not\subseteq I_{\ca{O}}(\mk{f})$
we may nevertheless identify $C_{\ca{O}}(\mk{f})$ 
with the \lq\lq quotient\rq\rq\s$I_{\ca{O}}(\mk{f})/P_{\ca{O},1}(\mk{f}\infty)$ in the sense just explained
above.

We will need the following elementary proposition:
\begin{Prop}\label{lilas}
Let $n,f\in\ZZ_{\geq 1}$ and let $\ca{O}$ and $\ca{O}'$ be orders such that
$\cond(\ca{O})=m$ and $\cond(\ca{O}')=nm$ (in particular $\ca{O}'\subseteq\ca{O}$). Set $\mk{f}=f\ca{O}$ and 
$\mk{f}'=f\ca{O}'$. 

$\mbox{(i)}$ Let $\mk{a}\in I_{\ca{O}}(n\mk{f})$ then $\mk{a}\cap\ca{O}'$ is $\ca{O}'$-invertible. Moreover,
the natural map
\begin{align}\label{papillon1}
\Theta:I_{\ca{O}}(n\mk{f})&\rightarrow I_{\ca{O}'}(\mk{f}')\\ \notag
\mk{a}&\mapsto\Theta(\mk{a})=\mk{a}\cap\ca{O}'
\end{align}
is multiplicative, i.e., for all
$\mk{a},\mk{a}'\in I_{\ca{O}}(n\mk{f})$ one has $\Theta(\mk{a}\mk{a}')=\Theta(\mk{a})\Theta(\mk{a}')$. 

$\mbox{(ii)}$ If $\mk{a},\mk{a}'\in I_{\ca{O}}(n\mk{f})$ are such that
$\mk{a}\sim_{nf}\mk{a}'$ then one has $\Theta(\mk{a})\sim_{f}\Theta(\mk{a}')$ and therefore
$\Theta$ induces a natural onto map
\begin{align}\label{papillon2}
\wt{\Theta}:C_{\ca{O}}(n\mk{f})\rightarrow C_{\ca{O}'}(\mk{f}').
\end{align}
\end{Prop}

\begin{proof} 
See Appendix \ref{appendix2}.
\end{proof}

We note that the map $\wt{\Theta}$ in the above proposition corresponds to the restriction map 
\begin{align*}
\Gal(K(\mk{f}n\infty)/K)\stackrel{res}{\rightarrow} \Gal(K(\mk{f}'\infty)/K).
\end{align*}
\section{An equivalence relation on pair of lattices with cyclic quotient}
For every $\tau\in K\bs\QQ$ we define 
$$
Q_{\tau}(x,y):=A(x-\tau y)(x-\tau^{\sigma}y)=
Ax^2+Bxy+Cy^2,
$$ 
with $A>0$, $A,B,C\in\ZZ$ and $(A,B,C)=1$ to be the unique \textit{primitive quadratic form} associated to $\tau$.
Usually, we will denote the coefficient of $x^2$ of $Q_{\tau}(x,y)$ by $A_{\tau}=A$. We note here 
that if $a$ is positive integer such that $a\Lambda_{\tau}\subseteq\ca{O}_{\tau}$ then $A_{\tau}|a$.
Recall that for every $\ZZ$-lattice $L\subseteq K$ we define 
$\End_K(L)=\{\lambda\in K:\lambda L\subseteq L\}$ and
when $\tau\in K\bs\QQ$ we denote $\End_K(\Lambda_{\tau})$ simply by $\ca{O}_{\tau}$. 
For an element $\tau\in K\bs\QQ$ one can check 
that $\ca{O}_{\tau}=\ZZ+A_{\tau}\tau\ZZ$. 
For every $\ZZ$-lattice $L\subseteq K$ we define $\Norm(L)$ to be the absolute value
of the determinant of a matrix with rational coefficients that takes a $\ZZ$-basis of $\ca{O}_K$ to a $\ZZ$-basis
of $L$. Note that for an element $a\in K^{\times}$ and an arbitrary order $\ca{O}$ one has
the formula $\Norm(a\ca{O})=|\Norm_{K/\QQ}(a)|\Norm(\ca{O})$. 
One may also check that 
\begin{align}\label{titi}
\Norm(\Lambda_{\tau})=\frac{\cond(\ca{O}_{\tau})}{A_{\tau}}\s\s\s\mbox{and}\s\s\s 
\Lambda_{\tau}\Lambda_{\tau^{\sigma}}=\frac{1}{A_{\tau}}\ca{O}_{\tau}.
\end{align}
In particular, if $\mk{a}$ is an invertible integral $\ca{O}$-ideal we obtain from the first equality of
\eqref{titi} that $\Norm(\mk{a})=\cond(\ca{O})[\ca{O}:\mk{a}]$.

Let $\mathcal{L}$ be the set of all lattices of $K$.
It is convenient to define a map that takes an element of $K\bs\QQ$ to an \textit{integral lattice} of $K$.
\begin{Def}\label{ini_def}
We set
\begin{align*}
I_{\_}:K\bs\QQ &\rightarrow \mathcal{L}\\
           \tau &\mapsto I_{\tau}:=A_{\tau}\Lambda_{\tau}, 
\end{align*}
where $A_{\tau}$ is the $x^2$-coefficient of $Q_{\tau}(x,y)$.
\end{Def}
We record the following  two useful formulas 
\begin{align}\label{blum}
\mbox{(i)}\s\s\s I_{\tau}^{-1}=\Lambda_{\tau^{\sigma}}\s\s\s\s\mbox{and}
\s\s\s\s\mbox{(ii)}\s\s\s\tau^{\sigma}I_{\tau}=I_{\frac{1}{\tau}}.
\end{align}

From now on we take the following convention:
\begin{Def}
Let $\ca{O}\subseteq\ca{O}'$ be two orders and let
$\mk{a}\in\mathcal{L}^{\ca{O}'}$ and $L\in\mathcal{L}^{\ca{O}}$. Assume furthermore 
that $\mk{a}$ is $\ca{O}'$-integral. Then the product $\mk{a}L$ is
taken to mean $(\mk{a}\cap\ca{O})L$ where the last product is the product as two
$\ca{O}$-ideals. 
\end{Def}
Let $N=\prod_{i=1}^r l_i^{e_i}$ where $l_i$'s are distinct prime numbers which split in $K$ and where 
the $e_i$'s are arbitrary positive integers. 
Choose a splitting of $\mk{N}$, i.e., an integral $\ca{O}_K$-ideal such that $\mk{N}\mk{N}^{\sigma}=N\ca{O}_K$.
Note that for any lattice $L\in\mathcal{L}^{\ca{O}}$ one has that
$\mk{N}L$ is again $\ca{O}$-invertible and that $L/\mk{N}L\simeq\ZZ/N\ZZ$.

We set $H(\mk{N})=\{\tau\in K:\mk{N}\Lambda_{\tau}=\Lambda_{N\tau},\tau-\tau^{\sigma}>0\}$,
and $H^{\ca{O}}(\mk{N})=\{\tau\in H(\mk{N}):\ca{O}_{\tau}=\ca{O}\}$. We note that
\begin{align}\label{thrill}
\mk{N}\Lambda_{\tau}=\Lambda_{N\tau}
\Longleftrightarrow
N\Lambda_{\tau}=\mk{N}^{\sigma}\Lambda_{N\tau}.
\end{align}

We have a natural stratification
\begin{align}\label{stratif}
H(\mk{N})=\bigcup_{\ca{O}}H^{\ca{O}}(\mk{N}),
\end{align}
where the union is disjoint and taken over all orders of $K$.
We note that if $\tau\in H^{\ca{O}}(\mk{N})$ then necessarily $N|A_{\tau}$. In particular, 
one may deduce from this that
\begin{align}\label{rouge}
\tau\in H^{\ca{O}}(\mk{N})\Longleftrightarrow \frac{-1}{N\tau}\in H^{\ca{O}}(\mk{N}).
\end{align}
We also define $\mathcal{L}(\mk{N})=\{[L,\mk{N}L]\in\mathcal{L}^2\}$
and $\mathcal{L}^{\ca{O}}(\mk{N})=\{[L,\mk{N}L]\in\mathcal{L}(\mk{N}):\End_K(L)=\ca{O}\}$,
where $\ca{O}$ is an arbitrary order of $K$. Similarly to \eqref{stratif} we have a natural stratification
\begin{align}\label{strat1}
\mathcal{L}(\mk{N})=\bigcup_{\ca{O}}\mathcal{L}^{\ca{O}}(\mk{N}).
\end{align}
\begin{Def}\label{tion}
Let $\ca{O}$ be a fixed order and let $\mk{g}\subseteq\ca{O}$ be a fixed $\ca{O}$-ideal. Let
$[\mk{a},\mk{b}],[\mk{a}',\mk{b}']\in\mathcal{L}^{\ca{O}}(\mk{N})$. We say that 
$[\mk{a},\mk{b}]\sim_{\mk{g}}[\mk{a}',\mk{b}']$ if and only if 
there exists $\lambda\in(\mk{a}^{-1}\mk{g}+1)$, $\lambda\gg 0$, such that 
$[\lambda\mk{a},\lambda\mk{b}]=[\mk{a}',\mk{b}']$.
\end{Def}
One may check that $\sim_{\mk{g}}$ gives rise to an equivalence relation on the
set $\mathcal{L}^{\ca{O}}(\mk{N})$. 

Now let $f\in\ZZ_{\geq 1}$ be an integer coprime to $Nd_K$. We now focus our attention on  
the two sets  $\ZZ/f\ZZ\times H(\mk{N})$ and $\mathcal{L}(\mk{N})$. On each of these
two sets we will define an equivalence relation and it will be shown that their respective quotients 
are in natural bijection. Consider the map
\begin{align*}
\psi:\ZZ/f\ZZ\times H(\mk{N})&\rightarrow\mathcal{L}(\mk{N})\\
(r,\tau)&\mapsto\left[\wt{r}I_{\tau},\wt{r}\mk{N}I_{\tau}\right]
\end{align*}
where $\sim\;:\ZZ/f\ZZ\rightarrow\{1,\ldots,f\}$ is the unique map such that for
any $r\in\ZZ/f\ZZ$ one has $\wt{r}\equiv r\pmod{f}$.
\begin{Def}
Let $(r,\tau),(r',\tau')\in\ZZ/f\ZZ\times H(\mk{N})$. We say that $(r,\tau)\sim_f (r',\tau')$
if and only if there exists a matrix $\gamma=\M{a}{b}{c}{d}\in\Gamma_0(fN)$ such that
$d^{-1}r\equiv r'\pmod{f}$ and $\tau'=\gamma\tau$. Let $[L,M],[L',M']\in
\mathcal{L}(\mk{N})$. We say that $[L,M]\sim_f [L',M']$ if and
only if there exists $\lambda\in(L^{-1}f+1)$, $\lambda\gg 0$, such that $[\lambda L,\lambda M]=[L',M']$. 
\end{Def}

We denote the class $(r,\tau)$ modulo $\sim_f$ by $[(r,\tau)]$ and the class
$[L,M]$ modulo $\sim_f$ by $[[L,M]]$.
\begin{Prop}\label{hard_prop}
The map
\begin{align*}
\wt{\psi}:(\ZZ/f\ZZ\times H(\mk{N}))/\sim_f &\rightarrow\mathcal{L}(\mk{N})/\sim_f\\
[(r,\tau)]&\mapsto\left[\psi(r,\tau)\right],
\end{align*}
is well defined and induces a bijection of sets.
\end{Prop}
\begin{proof} 
See Appendix \ref{appendix2}. 
\end{proof}

\subsection{An adelic action on $\mathcal{L}(\mk{N})$}
In this subsection we first define an ad\'elic action on the set
$\mathcal{L}^{\ca{O}}(\mk{N})/\sim_\mk{g}$.  
This ad\'elic action is the key ingredient that allows to define the
conjectural Shimura reciprocity law for the $p$-adic invariant $u_D$ 
(see Conjecture \ref{alg_conj}).
Second of all, using the identification between $\mathcal{L}(\mk{N})/\sim_f$ and  $\ZZ/f\ZZ\times H(\mk{N})/\sim_f$
(see Proposition \ref{hard_prop}) we transport this ad\'elic action to the set $\ZZ/f\ZZ\times H(\mk{N})/\sim_f$.
This ad\'elic action is used to define the
conjectural Shimura reciprocity law for the $p$-adic invariant $u_C$.

Let $C_K=\mathbf{A}_K^{\times}/K^{\times}$
be the id\`ele class group of $K$. By class field theory we have a short exact sequence
\begin{align*}
0\rightarrow C_{K}^{\circ}\rightarrow C_K\stackrel{rec}{\rightarrow} G_{K^{ab}/K}\rightarrow 0,
\end{align*}
where $C_{K}^{\circ}$ is the connected component of $C_K$ which contains the identity. 
For every integral $\ca{O}_K$-ideal $\mk{m}$ class field theory gives us a natural onto map
\begin{align*}
\pi_{\mk{m}}:C_{K}\rightarrow C_{\ca{O}_K}(\mk{m})
\end{align*} 
where $C_{\ca{O}_K}(\mk{m})=I_{\ca{O}_K}(\mk{m})/P_{\ca{O}_K,1}(\mk{m}\infty)$ stands for the usual 
narrow ray class group of conductor $\mk{m}$ of $K$.
\begin{Def}\label{atyran}
Let $\ca{O}$ be a fixed order of $K$ and let $\mk{g}\subseteq\ca{O}$ be an $\ca{O}$-ideal.
Let $n=\cond{\ca{O}}$, $g=[\ca{O}:\mk{g}]$ and let
\begin{align}\label{kettle}
\pi_{gn}:C_K\rightarrow C_{\ca{O}_K}(gn),
\end{align} 
be the surjective map given by class field theory. For every $c\in C_K$ and 
$[[L,M]]\in\mathcal{L}^{\ca{O}}(\mk{N})/\sim_{\mk{g}}$ (see Definition \ref{tion} for the meaning of $\sim_{\mk{g}}$) we define
\begin{align}\label{action}
c\star[[L,M]]=[[(\mk{a}\cap\ca{O})L,(\mk{a}\cap\ca{O})M]],
\end{align}
where $\pi_{gn}(c)=[\mk{a}]$.
\end{Def}
We note that the group action \eqref{action} is well defined. Indeed 
if $\mk{a},\mk{a}'\in I_{\ca{O}_K}(gn)$ and $\mk{a}\sim_{gn}\mk{a}'$ 
then from Proposition \ref{lilas} one has that $(\mk{a}\cap\ca{O})\sim_{g} (\mk{a}'\cap\ca{O})$ which 
in turns implies that $(\mk{a}\cap\ca{O})\sim_{\mk{g}} (\mk{a}'\cap\ca{O})$. 
Note also that since $(\mk{a},n\ca{O}_K)=\ca{O}_K$ one has that $\mk{a}\cap\ca{O}$ is $\ca{O}$-invertible.
Thus $(\mk{a}\cap\ca{O})L$ and $(\mk{a}\cap\ca{O})M$ are again $\ca{O}$-invertible ideals.

Using Proposition \ref{lilas}
we have a natural onto projection map 
\begin{align*}
\wt{\Theta}:C_{\ca{O}_K}(gn)&\rightarrow C_{\ca{O}}(g)\\
[\mk{a}]&\mapsto[\mk{a}\cap\ca{O}].
\end{align*}
Moreover, the natural map $p:C_{\ca{O}}(g)\rightarrow C_{\ca{O}}(\mk{g})$ is also onto. It thus 
follows that the map $p\circ\wt{\Theta}$ is onto. If one restricts the action of $C_K$ to a stratum 
$\mathcal{L}^{\ca{O}}(\mk{N})/\sim_{\mk{g}}$ then one may check that the action of $C_K$ factors
through the generalized ideal class group $C_{\ca{O}}(\mk{g})$ in the sense that for all
$[[L,M]]\in\mathcal{L}^{\ca{O}}(\mk{N})/\sim_{\mk{g}}$ one has that
\begin{align*}
c\star[[L,M]]=[[\mk{c}L,\mk{c}M]],
\end{align*}
where $[\mk{c}]=p\circ\wt{\Theta}\circ\pi_{gn}(c)$ and $\pi_{gn}$ is the map which appears in \eqref{kettle}.

Thanks to Proposition \ref{hard_prop}, we may now define an action of $C_K$ on 
the set $(\ZZ/f\ZZ\times H(\mk{N}))/\sim_f$.
\begin{Def}
Let $[(r,\tau)]\in(\ZZ/f\ZZ\times H(\mk{N}))/\sim_f$ and let
$\wt{\psi}[(r,\tau)]=[[L,M]]\in\mathcal{L}^{\ca{O}_{\tau}}(\mk{N})/\sim_f$. For $c\in C_K$ we define 
$$
c\star[(r,\tau)]:=\wt{\psi}^{-1}(c\star[[L,M]]).
$$ 
\end{Def}
\section{Zeta functions attached to real quadratic number fields}
In this section we introduce certain zeta functions attached to $K$ that play a key role
in the construction of the two $p$-adic invariants $u_C$ and $u_{D}$. There
are two sets of notation available if one wants to define these zeta functions. 
The first set of notation uses the language of quadratic 
forms. The second set of notation uses the language of ideals. Each of these two ways has its
own advantages and inconveniences. The language of quadratic forms was privileged in \cite{Dar-Das} and 
\cite{Ch1} while the language of ideals  was used in \cite{Das3}. 
One should note that for a general number field, 
the language of quadratic form has no equivalent. The reader will find in this section
the precise definitions of the zeta functions that appear in the construction of 
$u_C$ and $u_D$ and the relations between them.
We also collected some key facts about special values of partial zeta functions at negative integers 
which are used implicitly in the proof of Theorem \ref{vag_th} 

Let $\ca{O}$ be an arbitrary order of $K$ and let $\mk{f}$ be an $\ca{O}$-integral ideal.
We denote the group of
units (resp. totally positive units) of $\ca{O}^{\times}$ which are congruent to $1$ modulo $\mk{f}$ by $\ca{O}(\mk{f})^\times$ 
(resp. $\ca{O}(\mk{f}\infty)^\times$). In the case where $\mk{f}$ is a principal ideal generated
by an integer $f$ we will write $\ca{O}(f\infty)^\times$ rather than the ugly notation $\ca{O}(f\ca{O}\infty)^\times$.

\begin{Def}\label{matrix_def}
Let $\tau\in K\bs\QQ$. We define $\eta_{\tau}=\M{a}{b}{c}{d}$ to be the unique 
generator of $Stab_{\Gamma_1(f)}(\tau)$ such that $c\tau+d>1$. We also set
$\epsilon(\eta_{\tau}):=c\tau+d$.
\end{Def} 
\begin{Def}\label{tarte}
For every pair $(r,\tau)\in\ZZ/f\ZZ\times K\bs\QQ$ we define
\begin{enumerate}
\item[$(1)$] $\zeta((r,\tau),s)=-\displaystyle\sum_{\{(m,n)\in\ZZ^2\bs(0,0)\}/\laa\eta_{\tau}\raa}
\frac{\sign(Q_{\tau}(m,n))}{
|Q_{\tau}(m,n)|^s}e^{\frac{-2\pi ir n}{f}}$,\s\s\s $\Re(s)>1$,
\item[$(2)$] $\wh{\zeta}((r,\tau),s)=f^{2s}\displaystyle\sum_{
\left\{0\neq(m,n)\equiv (r,0)\pmod{f}\right\}/\laa\eta_{\tau}\raa} 
\frac{\sign(Q_{\tau}(m,n))}{|Q_{\tau}(m,n)|^s}$,\s\s\s $\Re(s)>1$.
\end{enumerate}
\end{Def}
The first summation is taken over a complete set of representatives of 
$\left\{(m,n)\in\ZZ^2\bs(0,0)\right\}$ modulo a right action 
of $\eta_{\tau}$ which we define below. Similarly the second summation
is taken over a complete set of representatives of 
\begin{align}\label{sett}
\left\{0\neq(m,n)\in\ZZ^2:(m,n)\equiv (r,0)\pmod{f}\right\}
\end{align}
modulo the right action of $\eta_{\tau}$.

The right action of a matrix $\M{a}{b}{c}{d}\in GL_2(\ZZ)$ on a pair $(x,y)$ is given by
\begin{align}\label{taille}
\R{x}{y}\sharp\M{a}{b}{c}{d}:=(-by+dx,ay-cx). 
\end{align}
We have used the symbol $\sharp$ in order to distinguish
it from the standard right multiplication of a row vector by a $2$ by $2$ matrix. We note that
\begin{align*}
\R{x}{y}\sharp\gamma=\left(\gamma^{-1}\V{x}{y}\right)^t.
\end{align*}
However, we have preferred to avoid the use of the transpose operation on vectors and matrices
so that the notation which appears in \eqref{taille} will be privileged. This right action is forced upon us for the following reason:
Let $\ca{O}(f\infty)^{\times}=\laa\epsilon\raa$ with $\epsilon>1$. Then the action of the unit $\epsilon$
on the column vector $\R{\tau}{1}^t$ is given by the standard left multiplication 
by the matrix $\eta_{\tau}$, i.e.,
\begin{align*}
\eta_{\tau}\V{\tau}{1}=\epsilon\V{\tau}{1}.
\end{align*} 
Now say that $\eta_{\tau}=\M{a}{b}{c}{d}$ then since 
\begin{align*}
\R{-y}{x}\V{\tau}{1}=x-\tau y
\end{align*}
we find that
\begin{align}\label{corbeille}
\R{-y}{x}\V{\epsilon\tau}{\epsilon}=\R{-y}{x}\M{a}{b}{c}{d}\V{\tau}{1},
\end{align}
which forces the right action defined in \eqref{taille}. We note that $(2)$ of Definition \ref{tarte} makes sense since for
$\gamma\in\Gamma_1(f)$ one has
\begin{align*}
\R{r}{0}\sharp\;\gamma\equiv\R{r}{0}\pmod{f}.
\end{align*}
We also note that the summations in $(1)$ and $(2)$ 
don't depend on the choice of the representatives.

\begin{Rem}
This subtle point concerning the right action of $\eta_{\tau}$ on the set
\eqref{sett} was overlooked in \cite{Ch1}. 
It was wrongly stated on p. 24 of \cite{Ch1} that 
the matrix $\eta_{\tau}$ was acting on the left by the rule $(x,y)\mapsto(ax+by,cx+dy)$.
This mistake had no implications for the results stated in \cite{Ch1} but it should nevertheless
be corrected. In fact the right action defined in \eqref{taille} will play 
a key role later on (see for example \eqref{apple} and \eqref{apple2}).
\end{Rem}
\begin{Rem}
The zeta functions appearing in $(1)$ and $(2)$ of Definition \ref{tarte} satisfy the identities
$\zeta((r,\tau),s)=\zeta((-r,\tau),s)$ and $\wh{\zeta}((r,\tau),s)=\wh{\zeta}((-r,\tau),s)$. 
In the special case where $f=1$ one has $\zeta((0,\tau),s)=-\wh{\zeta}((0,\tau),s)$.
One may also check that $\wh{\zeta}((0,\tau),s)$ coincides with the zeta function $\zeta_{Q_{\tau}}(s)$
which appears in equation (55) of \cite{Dar-Das}.
\end{Rem}

Recall that there is an action of $\Gamma_0(f)$ on the set $\ZZ/f\ZZ\times K\bs\QQ$ given by 
\begin{align*}
\gamma\star(r,\tau)=\left(d^{-1}r,\frac{a\tau+b}{c\tau+d}\right)\s\s\mbox{for}
\s\s\gamma=\M{a}{b}{c}{d}\in\Gamma_0(f)\s\s\mbox{and}\s\s (r,\tau)\in\ZZ/f\ZZ\times K\bs\QQ.
\end{align*} 
Let $(r,\tau),(r',\tau')\in\ZZ/f\ZZ\times K\bs\QQ$. A direct calculation shows that if $(r,\tau)$ is equivalent 
to $(r',\tau')$ modulo $\Gamma_0(f)$ then $\zeta((r,\tau),s)=\zeta((r',\tau'),s)$ and
$\wh{\zeta}((r,\tau),s)=\wh{\zeta}((r',\tau'),s)$.

The zeta functions $(1)$ and $(2)$ are related by a functional equation.
\begin{Th}\label{coal}
The function $\zeta((r,\tau),s)$ admits a meromorphic continuation to all of $\CC$. Moreover,
let $F_{w_1}(s)=\disc(Q_{\tau})^{s/2}\pi^{-s}\Gamma\left(\frac{s+1}{2}\right)^2$. Then
\begin{align}\label{charbon}
-F_{w_1}(s)\zeta((r,\tau),s)=F_{w_1}(1-s)\wh{\zeta}((r,\tau),1-s).
\end{align}
\end{Th}
\begin{proof}
This result can be deduced from the computations carried by Siegel in \cite{Sie2}. 
For a detailed proof which builds on ideas of \cite{Sie2}, see the proof of 
Theorem 8.2 of \cite{Ch}. 
For a different proof which uses the functional equation of a theta function see Theorem 1.1 of \cite{Ch4}. \fin
\end{proof}
\begin{Rem}
In particular, Theorem \ref{coal} allows us to give a meaning to the special values of 
$\wh{\zeta}((r,\tau),s)$ at negative integers. 
The zeta function appearing in $(1)$ will not be used in this paper. The author 
included it only in order to state the functional equation \eqref{charbon}.
\end{Rem}

Let $\ca{O}$ be a fixed order of $K$ and let $\mk{f}$ be an
$\ca{O}$-integral ideal (not necessarily invertible) which is distinct from $\ca{O}$. 
\begin{Def}
We say that an invertible $\ca{O}$-ideal $\mk{c}$ is $\mk{f}$-int if $\mk{c}$ can be written as
$\mk{c}=\mk{a}\mk{b}^{-1}$ where $\mk{a}$ and $\mk{b}$ are invertible integral $\ca{O}$-ideals such that
$(\mk{b},\mk{f})=\ca{O}$. 
\end{Def}
\begin{Rem}
We have used the notation $\mk{f}$-int in order to avoid any confusion with the 
notion of $\ca{O}$-integrality that was defined in Section \ref{notations}. Note also
that when $\mk{c}$ is an integral invertible $\ca{O}$-ideal then it is automatically
$\mk{f}$-int.
\end{Rem}
\begin{Def}\label{units}
Let $\mk{c}$ be an invertible $\ca{O}$-ideal which is $\mk{f}$-int and choose a writing
$\mk{c}=\mk{a}\mk{b}^{-1}$ such that $(\mk{b},\mk{f})=\ca{O}$.
We define
\begin{align*}
\Gamma_{\mk{c}}(\mk{f}):=\ca{O}(\infty)^{\times}\cap(\mk{f}\mk{c}^{-1}+n_{\mk{c}}),
\end{align*}
where $n_{\mk{c}}\in\ZZ$ is an arbitrarily chosen integer contained in $\mk{b}\subseteq\mk{c}^{-1}$ 
such that $n_{\mk{c}}\equiv 1\pmod{\mk{f}}$.
Note that the existence of such an integer $n_{\mk{c}}$ is guaranteed precisely because 
$\mk{c}$ is assumed to be $\mk{f}$-int.
\end{Def}
\begin{Rem} 
Let $\mk{a}$ and $\mk{b}$ be $\mk{f}$-int invertible $\ca{O}$-ideals. Then 
if there exists a $\lambda\in K^{\times}$ such that $\mk{a}=\lambda\mk{b}$ and that
$(\mk{a},\mk{f})=(\mk{b},\mk{f})$ then one has that $\Gamma_{\mk{a}}(\mk{f})=\Gamma_{\mk{b}}(\mk{f})$.
In particular, if one has that $\mk{a}\sim_{\mk{f}}\mk{b}$ then 
$\Gamma_{\mk{b}}(\mk{f})=\Gamma_{\mk{a}}(\mk{f})$.
\end{Rem}

We would like now to introduce partial zeta functions twisted by a \textit{sign character} 
$w:K^{\times}\rightarrow\{\pm 1\}$.
\begin{Def}
Let $\mk{c}$ be an invertible $\ca{O}$-ideal which is $\mk{f}$-int. We define
\begin{align*}
\zeta(\mk{c},\mk{f},w,s)=\Norm(\mk{c})^{-s}
\sum_{\Gamma_{\mk{c}}(\mk{f})\bs\{0\neq\mu\in(\mk{c}^{-1}\mk{f}+n_{\mk{c}})\}}
\frac{w(\mu)}{|\Norm_{K/\QQ}(\mu)|^s},\hspace{1.5cm} \Re(s)>1,
\end{align*}
where $n_{\mk{c}}\in\ZZ$ is an arbitrarily chosen integer as in Definition \ref{units}.
\end{Def}
The summation is taken over a complete set of representatives of $\{0\neq\mu\in(\mk{f}\mk{c}^{-1}+n_{\mk{c}})\}$
modulo $\Gamma_{\mk{c}}(\mk{f})$. One may check that $\zeta(\mk{c},\mk{f},w,s)$ is independent of 
the choice of $n_{\mk{c}}$ and that the first entry of $\zeta(\mk{c},\mk{f},w,s)$ depends on $\mk{c}$
only modulo $\sim_{\mk{f}}$. Let $\mk{b}$ be an invertible integral $\ca{O}$-ideal. Then 
if $\mk{c}$ is $\mk{f}$-int the $\ca{O}$-ideal $\mk{c}\mk{b}$ is automatically $\mk{f}\mk{b}$-integral and one has
\begin{align}\label{habiba}
\zeta(\mk{c},\mk{f},w,s)=\zeta(\mk{c}\mk{b},\mk{b}\mk{f},w,s).
\end{align}
\begin{Rem}
Similarly, one can define such zeta functions for an arbitrary number field. 
A functional equation similar to \eqref{charbon} holds, see \cite{Ch4}.
In \cite{Ch8}, the author worked out some of their arithmetic properties. 
\end{Rem}

We also define \textit{classical partial zeta functions} attached to an arbitrary order $\ca{O}$. 
\begin{Def}
For every invertible $\ca{O}$-ideal $\mk{c}$  we define
\begin{align*}
\zeta(\mk{c},\mk{f}\infty,s)=\Norm(\mk{c})^{-s}
\sum_{\Gamma_{\mk{c}}(\mk{f})\bs\{0\neq\mu\in(\mk{f}\mk{c}^{-1}+n_{\mk{c}}),\mu\gg 0\}}
\frac{1}{|\Norm_{K/\QQ}(\mu)|^s}\hspace{1.5cm} \Re(s)>1.
\end{align*}
\end{Def}
This function admits a meromorphic continuation to all of $\CC$ with a single pole (of order one)
at $s=1$. If $[\mk{a},\mk{f}]$ is primitive of conductor $[\ca{O},\mk{g}]$ we say that
$\zeta(\mk{a},\mk{f},w,s)$ (resp. $\zeta(\mk{a},\mk{f}\infty,s)$)
is primitive of conductor $[\ca{O},\mk{g}]$.
\begin{Rem}
If the ideal $\mk{a}$ is an integral $\ca{O}_K$-ideal (necessarily invertible) 
coprime to $\mk{f}$ one readily sees that
\begin{align}\label{fishy}
\zeta(\mk{a},\mk{f}\infty,s)=\zeta_R(K(\mk{f}\infty)/K,\sigma_{\mk{a}},s),
\end{align}
where $\sigma_{\mk{a}}\in G_{\mk{f}}:=\Gal(K(\mk{f}\infty)/K)$ is the Frobenius at $\mk{a}$ and
$R$ is the set of finite places of $K$ which divide $\mk{f}$.
The zeta function $\zeta_R(K(\mk{f}\infty)/K,\sigma_{\mk{a}},s)$ 
is the one that appears on the left hand side of \eqref{rabbit}.
\end{Rem}

The next two identities below give us a way to write a   
partial zeta function weighted by a sign character $w$
as a linear combination of classical partial zeta functions and vice-versa.
Let $\{\lambda_{i}\}_{i=1}^4$ be a complete set of representatives
of $Q_{\ca{O},1}(\mk{f})/Q_{\ca{O},1}(\mk{f}\infty)$ (for the definitions of $Q_{\ca{O},1}(\mk{f})$
and $Q_{\ca{O},1}(\mk{f}\infty)$ see Section \ref{notations}). Then for any invertible $\ca{O}$-ideal
$\mk{b}$ which is $\mk{f}$-int and any sign character $w$, a direct computation shows that
\begin{align}\label{tortue}
\zeta(\mk{b},\mk{f},w,s)=\sum_{i=1}^4|\Norm(\lambda_i)|^{-s}w(\lambda_i)\zeta(\mk{b}\lambda_i,\mk{f}\infty,s),
\end{align}
and that
\begin{align}\label{tortue2}
4\zeta(\mk{b},\mk{f}\infty,s)=
\sum_{\mbox{\tiny{$w$ is a sign character}}}\zeta(\mk{b},\mk{f},w,s).
\end{align}

There are two sign characters that will play a special role in our context, namely 
$w_0:=1$ and $w_1:=\sign\circ\Norm_{K/\QQ}$. 
We have the following lemma which relates special values 
at negative integers of $\zeta(\mk{a},\mk{f},w_i,s)$ (for $i\in\{0,1\}$) 
with the ones of $\zeta(\mk{a},\mk{f}\infty,s)$.
\begin{Lemma}\label{quiero}
Let $k\in\ZZ_{\geq 1}$ and $k\equiv i\pmod{2}$. Then
\begin{align*}
4\zeta(\mk{a},\mk{f}\infty,1-k)=\zeta(\mk{a},\mk{f},w_i,1-k).
\end{align*}
\end{Lemma}
\begin{proof}
This lemma follows from \eqref{tortue2} and the key observation 
that $\zeta(\mk{a},\mk{f},w,1-k)=0$
for $w=\pm \sign$ and $k\in\ZZ_{\geq 1}$. For a proof and an explanation 
of the latter fact see p. 812 of \cite{Ch8}. \fin
\end{proof}
\begin{Rem}
One can show that if there exists an $\epsilon\in\ca{O}(\mk{f})^{\times}$ such that 
$w_1(\epsilon)=\sign(\Norm_{K/\QQ}(\epsilon))=-1$ then 
$\zeta(\mk{a},\mk{f},w_1,s)$ is identically equal to $0$.
\end{Rem}
\begin{Def}\label{index}
Let $G_1$ and $G_2$ be subgroups of $\ca{O}_K(\infty)^{\times}$ where 
$G_1=\laa\epsilon_1\raa$, $G_2=\laa\epsilon_2\raa$ and $\epsilon_1,\epsilon_2>1$. 
We define 
\begin{align*}
[G_1:G_2]\in\QQ_{>0}
\end{align*} 
to be the unique positive rational number $r$ such that $\epsilon_1^{r}=\epsilon_2$.
\end{Def}

The next lemma says the zeta function associated to a quadratic form which appears in $(2)$ 
of Definition \ref{tarte} is the same (up to a simple fudge factor) to a partial zeta function twisted by the sign 
character $w_1$. 
\begin{Lemma}\label{convert}
Let $(r,\tau)\in\ZZ/f\ZZ\times K\bs\QQ$ and let $\mk{f}=f\ca{O}_{\tau}$.
Then 
\begin{align}\label{kase}
\wh{\zeta}((r,\tau),s)=\mu(r,\tau)\Norm(\mk{f})^s
\zeta(\wt{r}I_{\tau},\mk{f},w_1,s),
\end{align}
where 
\begin{align}\label{sang}
\mu(r,\tau):=[\Gamma_{\wt{r}I_{\tau}}(\mk{f}):\laa\epsilon(\eta_{\tau})\raa]\in\QQ_{>0}.
\end{align}
\end{Lemma}
\begin{proof}
This is straightforward computation. \fin
\end{proof}
\begin{Rem}\label{fromage}
We note that in the special case where $\tau\in H^{\ca{O}_K}(\mk{N})$ is such that $(A_{\tau},f)=1$, 
it may be shown that $\epsilon(\eta_{\tau})$ is a generator of $\ca{O}_K(\mk{f}\infty)^{\times}$
and therefore if $r\in(\ZZ/f\ZZ)^{\times}$ one has that $\mu(r,\tau)=1$.
\end{Rem}
\section{Two groups of divisors}
In this section we introduce two groups of divisors $\Div_f(N)$ and $\wt{\Div}(\mk{N})$ 
which are free $\ZZ$-modules of finite type
endowed with an additional structure of a $\ZZ/f\ZZ$-module. The first group of divisors 
captures in essence the choice of the family of modular units $\{\beta_r(z)\}_{r\in\ZZ/f\ZZ}$ 
which appeared in the introduction. The 
second group of divisors may be viewed as a convenient way of \lq\lq twisting\rq\rq\;Shintani zeta functions
in order to regularize their special values at negative integers. 

For the rest of the section we fix a triple $(p,f,N)$ as in the introduction where now $N$ 
is no longer assumed to be square-free. We let 
$N=\prod_{i=1}^r l_i^{e_i}$ and $\mk{N}=\prod_{i=1}^r\eta_i^{e_i}$ where $\eta_i$ is a choice of
a prime of $\ca{O}_K$ above $l_i$. Thus one has $\ca{O}_K/\mk{N}\ca{O}_K\simeq\ZZ/N\ZZ$.
\subsection{The group $\Div_f(N)$}
Let $\Div(N)$ be the free abelian group generated by the symbols 
$\{[d]:d|N,d>0\}$ and let
\begin{align*}
\Deg:\Div(N)&\rightarrow\ZZ\\[2mm]
\sum_{d|N}n(d)[d]&\mapsto\sum_{d|N}n(d)d,
\end{align*}
be the degree map. We denote by
$\Div^0(N)$ the kernel of $\Deg$. It is also convenient to define an involution 
$*:\Div(N)\rightarrow \Div(N)$ given by the rule $[d]^*=\left[\frac{N}{d}\right]$. Thus
if $\delta=\sum_{d|N}n(d)[d]$ one has $\delta^*=\sum_{d|N}n(d)\left[\frac{N}{d}\right]$. We define 
\begin{align*}
\Div_f(N):=\Div(N)\times\ZZ/f\ZZ.
\end{align*} 
A typical element of $\delta\in \Div_f(N)$ will be denoted by
\begin{align*}
\delta=\sum_{d|N,r\in\ZZ/f\ZZ}n(d,r)[d,r]\s\s\s\mbox{where}\s\s n(d,r)\in\ZZ,
\end{align*} 
where $[d,n]$ is a shorthand notation for the element $([d],r)\in\Div(N)\times\ZZ/f\ZZ$.

We endow the abelian group $\Div_f(N)$ with a left $\ZZ/f\ZZ$-module structure
which is given on a generator $[d,j]$ by the rule $r\star[d,j]=[d,rj]$ for $r\in\ZZ/f\ZZ$.
Then we extend $\star$ to all of $\Div_f(N)$ by $\ZZ$-linearity.
For a divisor $\delta\in \Div_f(N)$
we will also use the short hand notation $\delta_r$ which is taken to mean $r\star\delta$.
We extend the involution $*$ to $\Div_f(N)$ by the rule $[d,r]^*=[\frac{N}{d},-r]$. The 
$\ZZ$-module $\Div_f(N)$ admits a direct sum decomposition given by
\begin{align*}
\Div_f(N)=\bigoplus_{t\in\ZZ/f\ZZ}\Div_f(N)_t,
\end{align*}
where $\Div_f(N)_t=\{\delta\in \Div_f(N):\delta=\sum_{d|N}n(d)[d,t]\}$. 
For every $t\in\ZZ/f\ZZ$, we have a natural projection maps $\pi_t:\Div_f(N)\rightarrow \Div_f(N)_t$ given by
$\sum_{d|N,r\in\ZZ/f\ZZ}n(d,r)[d,r]\mapsto\sum_{d|N}n(d,t)[d]$. We also have 
a degree map on $\Div_f(N)_t$ which is given by $\sum_{d|N}n(d,t)[d]\mapsto\sum_{d|N}n(d,t)d$.

Let $(p,f,N)$ be a triple as in the introduction. The next definition is crucial for 
the construction of the $p$-adic invariants $u_{DD}$ and $u_C$.
\begin{Def}\label{good_div}
We say that a divisor $\delta\in \Div_f(N)$ is a good divisor (with respect to the triple $p,f,N$) if
the following two conditions are satisfied: 
\begin{enumerate}
\item[$(1)$] For all $t\in\ZZ/\ZZ$ one have $\Deg(\pi_t(\delta))=0$,
\item[$(2)$] $p\star\delta=\delta$.
\end{enumerate}
\end{Def}
\begin{Rem}
In Appendix \ref{goeland}, it is explained how to associate a modular unit $\beta_{\delta}(z)$ to a divisor 
$\delta\in \Div_f(N)$. The condition $(1)$ implies that the 
modular unit $\beta_{\delta}(z)$ has no zeros nor
poles on the set $\Gamma_0(fN)\{\infty\}$ and $(2)$ implies 
that $\frac{\beta_{\delta}(z)}{\beta_{\delta}(pz)}$
is $U_p$-invariant (see Section 4 of \cite{Ch_T} for further details). 
\end{Rem}
\subsection{The group $\wt{\Div}_f(\mk{N})$}\label{div_tilde}
In this subsection we fix an order $\ca{O}$, an $\ca{O}$-integral ideal
$\mk{f}=f\ca{O}$ and an element $\tau\in H^{\ca{O}}(\mk{N})$. We let
$S$ be the set of places of $K$ that consists exactly of all the infinite ones, $\wp=p\ca{O}_K$ 
and the ones that ramify in $K(\mk{f}\infty)$. We note that $K(\mk{f}\infty)$ is an 
unramified extension of $K$ outside $f\cdot \cond(\ca{O})$. We let $R=S\bs\{\wp\}$. 

Define $\wt{\Div}(\mk{N})$ be the free $\ZZ$-module generated by the symbols 
$[\mk{d}]$ where $\mk{d}|\mk{N}$. An element
$\wt{\delta}\in\wt{\Div}(\mk{N})$ will be denoted by
\begin{align*}
\sum_{\mk{d}|\mk{N}} n(\mk{d})[\mk{d}],
\end{align*}
where $n(\mk{d})\in\ZZ$. We also define an involution $*$
on $\wt{\Div}(\mk{N})$ which is given on generators by 
$[\mk{d}]^*=\left[\frac{\mk{N}}{\mk{d}}\right]$.

We now fix an identification between $\wt{\Div}(\mk{N})$ and $\Div(N)$.
\begin{Def}\label{phi}
Let
$$
\Phi:\wt{\Div}(\mk{N})\rightarrow \Div(N)
$$  
be the map which is given on generators by $[\mk{d}]
\mapsto\left[\frac{\Norm(\mk{N})}{\Norm(\mk{d})}\right]$. Then we extend
$\Phi$ to all of $\wt{\Div}(\mk{N})$ by $\ZZ$-linearity.
\end{Def}

We note that $\Phi(\wt{\delta}^*)=\Phi(\wt{\delta})^*$

Because of our choice of $\Phi$ we define the degree map on $\wt{\Div}(\mk{N})$ in the following way:
\begin{align*}
\Deg:\wt{\Div}(\mk{N})&\rightarrow\ZZ\\
\sum_{\mk{d}\mk{N}} n(\mk{d})[\mk{d}]
&\mapsto\sum_{\mk{d}|\mk{N}} n(\mk{d})\Norm\left(\frac{\mk{N}}{\mk{d}}\right).
\end{align*} 
In this way, elements of degree $0$ in 
$\wt{\Div}(\mk{N})$ map to elements of degree $0$ in $\Div(N)$ under $\Phi$. 
We denote the kernel of $\Deg$ by $\wt{\Div}^{0}(\mk{N})$.
\begin{Def}
Let $\mk{d}'$ and $\mk{d}$ be two divisors of $\mk{N}$. We say that 
$\mk{d}'$ is consecutive to $\mk{d}$ if $\frac{\mk{d}'}{\mk{d}}=\eta_i$
for some $i$.
\end{Def}

The next lemma will play a key role later on
\begin{Lemma}\label{tired}
The group $\wt{\Div}^{0}(\mk{N})$ is generated by elements of the form 
\begin{align*}
d'[\mk{d}']-d[\mk{d}],
\end{align*}
where $\mk{d}'$ is consecutive to $\mk{d}$, $d'=\Norm(\mk{d}')$ and $d=\Norm(\mk{d})$.
\end{Lemma}
\begin{proof}
This is an easy induction argument. \fin
\end{proof}
We define $\wt{\Div}_f(\mk{N}):=\wt{\Div}(\mk{N})\times\ZZ/f\ZZ$ and endow it with its natural structure
of $\ZZ/f\ZZ$-module. An element $\wt{\delta}\in\wt{\Div}_f(\mk{N})$ will be denoted as
$\sum_{\mk{d}|\mk{N},r\in\ZZ/f\ZZ}n(\mk{d},r)[\mk{d},r]$. There is 
a natural direct sum decomposition of $\wt{\Div}(\mk{N})$ which is given by
\begin{align*}
\wt{\Div}_f(\mk{N})=\bigoplus_{t\in\ZZ/f\ZZ}\wt{\Div}_f(\mk{N})_t,
\end{align*}  
where $\wt{\Div}_f(\mk{N})_t=\{\wt{\delta}\in\wt{\Div}(\mk{N}):\wt{\delta}=\sum_{\mk{d}|\mk{N}}n(\mk{d})[\mk{d},t]\}$.
The isomorphism $\Phi$
induces in a natural way an isomorphism of $\ZZ/f\ZZ$-module between $\wt{\Div}_f(\mk{N})$ and $\Div_f(N)$ 
which we again denote by $\Phi$. 
For every $t\in\ZZ/f\ZZ$, we let $\pi_t:\wt{\Div}_f(\mk{N})\rightarrow\wt{\Div}_f(\mk{N})_t$ denote the natural projection map. 
Note that $\pi_t\circ\Phi=\Phi\circ\pi_t$.

Finally, we say that a divisor $\wt{\delta}\in\wt{\Div}_f(\mk{N})$ 
is a good divisor (with respect to the triple $(p,f,N)$) if and only if 
$\Phi(\wt{\delta})\in \Div_f(N)$ is a good divisor.
\section{Zeta functions attached to divisors}
Let $(p,f,N)$ be a triple as in the introduction. Now we would like to associate to any divisor $\delta\in \Div_f(N)$ 
and every pair $(r,\tau)\in\ZZ/f\ZZ\times K\bs\QQ$ two zeta functions.
One can think of these zeta functions as being \lq\lq twisted\rq\rq\; by the divisor $\delta$.
\begin{Def}\label{def1}
Let $\delta=\sum_{d|N,r\in\ZZ/f\ZZ}n(d,r)[d,r]\in \Div_f(N)$ be a good divisor and
let $(r,\tau)\in\ZZ/f\ZZ\times K\bs\QQ$. We define
\begin{enumerate}
\item[$(1)$] $\zeta(\delta,(r,\tau),s)=\displaystyle\sum_{d|N,j\in\ZZ/f\ZZ}
n\left(d,j\right)d^{-s}\cdot\wh{\zeta}((rj,d\tau),s)$,
\item[$(2)$] $\zeta^*(\delta,(r,\tau),s)=\displaystyle\sum_{d|N,j\in\ZZ/f\ZZ}
n\left(\frac{N}{d},j\right)\left(\frac{N}{d}\right)^{s}\cdot\wh{\zeta}((-rj,d\tau^*),s)$,
\end{enumerate}
where $\tau^*=\frac{-1}{fN\tau}$ and $\wh{\zeta}$ is the zeta function which appears in $(2)$
of Definition \ref{tarte}.
\end{Def}
\begin{Rem}
We would like to warn the reader about a typo in \cite{Ch1}. In $(2)$ of Definition 5.4 of \cite{Ch1} 
one should replace the term $d_0^{s}$ by $(\frac{N_0}{d_0})^s$. Also the definition of $(1)$ above varies slightly from
the definition which appears in Definition 5.4 of \cite{Ch1} and Definition 9.2 of \cite{Ch_T}. 
This should not result in any inconsistency since all the statements that we will make
in this paper concerning $\zeta(\delta,(r,\tau),s)$ will be proved without using any  
reference to the papers \cite{Ch1} and \cite{Ch_T}. We note however that 
the special value at $s=0$ is independent of the chosen definition. 
\end{Rem}
\begin{Lemma}\label{efkaristo}
We have the following relations between the three types of zeta functions:
\begin{align}\label{chev}
\zeta^*(\delta,(r,\tau),s)=[\epsilon(\eta_{\tau}):\epsilon(\eta_{\tau^*})]N^s\zeta(\delta^*,(r,\tau^*),s).
\end{align}
\end{Lemma}
\begin{proof}
This is a straightforward computation. \fin
\end{proof}
In a similar way, one may associate a zeta function to a divisor $\wt{\delta}\in\wt{\Div}_f(\mk{N})$.
\begin{Def}\label{def2}
Let $\mk{c}$ be an invertible $\ca{O}$-ideal which is $\mk{f}$-int. For every divisor 
$\wt{\delta}=\sum_{\mk{d}|\mk{N}}n(\mk{d},r)[\mk{d},r]\in\wt{\Div}_f(\mk{N})$ we define
\begin{align*}
\zeta_{\wt{\delta}}(\mk{c},\mk{f}\infty,s)&:=\sum_{\mk{d}|\mk{N},r\in\ZZ/f\ZZ}
n(\mk{d},r)
\Norm(\mk{d})^{-s}\zeta\left(\wt{r}\mk{c}(\mk{d}^{\sigma})^{-1},\mk{f}\infty,s\right).
\end{align*}
\end{Def}
\begin{Rem}
Assume that $\ca{O}=\ca{O}_K$ and that $N$ is square-free. Set $T=\{\eta:\eta|\mk{N}\}$ and
$\wt{\delta}=\prod_{i=1}^r([1]-l_i[\eta_i])\in\wt{\Div}(\mk{N})$. Then in this case one has
\begin{align}\label{silo}
\zeta_{\wt{\delta}}(\mk{a},\mk{f}\infty,s)
&=\zeta_{R,T^{\sigma}}(K(\mk{f}\infty)/K,\sigma_{\mk{a}},s),
\end{align}
where $\sigma_{\mk{a}}\in G_{\mk{f}}$ 
is the Frobenius at $\mk{a}$
and $\zeta_{R,T^{\sigma}}(K(\mk{f}\infty/K)\sigma_{\mk{a}}^{-1},s)$
is the zeta function which appears in \eqref{protugal} and $T^{\sigma}=\{\eta^{\sigma}:\eta\in T\}$.
\end{Rem}
The next lemma relates special values at even negative integers of the 
zeta function which appears in $(1)$ of Definition \ref{def1} with the
one which appears in Definition \ref{def2}.
\begin{Lemma}
Assume that $(r,\tau)\in \ZZ/f\ZZ\times H^{\ca{O}}(\mk{N})$. 
Let $k\in\ZZ_{\geq 1}$ with $k\equiv 1\pmod{2}$. Then we have
\begin{align}\label{lievre}
4\mu(r,\tau)\zeta_{\wt{\delta}}(\wt{r}I_{\tau},\mk{f}\infty,1-k)
=\Norm(\mk{f})^{-(1-k)}\zeta(\delta^{*},(r,\tau),1-k),
\end{align}
where $\Phi(\wt{\delta})=\delta$ and $\mu(r,\tau)$ is the positive rational number which appears in \eqref{sang} and
$I_{\tau}$ is the ideal attached to $\tau$ which appears in Definition \ref{ini_def}.
\end{Lemma}
\begin{proof}
Let $\wt{\delta}=\sum_{d|N,j\in\ZZ/f\ZZ}n(\mk{d},j)[\mk{d}]\in\wt{\Div}_f(\mk{N})$.
By definition of the application $\Phi$ (see Definition \ref{phi}) we have
\begin{align*}
\Phi(\wt{\delta})=\sum_{d|N}n\left(d,j\right)\left[\frac{N}{d}\right]\s\s\s\mbox{and therefore}\s\s\s
\Phi(\wt{\delta}^*)=\sum_{d|N}n\left(d,j\right)\left[d\right],
\end{align*}
where for all $j\in\ZZ/f\ZZ$ and $\mk{d}|\mk{N}$ one has that $n(\Norm(\mk{d}),j)=n(\mk{d},j)$. 
Set $\mk{a}=\wt{r}I_{\tau}$. By definition we have
\begin{align}\label{bruit_1}
\zeta_{\wt{\delta}}(\mk{a},\mk{f}\infty,s)
&=\sum_{\mk{d}|\mk{N},j\in\ZZ/f\ZZ}n(\mk{d},j)\Norm(\mk{d})^{-s}\zeta
\left(\frac{\mk{a}\mk{d}}{\Norm(\mk{d})},\mk{f}\infty,s\right).
\end{align}
For every divisor $\mk{d}|\mk{N}$ one has that
\begin{align}\label{bruit_demi}
\zeta\left(\frac{\mk{a}\mk{d}}{\Norm(\mk{d})},\mk{f},w_1,s\right)=
\zeta\left(\frac{\wt{r}A_{\tau}}{d}\Lambda_{\tau}\mk{d},
\mk{f},w_1,s\right)=\frac{1}{\mu(r,\tau)}\Norm(\mk{f})^{-s}\wh{\zeta}((r,d\tau),s),
\end{align}
where the first equality follows from the 
fact that $\frac{A_{\tau}}{d}=A_{d\tau}$, $\mk{d}\Lambda_{\tau}
=\Lambda_{d\tau}$ and the second equality follows from Lemma \ref{convert}.

Combining \eqref{bruit_demi} and Lemma \ref{quiero} one finds that for every odd integer $k\in\ZZ_{\geq 1}$
\begin{align}\label{bruit_2}
4\zeta\left(\frac{\wt{r}A_{\tau}}{d}\Lambda_{\tau}\mk{d},
\mk{f}\infty,1-k\right)
=\Norm(\mk{f})^{-(1-k)}\wh{\zeta}((r,d\tau),1-k).
\end{align}
Finally, combining \eqref{bruit_1} and \eqref{bruit_2} we find that
$$
4\zeta_{\wt{\delta}}(\wt{r}I_{\tau},\mk{f}\infty,1-k)=
\mu(r,\tau)\Norm(\mk{f})^{-(1-k)}\zeta(\delta^*,(r,\tau),1-k),
$$
where $\Phi(\wt{\delta})=\delta$. \fin
\end{proof}
\begin{Rem}
Suppose that $\wt{\delta}\in\wt{\Div}(\mk{N})$ is chosen so that $p\star\wt{\delta}=\wt{\delta}$. 
Then a direct computation shows that
\begin{align}\label{fille}
\zeta_{\wt{\delta},\{p\}}
(\mk{a},\mk{f}\infty,s)=\left(1-\frac{1}{p^{2s}}\right)\zeta_{\wt{\delta}}
(\mk{a},\mk{f}\infty,s).
\end{align}
The subscript $\{p\}$ means that we only sum over
elements which are coprime to $p$. Assume that every prime divisor $\mk{q}|\mk{f}$ is invertible (for example
this is automatically satisfied when $\ca{O}$ is the maximal order). Then under this assumption,
for all residue class $r\in\ZZ/f\ZZ$ one has that 
$(\mk{f},\wt{r}\mk{a})$ is again $\ca{O}$-invertible. Then, 
in light of the identity \eqref{habiba} we see that the 
zeta function $\zeta_{\wt{\delta},\{p\}}(\mk{a},\mk{f}\infty,s)$ may be written in the following form
 \begin{align}\label{key_eqn}
\zeta_{\wt{\delta},\{p\}}(\mk{a},\mk{f}\infty,s)&=\sum_{\mk{d}|\mk{N},r\in\ZZ/f\ZZ}
n(\mk{d},r)
\Norm(\mk{d})^{-s}\zeta_{\{p\}}\left(\frac{\wt{r}\mk{a}\mk{d}}{(\mk{f},\wt{r}\mk{a})},
\frac{\mk{f}}{(\mk{f},\wt{r}\mk{a})}\infty,s\right),
\end{align}
which is a linear combination of \textit{primitive zeta functions} of conductor $\mk{f}'$ for various
divisors $\mk{f}'|\mk{f}$. In particular, it follows from the
work of Deligne and Ribet (see \cite{D-R}) that the special values of 
$\zeta_{\wt{\delta},\{p\}}(\mk{a},\mk{f}\infty,s)$ at negative even integers which are congruent 
to $0$ modulo $2(p-1)$ can be $p$-adically interpolated. We denote the corresponding $p$-adic zeta function by 
$\zeta_{\wt{\delta},p}(\mk{a},\mk{f}\infty,s)$. 
\end{Rem}
\section{Shintani zeta functions}
In this section we first recall the notion of a cone decomposition 
in the setting of a real quadratic number field. Then we introduce the so-called 
Shintani zeta function associated to the choice of such a cone decomposition. Recall that $K$ comes 
equipped with a fixed embedding $K\subseteq\RR$ and that $\sigma:K\rightarrow K$ denotes the
non-trivial automorphism of $K$.
Let $\iota:K\to\RR^2$ be the embedding given by $x\mapsto(x,x^{\sigma})$. Under
$\iota$, every invertible $\ca{O}$-ideal $\mk{a}\subseteq K$ may be viewed as a lattice
in $\RR^2$. For the rest of this section, we will view elements of $K$ as elements of $\RR^2$
via $\iota$. Note that the group $K^{\times}$ acts naturally on $\RR^2$ and that
its subgroup $K_+^{\times}=\{x\in K^{\times}:x\gg 0\}$ acts naturally on the positive
quadrant $Q:=\RR_{>0}^2$.
\begin{Def}
For $\RR$-linearly independent vectors $v_1,\ldots,v_r\in Q\cap K$ (in our setting $r\leq 2$)
we let
\begin{align*}
C=C(v_1\ldots,v_r)=\left\{\sum_{i=1}^r c_iv_i\in Q:c_i\in\RR_{>0}\right\}.
\end{align*}
We call $C$ a Shintani cone of dimension $r$. We say that $\ca{D}\subseteq Q$ is a Shintani set
if it can be written as a finite disjoint union of Shintani cones.
\end{Def}

Let $W$ be a finite union of Shintani cones. One can show that
$W$ can always be written as a finite
disjoint union of possibly smaller Shintani cones. Therefore $W$ 
is a Shintani set.

Let $\ca{O}$ be an arbitrary order of $K$ and let $\mk{f}=f\ca{O}$.
Let $\ca{O}(\mk{f}\infty)^{\times}=\laa\epsilon\raa$ 
where $\epsilon$ is the unique generator such that $\epsilon>1$.
The group $\ca{O}(\mk{f}\infty)^{\times}$ acts discretely on $Q$ and admits a fundamental 
domain (the existence of a fundamental in the setting of 
real quadratic field is obvious). It is convenient to define 
a privileged choice of a fundamental domain for the action of $\ca{O}(\mk{f}\infty)^{\times}$ on $Q$, namely
\begin{align}\label{acro}
\mathcal{D}_{\mk{f}}^{can}:=C(1,\epsilon)\cup C(1).
\end{align}

Now we would like to extend slightly a key definition that was introduced in \cite{Das3}.
\begin{Def}\label{good_prime}
A prime $\ca{O}_K$-ideal $\eta$ is called $\ca{O}$-{\bf good} for a Shintani cone $C$ if
\begin{enumerate}
\item[$(1)$] $\Norm(\eta)$ is a rational prime $\ell$;
\item[$(2)$] the cone $C$ may be written $C=C(v_1,\ldots,v_r)$ such that 
for all $i$, $v_i\in\ca{O}$ and $v_i\notin(\eta\cap\ca{O})$.
\end{enumerate}
In general we say that $\eta$ is $\ca{O}$-good for a subset $W\subseteq Q$
if $W$ may be written as a finite disjoint union of Shintani cones 
$W=\bigcup_i C_i$ such that
$\eta$ is $\ca{O}$-good for each $C_i$.
\end{Def}
\begin{Rem}
Let $\mk{f}=f\ca{O}_K$. Then any prime $\ca{O}_K$-ideal $\eta$ of degree $1$ is 
$\ca{O}_K$-good for the canonical fundamental domain $\mathcal{D}_{\mk{f}}^{can}$. Suppose that $\eta$ is
$\ca{O}_K$-good for a Shintani cone $C=C(v_1,v_2)$ where $v_1,v_2\in\ca{O}_K$
and $v_1,v_2\notin\eta$. Let $l=\Norm(\eta)$. Then if $n$ is a positive
integer coprime to $l$ we readily see that $nv_1,nv_2\in\ca{O}_n$ and
$nv_1,nv_2\notin\eta_n=(\ca{O}_n\cap\eta)$. In particular, $\eta$ is 
$\ca{O}_n$-good for the Shintani cone $C$.
\end{Rem}
\begin{Def}
A finite set of places $T$ of $K$ is said to be $\ca{O}$-good for a Shintani set $\ca{D}$ if $\ca{D}$ can be written as a finite
disjoint union of Shintani cones $\ca{D}=\bigcup C_i$ such that for each cone $C_i$ there are
at least two primes in $T$ of different residue characteristic which are $\ca{O}$-good for $C_i$ or 
one prime $\eta\in T$ which is $\ca{O}$-good for $C_i$ and has absolute ramification $\leq l-2$ where
$l=\Norm(\eta)$.
\end{Def}

For the rest of the section we fix an order $\ca{O}$ of $K$ and an $\ca{O}$-integral ideal $\mk{f}$.
\begin{Def}
Let $W\subseteq Q$ be an arbitrary subset, $\mk{c}$ be an invertible $\ca{O}$-ideal
 and let $x\in K$. For complex numbers $s$ such that $\Re(s)>1$ we define
\begin{align}\label{corbeau}
\zeta(\mk{c},\mk{f},x,W,s):=
\Norm(\mk{c})^{-s}\sum_{0\neq\mu\in(\mk{c}^{-1}\mk{f}+x)\cap W}
\frac{1}{|\Norm(\mu)|^s}.
\end{align} 
\end{Def}

For $\lambda\in K_+^{\times}$ and $\Re(s)>1$ one has the formula
\begin{align}\label{paon1}
\Norm(\lambda)^{-s}\zeta(\mk{c},\mk{f},x,W,s)=
\zeta(\lambda^{-1}\mk{c},\mk{f},\lambda x,\lambda W,s).
\end{align}

In general, for an arbitrary subset $W\subseteq Q$ the function $\zeta(\mk{c},\mk{f},x,W,s)$
will not admit a meromorphic continuation to all of $\CC$. However, there is an important
special case where it does namely in the case where $W$ is a Shintani cone.
\begin{Prop}\label{Shin_prol}(Shintani)
Let $C$ be a Shintani cone. Then the function $\zeta(\mk{c},\mk{f},x,C,s)$ admits a
meromorphic to all of $\CC$. Moreover, if $k\in\ZZ_{\geq 1}$, the special value
$\zeta(\mk{c},\mk{f},x,C,1-k)$ is a rational number.
\end{Prop}
\begin{proof}
This follows from Proposition 1 of \cite{Shin}.\fin
\end{proof}
Let $W=\bigcup_i C_i$ (disjoint union) be a Shintani domain. Since
\begin{align*}
\zeta(\mk{c},\mk{f},x,W,s)=\sum_{i}\zeta(\mk{c},\mk{f},x,C_i,s),
\end{align*}
we see that Proposition \ref{Shin_prol} continues to hold for such $W$.
In particular, Proposition \ref{Shin_prol} holds true with $W=\mathcal{D}$ where $\mathcal{D}$ is
any fundamental domain for the action of $\ca{O}(\mk{f}\infty)^{\times}$ on $Q$. 
\begin{Def}
We say that a meromorphic function $f(s)$ on the complex plane is a Shintani zeta function
(for the real quadratic field $K$) if there exist quantities $\mk{c},\mk{f}\subseteq K$, $x\in K$ and
a Shintani set $W$ such that 
$f(s)=\zeta(\mk{c},\mk{f},x,W,s)$.
\end{Def}
\subsection{$\QQ$-valued distributions on $\ca{O}_{K_p}$}
In this subsection, we want to define $\QQ$-valued distributions on $\ca{O}_K$ using
special values at $s=0$ of Shintani zeta functions.
Let $(p,f,N)$ be a triple as in the introduction and let 
$\mk{N}=\prod_{i=1}^r\eta_i^{e_i}$ where $\mk{N}\mk{N}^{\sigma}=N\ca{O}_K$.
For the rest of the section we fix an order $\ca{O}$ and we let $\mk{f}=f\ca{O}$.

From now on, for an arbitrary compact-open set $U\subseteq\ca{O}_{K_p}$ and an arbitrary subset $W\subseteq Q$ 
the notation $W\cap U$ is taken to mean $W\cap (U\cap K)$. 
The next proposition is an easy corollary of Proposition \ref{Shin_prol}.
\begin{Prop}
Let $W$ be a finite union of Shintani cones and 
let $U\subseteq\ca{O}_{K_p}$ be a compact open set. Then the function
\begin{align*}
\zeta(\mk{c},\mk{f},x,W\cap U,s),
\end{align*}
admits a meromorphic continuation to all of $\CC$ and its special values at negative
integers are rational numbers.
\end{Prop}
\begin{proof}
Note that $(\mk{c}^{-1}\mk{f}+x)\cap U$ may be written as
\begin{align*}
(\mk{c}^{-1}\mk{f}+x)\cap U=\bigcup_i (\mk{b}^{-1}\mk{f}+y_i),
\end{align*}
where the union is finite and disjoint, $\mk{b}$ is a suitable invertible $\ca{O}$-ideal 
and $y_i\in K$. It thus follows that
\begin{align*}
\zeta(\mk{c},\mk{f},x,W\cap U,s)=\sum_i\zeta(\mk{b},\mk{f},y_i,W,s).
\end{align*}
This concludes the proof. \fin
\end{proof}
\begin{Def}
Let $W\subseteq Q$ be an arbitrary subset and let  
\begin{align*}
\wt{\delta}=\sum_{\mk{d}|\mk{N},r\in\ZZ/f\ZZ}n(\mk{d},r)[\mk{d},r]\in\wt{\Div}_f(\mk{N}).
\end{align*}
For an invertible $\ca{O}$-ideal $\mk{c}$ and $\lambda\in K$ we define
\begin{align*}
\zeta_{\wt{\delta}}
(\mk{c},\mk{f},x,W,s):=
\sum_{\mk{d}|\mk{N},r\in\ZZ/f\ZZ}
n(\mk{d},r)\zeta(r\mk{c}(\mk{d}^{\sigma})^{-1},\mk{f},x,W,s).
\end{align*}
\end{Def}

\begin{Def}\label{grenier}
Let $\wt{\delta}\in \wt{\Div}_f(\mk{N})$ and let 
$W$ be a finite union of Shintani cones. We define a $\QQ$-valued distribution 
on $\ca{O}_{K_p}$ by the rule
\begin{align}\label{crapaud}
U\mapsto
\zeta_{\wt{\delta}}(\mk{c},\mk{f},x,W\cap U,0),
\end{align}
where $U$ is an arbitrary compact open set of $\ca{O}_{K_p}$.
\end{Def} 

\begin{Rem}
In the case where $\mk{c}$ is an integral $\ca{O}_K$ ideal coprime to $f\ca{O}_K$,  
$
\wt{\delta}=\prod_{i=1}^r(1-l_i[\eta_i])
$, $x=1$ and $\ca{D}$ is a Shintani set one has
\begin{align*}
\zeta_{\wt{\delta}}(\mk{c},\mk{f},x,\ca{D}\cap U,0)=\nu(\mk{c},\ca{D},U),
\end{align*}
where $\nu(\mk{c},\ca{D},\_)$ is the measure defined on the line $(20)$ of \cite{Das3}.
\end{Rem}
\begin{Prop}
Let $\lambda\in K_+^{\times}$, $x\in K$ and let $\mk{a}$ and $\mk{b}$ be 
two invertible $\ca{O}$-ideals.
Let $W$ be a Shintani domain.
Then the following two formulas hold:
\begin{enumerate}
\item[$(1)$] $\zeta_{\wt{\delta}}
(\mk{a}\mk{b},\mk{f}\mk{b},x,W,0)=\zeta_{\wt{\delta}}
(\mk{a},\mk{f},x,W)$,
\item[$(2)$] $\zeta_{\wt{\delta}}
(\mk{a},\mk{f},x,W\cap U,0)=\zeta_{\wt{\delta}}
(\lambda^{-1}\mk{a},\mk{f},\lambda x,\lambda (W\cap U),0)$.
\end{enumerate}
\end{Prop}
\begin{proof}
This a straightforward computation. \fin
\end{proof}
\begin{Prop}
Let $\mathcal{D}$ be a fundamental domain for the action of $\ca{O}_K(\mk{f}\infty)^{\times}$
on the positive quadrant $Q$ and let $\mk{c}$ be an invertible $\ca{O}$-ideal which is $\mk{f}$-int. 
Then we have
\begin{align}\label{one}
\zeta_{\wt{\delta}}(\mk{c},\mk{f},n_{\mk{c}},\mathcal{D}\cap\ca{O}_{K_p},0)=\zeta_{\wt{\delta}}(\mk{c},\mk{f}\infty,0),
\end{align}
where $\zeta_{\wt{\delta}}(\mk{c},\mk{f}\infty,s)$ is the zeta function which appears in Definition \ref{def2}
and $n_{\mk{c}}$ is an integer chosen as in Definition \ref{units}. 
Moreover, if $p\star\wt{\delta}=\wt{\delta}$ then
\begin{align}\label{two}
\zeta_{\wt{\delta}}
(\mk{c},\mk{f},n_{\mk{c}},\mathcal{D}\cap\ca{O}_{K_p}^{\times},0)=\zeta_{\wt{\delta},\{p\}}(\mk{c},\mk{f}\infty,0)=0.
\end{align}
\end{Prop}
The subscript $\{p\}$ of $\zeta_{\wt{\delta},\{p\}}$ means that one restricts the sum 
over elements coprime to $p$. 
\begin{proof}
The proof of \eqref{one} is straightforward and \eqref{two} follows from \eqref{fille}. \fin
\end{proof}
Let $T$ be the set of all places of $K$ which divide $\mk{N}$.
For technical reasons we will assume until the end of the section that $(\cond(\ca{O}),N)=1$.
Therefore if $\eta_0\in T=\{\eta|\mk{N}\}$ it gives rise to a discrete valuation on 
$\ca{O}$ which we denote by $v_{\eta_0}$. We can now state a variant of a key proposition that was proved in \cite{Das3}:
\begin{Prop}\label{tourtiere}
Let $C$ be a Shintani cone of dimension $m$ and assume that there exists 
$\eta\in T$ which is $\ca{O}$-good for $C$. Let $\mk{a}$ be an invertible $\ca{O}$-ideal
which is $\mk{f}$-int and such that $v_{\eta}(\mk{a})\leq 0$. Let $l=[\ca{O}:\eta]$ and let $y\in K$. Then
\begin{align*}
\zeta(\mk{a},\mk{f},y,C,0)-l\zeta(\mk{a}\eta^{-1},\mk{f},y,C,0)
\in\ZZ[\mbox{$\frac{1}{l}$}],
\end{align*}
where the denominator is at most $l^{m/(l-1)}$.
\end{Prop}
\begin{proof}
We follow closely a part of the proof of Proposition 6.1 of \cite{Das3}. 
Let $\eta'=\eta\cap\ca{O}$.
Since $\eta$ is $\ca{O}$-good for $C$ we may write 
$C=C(v_1,\ldots,v_r)$ with $v_i\in\ca{O}$ and $v_i\notin\eta'$. We claim that we can always 
find a positive integer $n_i\geq 1$ such that $n_iv_i\in\mk{a}^{-1}\mk{f}$ but
$n_iv_i\notin\mk{a}^{-1}\mk{f}\eta'$. Let us prove this. 
Since $v_{\eta}(\mk{a})\leq 0$, we may write 
$\mk{a}=\frac{\mk{m}}{\mk{n}}$ where $\mk{m}$ and
$\mk{n}$ are invertible integral $\ca{O}$-ideals and $v_{\eta}(\mk{m})=0$. Note that 
$\mk{a}^{-1}=\mk{n}\mk{m}^{-1}\mk{f}\supseteq\mk{n}\mk{f}$. Let $n_i$ be a positive integer
such that $n_iv_i\in\mk{n}\mk{f}$ and such that 
$v_{\eta}(n_i)=v_{\eta}(\mk{n})$ (such an integer $n_i$ exists since $v_i\notin\ca{O}$). 
Since $v_{\eta}(v_i)=v_{\eta}(\mk{m})=v_{\eta}(\mk{f})=0$ 
we see that $n_iv_i\notin\mk{a}^{-1}\mk{f}\eta'$. So this proves our claim. Therefore, 
without loss of generality, we may assume that 
$v_i\in\mk{a}^{-1}\mk{f}$ and $v_i\notin\mk{a}^{-1}\mk{f}\eta'$. 

Any element $\alpha\in C$ may be written 
uniquely as $\alpha=\sum_{i=1}^r (x_i+z_i)v_i$,
for real numbers $0<x_i\leq 1$ and non-negative integers $z_i$. Since 
$v_i\in\mk{a}^{-1}\mk{f}$, the element $\alpha$ lie in $\mk{a}^{-1}\mk{f}+y$ if and
only if $\sum x_iv_i$ does. Thus if we let
\begin{align*}
\Omega(\mk{c},y,v)=\{x\in\mk{c}+y:x=\sum x_iv_i\s\mbox{with}\s 0<x_i\leq 1\},
\end{align*}
then
\begin{align*}
\zeta(\mk{a},\mk{f},y,C,s)=\sum_{x\in\Omega(\mk{a}^{-1}\mk{f},y,v)}\sum_{z_1,\ldots,z_r=0}
\Norm\left(\sum (x_i+z_i)v_i\right)^{-s}.
\end{align*}
One has $\zeta(\mk{a},\mk{f},y,C,s)=Z(\mk{a}^{-1}\mk{f},y,C,s)$ where
\begin{align}\label{tornado}
Z(\mk{b},y,C,s)=\sum_{\alpha\in(\mk{b}+y)\cap C}\Norm(\alpha)^{-s},
\end{align}
is the zeta function which appears on line $(69)$ of \cite{Das3}.
From this point the rest of the argument is identical to the end of the 
proof of Proposition 6.1 of \cite{Das3}, so we skip it. \fin
\end{proof}
\subsection{$\QQ$-valued distributions on $\ZZ_p^2$}
In this subsection we want define $\QQ$-valued distributions on $\ZZ_p^2$ by using 
the distributions constructed in the previous section on $\ca{O}_{K_p}$.
For every pair of integers $(u,v)$ and for every non-negative integer $n$ we let
\begin{align}\label{gren}
U_{u,v,n}=\{(x,y)\in\ZZ_p^2:(x,y)\equiv(u,v)\pmod{p^n}\},
\end{align}
be the ball of radius $\frac{1}{p^n}$ centered at $(u,v)$. We denote the set of
all balls of $\ZZ_p^2$ by $\ca{B}$. We note that a distribution 
on $\ZZ_p^2$ is completely determined by its values on elements of $\ca{B}$ since the set $\ca{B}$ is
a basis for the topology of $\ZZ_p^2$. There is a right action of $GL_2(\ZZ_p)$ on $\ZZ_p^2$ which is
given by the following rule
\begin{align*}
\R{x}{y}\sharp\M{a}{b}{c}{d}=\R{dx-by}{-cx+ay},
\end{align*}
where $\M{a}{b}{c}{d}\in GL_2(\ZZ_p)$ and $\R{x}{y}\in\ZZ_p^2$. We have chosen this action in agreement
with \eqref{corbeille}. We thus obtain a right action 
of $GL_2(\ZZ_p)$ on $\ca{B}$ which is given on balls by the rule
\begin{align}\label{paon}
(U_{u,v,n})\sharp\;\gamma=U_{du-bv,-cu+av,n},
\end{align}
where $\gamma=\M{a}{b}{c}{d}\in GL_2(\ZZ_p)$ and $U_{u,v,n}\in\ca{B}$.
We note that the right action in \eqref{paon} may be written in terms of the usual left action of
$GL_2(\ZZ_p)$ on $\ZZ_p^2$ namely
\begin{align*}
(U_{u,v,n})\sharp\;\gamma=\gamma^{-1}U_{u,v,n},
\end{align*}
where for a matrix $\eta=\M{a}{b}{c}{d}\in GL_2(\ZZ_p)$, $\eta U_{u,v,n}:=U_{au+bv,cu+dv}$.

The space of $\QQ$-valued distributions on $\ZZ_p^2$, which we denote by $Dist(\ZZ_p^2,\QQ)$, has an 
induced right action by $GL_2(\ZZ_p)$.
For a distribution $\mu\in Dist(\ZZ_p^2,\QQ)$ and an element $\gamma\in GL_2(\ZZ_p)$ we define the 
\textit{right $\gamma$-twist} of $\mu$ by the rule 
\begin{align}\label{apple}
\mu^{\gamma}(U):=\mu(U\sharp\gamma^{-1})=\mu(\gamma U).
\end{align}

We also endow the space of rational binary quadratic forms with a left action of $GL_2(\QQ)$ by the rule
\begin{align}\label{apple2}
\gamma Q(x,y):=Q(\R{x}{y}\sharp\gamma)\s\s\s\mbox{for}\s\s\s\gamma=\M{a}{b}{c}{d}\in GL_2(\QQ).
\end{align}

We want to define a $\QQ$-valued distribution on $\ZZ_p^2$ using the distribution 
which appears in Definition \ref{grenier}. In order to do so we need to choose an identification of
$\ZZ_p^2$ with $\ca{O}_{K_p}$.
\begin{Def}
For every $\tau\in K\bs\QQ$, we define an injective map
\begin{align*}
\phi_{\tau}:\ZZ_p^2 &\rightarrow K_p\\
             (x,y)&\mapsto x-y\tau^{\sigma}.
\end{align*}
\end{Def}
Note that in general, the image $\phi_{\tau}$ is not necessarily equal to $\ca{O}_{K_p}$. 
We say that an element $\tau\in\ca{H}_p\cap K$ is \textit{reduced} if for $j=0,\ldots,p-1$, 
$|\tau-j|_p\geq 1$ and $\left|\frac{1}{\tau}\right|_p\geq 1$. When $\tau$ is reduced, it is easy to 
see that the map $\phi_\tau$ gives an isomorphism between $\ZZ_p^2$ and $\ca{O}_{K_p}$.

We note that
\begin{align}\label{simple2}
\phi_{\tau}(\gamma U)=(-c\tau^{\sigma}+a)\phi_{\gamma^{-1}\tau}(U).
\end{align}
We are now ready to define $\QQ$-valued distributions on $\XX$.
\begin{Def}\label{grenier2}
Let $\tau\in\ca{H}_p\cap K$ where $\tau$ is reduced and let $W$ be a Shintani set. 
Let $x\in K$ and let $\mk{a}$ be a fractional ideal. 
For every divisor $\wt{\delta}\in\wt{\Div}_f(\mk{N})$ we define the following 
$\QQ$-valued distribution on $\ZZ_p^2$:
\begin{align*}
U\mapsto\nu_{\wt{\delta}}(\mk{a},\tau,x,W)(U):=\zeta_{\wt{\delta}}(
\mk{a},\mk{f},x,W\cap\phi_{\tau}(U),0),
\end{align*}
for $U$ an arbitrary compact-open set of $\ZZ_p^2$.
\end{Def}
The next lemma will play a key role later on
\begin{Lemma}\label{Sonne}
Let $\tau\in K\bs\QQ$ and $x\in K^{\times}$. Suppose that $\gamma=\M{a}{b}{c}{d}\in GL_2(\ZZ)$ 
and that $-c\tau^{\sigma}+a\gg 0$. Then we have
\begin{align*}
\nu_{\wt{\delta}}^{\gamma}(I_{\tau},\tau,
x,W)=\nu_{\wt{\delta}}(I_{\gamma\tau},\gamma\tau,
(-c\tau^{\sigma}+a)^{-1}x,(-c\tau^{\sigma}+a)^{-1}W).
\end{align*}
\end{Lemma}
\begin{proof}
This is a straightforward computation. \fin
\end{proof}
\section{$\ZZ$-valued measures on $\XX$}
In \cite{Ch1}, a family of $\ZZ$-valued measures $\wt{\mu}_{\delta}$ 
(see Definition \ref{def_measure_balls} below) 
on $\XX$ was constructed using periods of Eisenstein series. Here $\XX=\ZZ_p\times\ZZ_p\bs(p\ZZ_p\times p\ZZ_p)$ 
denotes the set of primitive vectors of $\ZZ_p^2$.
The author defined
the $p$-invariant $u_C$ as a certain multiplicative integral on the space $\XX$ 
which involves the measure $\wt{\mu}_{\delta}$. In a similar way, the invariant $u_D$ is
defined as $p$-adic multiplicative integral on the space $\ca{O}_{K_p}$ 
which involves the measure $\nu_{\wt{\delta}}$
of Definition \ref{grenier2}. In Section \ref{def_invariants}, precise definitions
of $u_C$ and $u_D$ are given. 
The key ingredient that allows us to relate the invariant $u_C$ to the invariant
$u_D$ are explicit formulas of the measures $\wt{\mu}_{\delta}$ and $\nu_{\wt{\delta}}$
on balls of $\XX$.  

\begin{Def}\label{def_measure_balls}
Assume that $\mk{d}'$ is consecutive to $\mk{d}$  and let
$\eta=\frac{\mk{d}'}{\mk{d}}$, $l=\Norm(\eta)$, 
$d=\Norm(\mk{d})$ and $d'=ld$. Let us fix an integer $1\leq j\leq f$ and let
$$
\delta=\frac{N}{d}[d,j]-\frac{N}{d'}[d',j]\in \Div_f(N).
$$
Let $\infty=\frac{1}{0}$ be the standard cusp at infinity
and $\frac{a}{c}\in\Gamma_0(fN)\{\infty\}$.
For every ball $U_{u,v,s}\subseteq\XX$, we define
\begin{align}\label{libellule} \notag
&\wt{\mu}_{\delta}\left\{\infty\rightarrow\frac{a}{c}\right\}(U_{u,v,s})\\
&:=-12\left(\frac{N}{d}\right)
\sum_{\substack{1\leq h\leq\frac{c}{fd}}}
\wt{B}_1\left(\frac{a}{\frac{c}{fd}}\left(h+\frac{v}{p^s}+\frac{j}{f}\right)
-\frac{dfu}{p^s}\right)\wt{B}_1
\left(\frac{1}{\frac{c}{fd}}\left(h+\frac{v}{p^s}+\frac{j}{f}\right)\right)\\ \notag
&+12\left(\frac{N}{d'}\right)
\sum_{\substack{1\leq h'\leq\frac{c}{fd'}}}
\wt{B}_1\left(\frac{a}{\frac{c}{fd'}}\left(h'+\frac{v}{p^s}+\frac{j}{f}\right)
-\frac{d'fu}{p^s}\right)\wt{B}_1
\left(\frac{1}{\frac{c}{fd'}}\left(h'+\frac{v}{p^s}+\frac{j}{f}\right)\right),
\end{align}
where $\wt{B}_1(x)=\{x\}-\frac{1}{2}+\frac{\indicator_\ZZ(x)}{2}$, $0\leq\{x\}< 1$
denotes the fractional part of a real number $x$ and $\indicator_\ZZ(x)$ stands for the
characteristic function of $\ZZ$. 
\end{Def}
\begin{Rem}
The expression which appears on the right hand side of \eqref{libellule} is a special case of Dedekinds sums that
have been considered in \cite{Hal}.
\end{Rem}
Using Lemma \ref{tired}, we first extend by linearity, the definition of 
$\wt{\mu}_{\delta}\left\{\infty\rightarrow\frac{a}{c}\right\}$ to any divisor $\delta\in \Div^0(N)_j$.
In a second step, we extend the definition of 
$\wt{\mu}_{\delta}\left\{\infty\rightarrow\frac{a}{c}\right\}$ to all $\delta\in \Div_f(N)$ which satisfies
$(1)$ of Definition \ref{good_div}.

The next proposition justifies the previous definition.
\begin{Prop}\label{tyran}
Let $\delta\in \Div_f(N)$ be a good divisor for the triple $(p,f,N)$. Then the assignment $U_{u,v,s}\mapsto\wt{\mu}_{\delta}
\left\{\infty\rightarrow\frac{a}{c}\right\}(U_{u,v,s})$ gives rise to a $\ZZ$-valued measure
on $\XX$ with total measure zero, i.e., $\wt{\mu}_{\delta_r}\left\{\infty\rightarrow\frac{a}{c}\right\}(\XX)=0$.
\end{Prop}
\begin{proof}
See Proposition 14.1 in \cite{Ch}. \fin
\end{proof}
In the next subsection we will give a different proof of Proposition \ref{tyran} by relating the
measure $\wt{\mu}_{\delta}$ to the measure
$\nu_{\wt{\delta}}$.

It will be convenient to rewrite \eqref{libellule} in a different way. 
Every element $1\leq h\leq\frac{c}{fd}$ may be written uniquely as $h=\frac{c}{fd'}r+h'$
for some $0\leq r\leq l-1$ and $1\leq h'\leq\frac{c}{fd'}$. We have
\begin{align*}
&\sum_{\substack{1\leq h'\leq\frac{c}{fd'}}}
\wt{B}_1\left(\frac{a}{\frac{c}{fd'}}\left(h'+\frac{v}{p^s}+\frac{j}{f}\right)
-\frac{d'fu}{p^s}\right)\wt{B}_1
\left(\frac{1}{\frac{c}{fd'}}\left(h'+\frac{v}{p^s}+\frac{j}{f}\right)\right)\\
&=\sum_{\substack{1\leq h'\leq\frac{c}{fd'}}}\sum_{0\leq r\leq l-1}
\wt{B}_1\left(\frac{a}{\frac{c}{fd}}\left(h'+r\frac{c}{fd'}+\frac{v}{p^s}+\frac{j}{f}\right)
-\frac{dfu}{p^s}\right)\wt{B}_1
\left(\frac{1}{\frac{c}{fd'}}\left(h'+r\frac{c}{fd'}+\frac{v}{p^s}+\frac{j}{f}\right)\right)\\
&=\sum_{\substack{1\leq h\leq\frac{c}{fd}}}
\wt{B}_1\left(\frac{a}{\frac{c}{fd}}\left(h+\frac{v}{p^s}+\frac{j}{f}\right)
-\frac{dfu}{p^s}\right)\wt{B}_1
\left(\frac{1}{\frac{c}{fd'}}\left(h+\frac{v}{p^s}+\frac{j}{f}\right)\right)
\end{align*}
where the first equality follows from the distribution relation 
$\sum_{j=0}^{l-1}\wt{B}_1(x+\frac{j}{l})=\wt{B}_1(lx)$ and that
$\wt{B}_1(x+1)=\wt{B}_1(x)$. The second equality follows from the definition of $h$. 
We may thus rewrite \eqref{libellule} as
\begin{align}\label{libellule2} \notag
&\wt{\mu}_{\delta}\left\{\infty\rightarrow\frac{a}{c}\right\}(U_{u,v,s})
=-12\sum_{1\leq h\leq\frac{c}{fd}}
\wt{B}_1\left(\frac{a}{\frac{c}{fd}}\left(h+\frac{j}{f}+\frac{v}{p^s}\right)
-\frac{fdu}{p^s}\right)\\
&\hspace{4cm}\cdot\left[\left(\frac{N}{d}\right)\wt{B}_1
\left(\frac{1}{\frac{c}{fd}}\left(h+\frac{j}{f}+\frac{v}{p^s}\right)\right)
-\left(\frac{N}{d'}\right)\wt{B}_1\left(\frac{1}{\frac{c}{fd'}}\left(h+\frac{j}{f}+\frac{v}{p^s}\right)\right)\right].
\end{align}
\begin{Rem}\label{hausen}
Let $f=1$ and let $\delta\in Div_f(N)=Div_f(N)$ be a fixed good divisor for the triple
$(p,f,N)$. 
One may verify that the right hand side of \eqref{libellule2} coincides with the 
measure which appears in Proposition
3.2 of \cite{Das2}. In particular, since in this special case the $p$-adic measures
which are used to define $u_{DD}$ and $u_C$ agree one gets that
$u_{DD}(\alpha_{\delta},\tau)=u_{C,\delta}(1,\tau)$. We would like to point 
out here that the factor $12$ which appears in the right hand side of \ref{libellule}
is not optimal. It will follow from
Proposition \ref{Eleni} that under the additional assumption that $l\geq 5$ one has that 
$\frac{1}{6}\wt{\mu}_{\delta}\left\{\infty\rightarrow\frac{a}{c}\right\}(U_{u,v,s})$ is an 
integer. This fact was used by Dasgupta in \cite{Das3} in the following case: 
let $N=l$ be a prime and set $\delta=[l]-[1]$. Looking at the right hand side
of \eqref{libellule2} one may check that
$\frac{1}{6}\wt{\mu}_{\delta}\left\{\infty\rightarrow\frac{a}{c}\right\}(U_{u,v,s})$ 
coincides with the measure which appears on line $(82)$ of \cite{Das3}.
\end{Rem}
\subsection{Explicit formulas of $\nu_{\wt{\delta}}(\mk{b},\tau,x,C)$ on balls of $\XX$}
In this subsection we give explicit formulas of the measure $\nu_{\wt{\delta}}(\mk{b},\tau,x,C)$ 
(see Definition \ref{grenier2}) when evaluated on balls of $\XX$.
Let $(p,f,N)$ be a triple as in the introduction, $\mk{f}=f\ca{O}$ and let 
$\epsilon>1$ be a generator of $\ca{O}(\mk{f}\infty)^{\times}$.
\begin{Prop}\label{fish}
Assume that $\mk{d}'$ is consecutive to $\mk{d}$ and let
$\eta=\frac{\mk{d}'}{\mk{d}}$, $l=\Norm(\eta)$, 
$d=\Norm(\mk{d})$ and $d'=ld$. Set 
\begin{align*}
\wt{\delta}=d[\mk{d}]-d'[\mk{d}'].
\end{align*}
Note that $\Deg(\wt{\delta})=0$. 
Let $U_{u,v,n}$ be a ball contained in $\XX$ and assume that $u\equiv v\equiv 0\pmod{f}$. 
Let $x=x_1+x_2\tau$ with $x_i\in\ZZ$ and let $C=C(1,\epsilon)$ where $\epsilon>1$ is
a generator of $\ca{O}(\mk{f}\infty)^{\times}$. Let be $\tau\in K\bs\QQ$ and suppose 
$\frac{\tau}{d'}\in H^{\ca{O}}(\mk{N})$ and set
$\mk{b}=I_{\tau}$. 
Then we have the following formula:
\begin{align}\label{tempou}
&\nu_{\wt{\delta}}(\mk{b},\tau,x,C)
\left(U_{u,v,n}\right)
=\sum_{h\pmod{d'c}}B_1^*(y_1(h))[dB_1^*\left(ly_2(h)\right)-d'B_1^*\left(y_2(h)\right)],
\end{align}
where
\begin{align*}
y_1(h)&=\frac{a}{d'c}\left(h-\frac{v}{fp^n}+\frac{x_2}{f}\right)+
\frac{u}{fd'p^n}+\frac{x_1}{fd'}\\
y_2(h)&=\frac{-1}{d'c}\left(h-\frac{v}{fp^n}+\frac{x_2}{f}\right).
\end{align*}
Here $B_1^*(y)=y'-\frac{1}{2}$ where $y'$ is the unique real number $0<y'\leq 1$ such that
$y-y'\in\ZZ$.
\end{Prop}
\begin{proof}
We will prove Proposition \ref{fish} under the simplifying assumption that
$p\equiv 1\pmod{f}$ since it simplifies the presentation. 
It is easy to adapt the proof to the general case since it only amounts to
taking a \lq\lq trace\rq\rq\s over the subgroup $\laa p\pmod{f}\raa$ of $(\ZZ/f\ZZ)^{\times}$. Note that
\begin{align*}
\phi_{\tau}(U_{u,v,n})=u-v\tau+p^N\ca{O}_{K_p}.
\end{align*}
Using the assumptions that $p\equiv 1\pmod{f}$, $(\mk{b},p)=1$ and that $u\equiv v\equiv 0\pmod{f}$  
we readily see that
\begin{align*}
(\mk{b}^{-1}\mk{d}^{\sigma}f+x)\cap\phi_{\tau}(U_{u,v,n})=\mk{b}^{-1}\mk{d}^{\sigma}fp^n+u-v\tau+p^n(x_1+x_2\tau).
\end{align*}

We have $\mk{b}^{-1}=\Lambda_{\tau}$. Since $\frac{\tau}{d'}\in H^{\ca{O}}(\mk{N})$ 
we may deduce from \eqref{thrill} that
\begin{align}\label{moon}
\mk{b}^{-1}\mk{d}^{\sigma}=d\ZZ+\tau\ZZ\s\s\s\mbox{and}\s\s\s
\mk{b}^{-1}(\mk{d}')^{\sigma}=d'\ZZ+\tau\ZZ.
\end{align}
Thus 
\begin{align*}
\mk{b}^{-1}\mk{d}^{\sigma}p^nf=p^nfd\ZZ+p^nfN\tau\ZZ.
\end{align*}
Let $\eta_{\tau}=\M{a}{b}{c}{\wt{d}}\in\Gamma_1(f)\cap\Gamma_0(N)$. Make the
crucial observation that $(a,fN)=1$. Note that $1<\epsilon=c\tau+\wt{d}$ and
thus $\epsilon^{-1}=-c\tau+a$ where $0<\epsilon^{-1}<1$.

Let $m=\cond(\ca{O})=1$ where $(m,N)=1$. Set
\begin{align*}
w_1=d'fp^n\s\s\s\mbox{and}\s\s\s w_2=d'fp^n\epsilon^{-1}.
\end{align*} 
For $i\in\{1,2\}$ we have 
\begin{align*}
\frac{w_i}{l}\in\mk{b}^{-1}\mk{d}^{\sigma}p^nf\s\s\s\mbox{but}\s\s\s
\frac{w_i}{l}\notin\mk{b}^{-1}(\mk{d}')^{\sigma}p^nf, 
\end{align*}
as required. Unfolding the definition of 
$\nu_{\wt{\delta}}(\mk{b},\tau, x,C)\left(U_{u,v,n}\right)$ we find that
\begin{align}\label{sunny}\notag
&\nu_{\wt{\delta}}(\mk{b},\tau,x,C)\left(U_{u,v,n}\right)\\[3mm]
&=dZ(\mk{b}^{-1}\mk{d}fp^n,u-v\tau+p^nx,w,0)-d'Z(\mk{b}^{-1}(\mk{d}')^{\sigma}fp^n,u-v\tau+p^nx,w,0),
\end{align} 
where $Z(\mk{a},y,w,s)$ is the function which appears in \eqref{tornado}. 
Applying the formula which appears in $(77)$ of \cite{Das3} to the right hand side of \eqref{sunny} we obtain
\begin{align}\label{6temps}\notag
&\nu_{\wt{\delta}}(\mk{b},\tau,x,C)\left(U_{u,v,n}\right)\\
&=d\sum_{y\in\Omega(\mk{b}^{-1}\mk{d}^{\sigma}pf^n,u-v\tau+p^nx,w,0)}B_1(y_1)B_1(y_2)-
d'\sum_{y'\in\Omega(\mk{b}^{-1}(\mk{d}')^{\sigma}pf^n,u-v\tau+p^nx,w,0)}B_1(y_1')B_1(y_2').
\end{align}

Using \eqref{moon} we find that $\mk{b}^{-1}\mk{d}^{\sigma}fp^n$ is the set of elements of the form
\begin{align*}
hfp^n\tau+jfp^nd\s\s\mbox{for}\s\s h,j\in\ZZ.
\end{align*}

Now consider the system of linear equations (in the $y_i$'s) 
\begin{align*}
hp^nfN\tau+jfp^nd+u-v\tau+p^n(x_1+x_2\tau)=y_1(d'fp^n)+y_2d'fp^n\epsilon^{-1},
\end{align*}
Solving the system we find
\begin{align}\label{system}
y_2 &=\frac{-1}{d'c}\left(h-\frac{v}{fp^n}+\frac{x_2}{f}\right)\\ \notag
y_1 &=\frac{j}{l}+\frac{u}{fd'p^n}+\frac{x_1}{fd'}+\frac{a}{d'c}\left(h-\frac{v}{fp^n}+\frac{x_2}{f}\right).
\end{align}
For each residue class modulo $d'c$, there exists a unique integer $h$ in that class such 
that $0<y_2\leq 1$. For this fixed $h$, and each possible residue class modulo $l$, there exists
a unique integer $j$ in that class such that $0<y_1\leq 1$. Thus the first summation in \eqref{6temps} equals
\begin{align}\label{rose}
\sum_{j\pmod{l}}\sum_{h\pmod{d'c}}B_1^*\left(
\frac{a}{d'c}\left(h-\frac{v}{fp^n}+\frac{x_2}{f}\right)+
\frac{j}{l}+\frac{u}{fd'p^n}+\frac{x_1}{fd'}\right)B_1^*\left(
\frac{1}{d'c}\left(h-\frac{v}{fp^n}-\frac{x_2}{f}\right)\right).
\end{align}
Now every element $0\leq h\leq\frac{d'c}{N}-1$ can be written uniquely as $h=\frac{dc}{N}r+s$ with
$0\leq r\leq l-1$ and $0\leq s\leq\frac{dc}{N}-1$. Thus \eqref{rose} may be written as
\begin{align*}
\hspace{-2cm}\sum_{r,s}B_1^*\left(
\frac{r}{l}+\frac{s}{d'c}+\frac{1}{d'c}\left(-\frac{v}{fp^n}+\frac{x_2}{f}\right)\right)
\sum_{j\pmod{l}}B_1^*\left(\frac{ra}{l}+
\frac{as}{d'c}-\frac{a}{d'c}\left(-\frac{v}{fp^n}+\frac{x_2}{f}\right)+
\frac{j}{l}+\frac{u}{fd'p^n}+\frac{x_1}{fd'}\right).
\end{align*}
Using the distribution relation $\sum_{j=0}^{m-1}B_1^*(x+\frac{j}{m})=B_1^*(mx)$ for $m=l$ and the
fact that $B^*(x+1)=B^*(x)$ we may rewrite the expression above as
\begin{align}\label{rose2}
\sum_{h\pmod{d'c}}\wt{B}_1^*\left(
\frac{-1}{dc}\left(h-\frac{v}{fp^n}+\frac{x_2}{f}\right)\right)
\wt{B}_1^*\left(
\frac{a}{d'c}\left(h-\frac{v}{fp^n}+\frac{x_2}{f}\right)+
\frac{u}{fd'p^n}+\frac{x_1}{fd'}\right).
\end{align}
So this computes the first summation of \eqref{6temps}.
A similar calculation to what we did shows that 
the second summation in \eqref{6temps} is equal to 
\begin{align}\label{rose3}
\sum_{h\pmod{d'c}}B_1^*\left(
\frac{-1}{d'c}\left(h-\frac{v}{fp^n}+\frac{x_2}{f}\right)\right)
B_1^*\left(
\frac{a}{d'c}\left(h-\frac{v}{fp^n}+\frac{x_2}{f}\right)+
\frac{u}{fd'p^n}+\frac{x_1}{fd'}\right).
\end{align}
This shows \eqref{tempou} and therefore it concludes the proof. \fin
\end{proof}
\begin{Prop}\label{tricky2}
We use the same notation as in the previous proposition. Let
\begin{align*}
\wt{\delta}=d[\mk{d},j]-d'[\mk{d}',j]\s\s\s\mbox{and}\s\s\s\delta=\Phi(\wt{\delta})
=d\left[\frac{N}{d},j\right]-d'\left[\frac{N}{d'},j\right].
\end{align*}
Let $\tau\in H^{\ca{O}}(\mk{N})$ and set $\mk{b}=I_{fN\tau}$. Note that
$\frac{Nf\tau}{d'}\in H^{\ca{O}}(\mk{N})$.
Let $\gamma=\M{-f^2N^2}{0}{0}{-fN}$ and define
\begin{align}\label{bouton}
\xi(U_{u,v,n}):=12\nu_{\wt{\delta}}^{\gamma}(\mk{b},fN\tau,fN\tau,C)(U_{u,v,n})-
\wt{\mu}_{\delta}\left\{\infty\rightarrow\frac{a}{c}\right\}(U_{u,v,n}).
\end{align}
Then for every ball $U_{u,v,n}\subseteq\XX$ one has $\frac{1}{3}\xi(U_{u,v,n})\in\ZZ$ and
\begin{align}\label{terreur}
\xi(U_{u,v,n})+\xi(U_{-u,-v,n})=0.
\end{align}
\end{Prop}
\begin{proof}
Using \eqref{tempou} we find that 
\begin{align*}
&12\nu_{\wt{\delta}}^{\gamma}(\mk{b},fN\tau,fN\tau,C)(U_{u,v,n})\\
&=-12d'
\sum_{\substack{1\leq h\leq\frac{cd'}{fN}}}
B_1^*\left(\frac{a}{\frac{cd'}{fN}}\left(h+\frac{v}{p^s}+\frac{j}{f}\right)
-\frac{\frac{N}{d'}fu}{p^s}\right)B_1^*
\left(\frac{1}{\frac{cd'}{fN}}\left(h+\frac{v}{p^s}+\frac{j}{f}\right)\right)\\ \notag
&-12d
\sum_{\substack{1\leq h\leq\frac{cd}{fN}}}
B_1^*\left(\frac{a}{\frac{cd}{fN}}\left(h+\frac{v}{p^s}+\frac{j}{f}\right)
-\frac{\frac{N}{d}fu}{p^s}\right)B_1^*
\left(\frac{1}{\frac{cd}{fN}}\left(h+\frac{v}{p^s}+\frac{j}{f}\right)\right).
\end{align*}
Using \eqref{libellule} one may check that the right hand side of the equality above 
gives the same expression as
$\wt{\mu}_{\delta}\left\{\infty\rightarrow\frac{a}{c}\right\}(U_{u,v,n})$,
except that one evaluates $\wt{B}_1$ rather than $B_1^*$. Once this observation is made, 
the proof of \eqref{terreur} is elementary. The complete proof may be found in Appendix \ref{appendix2}. \fin 
\end{proof}

We can now give a new proof of Proposition \ref{tyran}
\begin{Prop}\label{Eleni}
Let
\begin{align*}
\wt{\delta}=d[\mk{d},j]-d'[\mk{d}',j]\s\s\s\mbox{and}\s\s\s\delta=\Phi(\wt{\delta})
=d\left[\frac{N}{d},j\right]-d'\left[\frac{N}{d'},j\right].
\end{align*}
Then $m_l\cdot\wt{\mu}_{\delta}\left\{\infty\rightarrow\frac{a}{c}\right\}$ 
is a $\ZZ$-valued measure where 
\begin{align*}
m_l=\left\{
\begin{array}{ccc}
4 & \mbox{if} & l=2 \\
3 & \mbox{if} & l=3\\
1 & \mbox{if} & l\geq 5.
\end{array}
\right.
\end{align*}
\end{Prop}
\begin{proof}
From Proposition \ref{tricky2} we see that it is enough to show that for every
ball $U_{u,v,n}\subseteq\XX$ one has
\begin{align*} 
m_l\cdot\nu_{\wt{\delta}}^{\gamma}(\mk{b},fN\tau,fN\tau,C)(U_{u,v,n})\in\ZZ.
\end{align*}
We may write $(\mk{b}^{-1}\mk{f}+z)$ in the following way:
\begin{align}\label{oeuf}
(\mk{b}^{-1}\mk{f}+z)\cap U_{u,v,n}=\bigcup_{i=1}^d (\mk{a}^{-1}+y_i),
\end{align} 
where $\mk{a}^{-1}=\mk{b}^{-1}\mk{f}\mk{p}^n$, $\mk{p}=p\ca{O}_K$ and $y_i\in K$. 
Since $(\mk{f}\mk{p},\eta)=\ca{O}$ and $v_{\eta}(\mk{b})\leq 0$ 
we again have $v_{\eta}(\mk{a})\leq 0$. From \eqref{oeuf} we get 
\begin{align}\label{baum}
\zeta(\mk{b},\mk{f},z,C\cap U_{u,v,n})=\sum_{i=1}^d\zeta(\mk{a},\mk{f},y_i,C,0).
\end{align}
From this one gets that
\begin{align*}
\zeta_{\wt{\delta}}(\mk{b},\mk{f},z,C\cap U_{u,v,n},0)=\sum_{i=1}^d
\zeta_{\wt{\delta}}(\mk{a},\mk{f},y_i,C,0).
\end{align*}
From Proposition \ref{tourtiere} we find that $\zeta_{\wt{\delta}}(\mk{a},\mk{f},y_i,C,0)\in\ZZ[\frac{1}{l}]$
with a denominator at most $l^{m/(l-1)}$. In particular, if $l>3$ we find that $\zeta_{\wt{\delta}}(\mk{a},\mk{f},y_i,C,0)\in\ZZ$.
If $l=3$ we have $\zeta_{\wt{\delta}}(\mk{a},\mk{f},y_i,C,0)\in\frac{1}{3}\ZZ$ and if $l=2$ we have
$\zeta_{\wt{\delta}}(\mk{a},\mk{f},y_i,C,0)\in\frac{1}{4}\ZZ$. This concludes the proof. \fin
\end{proof}

\section{The $p$-adic invariants $u_C$ and $u_D$}\label{def_invariants}

In this section we give the precise definitions of the $p$-adic invariants $u_C$ and
$u_D$. We also recall two conjectures which predict that $u_C$ and $u_D$ 
are strong $p$-units in an appropriate abelian extension of $K$. Finally, 
we state conjectural Shimura reciprocity laws which give an analytic description
of $G_{K^{ab}/K}$ on $u_C$ and $u_D$. For the rest of
this section we fix a triple $(p,f,N)$ and an order $\ca{O}$ of conductor coprime
to $N$. We let $\mk{N}\mk{N}^{\sigma}=N\ca{O}_K$ be a splitting of $N$ and we let $\mk{f}=f\ca{O}$.
Finally, we fix an element $\tau\in H^{\ca{O}}(\mk{N})$.
\subsection{Definition of $u_C$}\label{def_u_C}
Let $\XX=(\ZZ_p\times\ZZ_p)\bs(p\ZZ_p\times p\ZZ_p)$ be the set of primitive 
vectors of $\ZZ_p^2$. Assume furthermore 
that $\tau$ is reduced. Let $\delta\in \Div_f(N)$ be a good divisor and
let $(r,\tau)\in\ZZ/f\ZZ\times H^{\ca{O}}(\mk{N})$. 
\begin{Def}
Then the formula for the $p$-adic invariant 
$u_{C,\delta}(r,\tau)$ is given by
\begin{align}\label{flying}
u_{C,\delta}(r,\tau):=p^{3\zeta^*(\delta,(r,\tau),0)}\mint{\XX}{}(x-\tau y)
d\wt{\mu}_{\delta_r}\{\infty\rightarrow\eta_{\tau}\infty\}(x,y),
\end{align} 
where the measure $\wt{\mu}_{\delta_r}\{\infty\rightarrow\eta_{\tau}\infty\}$ is the one
which appears in Definition \ref{def_measure_balls}, the matrix $\eta_{\tau}$ is the 
one which appears in Definition \ref{matrix_def} and $\zeta^*(\delta,(r,\tau),s)$ is the
zeta function which appears in Definition \ref{def1}. 
\end{Def}
The multiplicative integral in \eqref{flying} indicates that one considers a limit of 
\lq\lq Riemann products\rq\rq rather than usual Riemann sums where the right hand side
of \eqref{flying} is taken to mean
\begin{align}\label{multiplicative product}
\mint{\XX}{}(x-\tau y)d\wt{\mu}_{\delta_r}\{\infty\rightarrow \gamma_{\tau}\infty\}(x,y):=
\lim_{||\mathcal{U}||\rightarrow 0}\prod_{U\in\mathcal{U}}(x_U-\tau y_U)^{\wt{\mu}_{\delta_r}
\{\infty\rightarrow \gamma_{\tau}\infty\}(U)}\in K_p^{\times},
\end{align}
where $\mathcal{U}$ is a cover of $\XX$ by disjoint compact open sets, $(x_U,y_U)$ is an arbitrary point of $U\in\mathcal{U}$, 
and the $p$-adic limit is taken over increasingly fine covers $\mathcal{U}$. 
The product in \eqref{multiplicative product}
makes sense since the measures $\wt{\mu}_{\delta_r}\{\infty\rightarrow\gamma\infty\}$ is $\ZZ$-valued 
and not only $\ZZ_p$-valued, thanks to Proposition \ref{Eleni}. The appellation of 
$p$-adic invariant for $u_C(r,\tau)$ is appropriate in the light of the following proposition:
\begin{Prop}
Assume that $(r,\tau)\sim_f(r',\tau')$ then 
\begin{align}\label{terrible}
u_{C,\delta}(r,\tau)=u_{C,\delta}(r',\tau') \pmod{\mu_{p^2-1}}.
\end{align}
\end{Prop}
\begin{proof}
See Theorem 1.2 of \cite{Ch1}. \fin.
\end{proof}
\begin{Rem}
As a consequence of relating the $p$-adic invariant $u_C$ to the
$p$-adic invariant $u_D$ we will show that \eqref{terrible} continues to
hold without the modulo $\mu_{p^2-1}$.
\end{Rem}
Using the fact that $\tau$ is reduced and taking $\ord_p$ on both sides of \eqref{flying} 
one finds that
\begin{align}\label{chico}
\ord_p(u_{C,\delta}(r,\tau))=3\zeta_p^*(\delta,(r,\tau),0).
\end{align}

We let $*$ be the involution on $H(\mk{N})$ given by $\tau\mapsto\tau^*=\frac{-1}{fN\tau}$.
We note that $*$ preserves the orientation of $\tau$ in the sense that if $\tau>\tau^{\sigma}$
then $\tau^*>(\tau^*)^{\sigma}$. However, 
the involution $*$ does not preserve the stratification \eqref{strat1} since
in general the order $\ca{O}_{\tau}$ will be different from the order $\ca{O}_{\tau^*}$.

Now let us define an ad\'elic action of $C_K$ on $u_C$.
\begin{Def}\label{add_action}
Let $(r,\tau)\in\ZZ/f\ZZ\times H^{\ca{O}}(\mk{N})$ and set 
$\mk{b}=\wt{r}I_{\tau}$. Note that $[\mk{b},\mk{b}\mk{N}]\in\mathcal{L}^{\ca{O}}(\mk{N})$. Let us 
denote the class of $[\mk{b},\mk{b}\mk{N}]$ in $\mathcal{L}^{\ca{O}}(\mk{N})/\sim_{f}$ by 
$[[\mk{b},\mk{b}\mk{N}]]$. Let
$c\in C_K$ and let $c\star[[\mk{b},\mk{b}\mk{N}]]=[[\mk{a},\mk{a}\mk{N}]]$ where $\star$ is the
action which appears in Definition \ref{atyran}. For
$c\in C_K$ we define
\begin{align*}
c\star u_{C}(r,\tau):=u_{D}(r',\tau'),
\end{align*}
where $(r',\tau')\in\ZZ/f\ZZ\times H^{\ca{O}}(\mk{N})$ is such that 
\begin{align}\label{oparleur}
\Theta(\mk{a})\wt{r}I_{\tau^*}\sim_f\wt{r}'I_{\tau'^*}.
\end{align}
Here $\Theta:I_{\ca{O}}(f)\rightarrow I_{\ca{O}^*}(f)$ is the map which appears in Proposition \ref{lilas}
where $\ca{O}^*=\ca{O}_{\tau^*}$.
\end{Def} 
We would like to point out here that the existence of a pair 
$(r',\tau')\in\ZZ/f\ZZ\times H^{\ca{O}}(\mk{N})$ such that
\eqref{oparleur} holds is guaranteed by Corollary \ref{fusee}.
\begin{Conj}\label{u_H_conj}
Let $[(r,\tau)]\in(\ZZ/f\ZZ\times H^{\ca{O}}(\mk{N}))/\sim_f$ and assume that $\tau$ is reduced.
Let $\ca{O}=\ca{O}_{\tau}$, $\mk{f}=f\ca{O}$, $\wp=p\ca{O}_K$
and $L=K(\mk{f}\infty)^{\laa\sigma_{\wp}\raa}$. Then 
$u_C(r,\tau)$ is a strong $p$-unit in $L$. Let
$rec:C_K\rightarrow G_{K^{ab}/K}$ be the reciprocity map given by class field
theory. Let $c\in C_K$ be 
an id\`ele of $K$ and let $rec(c)=\sigma\in \Gal(K^{ab}/K)$. Then
\begin{align*}
u_{C,\delta}(r,\tau)^{\sigma^{-1}}=c\star u_{C,\delta}(r',\tau'),
\end{align*}
where $\star$ is the action which appears in Definition \ref{add_action}.
\end{Conj}
\begin{Rem}
We would like to warn the reader that
the statement of the Shimura reciprocity law for $u_C$ as it appears 
in Conjecture 4.1 of \cite{Ch1} is wrong (unless $\frac{A_{\tau}}{A_{\tau'}}\equiv 1\pmod{f}$). 
For some additional information which motivate 
\eqref{oparleur} see \cite{erratum_1}.
\end{Rem}
\subsection{Definition of $u_{D}$}
We now want to describe Dasgupta's invariant $u_D$. Let
$S$ be the set of places that contains the infinite places of $K$, $\wp=p\ca{O}_K$ 
and the ones that ramify in $K(\mk{f}\infty)$. We let $R=S\bs\{\wp\}$ and 
$T=\{\eta|\mk{N}\}$.
Let $\wt{\delta}\in\wt{\Div}(\mk{N})$ be a good divisor for the triple $(p,f,N)$. 
Let us fix an order $\ca{O}$ of $K$ such that $(\cond(\ca{O}),N)=1$.
Let $\mk{f}\subseteq\ca{O}$ be an arbitrary $\ca{O}$-ideal (not necessarily invertible)
and let $\mk{b}\subseteq K$ be an arbitrary invertible $\ca{O}$-ideals which is
$\mk{f}$-int. We also let $\mathcal{D}$ be a Shintani domain for 
the action of $\Gamma_{\mk{b}}(\mk{f})$ (see Definition \ref{units}) on the positive quadrant $Q=\RR_{>0}^2$ and
make the crucial assumption that $T$ is $\ca{O}$-good for $\mathcal{D}$. 
\begin{Def}
Following \cite{Das3}, we define 
\begin{align}\label{square}
u_{D,\wt{\delta}}(\mk{b},\mathcal{D}):=p^{\zeta_{\wt{\delta}}(\mk{b},
\mk{f}\infty,0)}\mint{\ca{O}_{K_p}^{\times}}{}
z\cdot d\nu(\mk{b},\mathcal{D})(z),
\end{align}
where $\nu(\mk{b},\mathcal{D})(U)=\zeta_{\wt{\delta}}(\mk{b},\mk{f},n_{\mk{b}},\mathcal{D}\cap U,0)$ and
$\zeta_{\wt{\delta}}(\mk{b},\mk{f}\infty,s)$ is the zeta function which appears in Definition \ref{def2}.
For the definition of the multiplicative integral which appears in \eqref{square} see Section 3 of \cite{Das3}.
\end{Def}
\begin{Rem}
We note that \eqref{square} was defined in \cite{Das3} when $\ca{O}=\ca{O}_K$, 
$\mk{b}$ is an $\ca{O}_K$-ideal coprime to $\mk{f}=f\ca{O}_K$ and 
$$
\wt{\delta}=\sum_{\mk{d}|\mk{N},r\in\ZZ/f\ZZ}(-1)^{\wt{\mk{e}}(\mk{d})}
\Norm(\mk{d})[\mk{d},r]\in\wt{\Div}_f(\mk{N}),
$$ 
where $\mk{N}$ square-free. We note that in this special case one has 
$\Gamma_{\mk{b}}(\mk{f})=\ca{O}_K(\mk{f}\infty)^{\times}$.
\end{Rem}
The following key proposition was proved in \cite{Das3}.
\begin{Prop}
Let $\mathcal{D}$ and $\mathcal{D}'$ be two Shintani domains
for the action of $\ca{O}(\mk{f}\infty)^{\times}$ on $Q$. Assume that there exists
a prime $\eta|\mk{N}$ which is $\ca{O}$-good for $\mathcal{D}$ and $\mathcal{D}'$. Suppose also 
that $\sigma_\mk{b}=\sigma_{\mk{b}'}$. Then
\begin{align*}
u_{D,\wt{\delta}}(\mk{b},\mathcal{D})=u_{D,\wt{\delta}}(\mk{b}',\mathcal{D}').
\end{align*}
\end{Prop}
\begin{proof}
All the proofs in Section $5$ of \cite{Das3} carry over mutatis mutandis to our more general setting. \fin
\end{proof}
\begin{Def}\label{crocus}
Let $\mathcal{D}$ be a fundamental domain for the action of 
$\ca{O}(\mk{f}\infty)^{\times}$ on $Q$. Assume that there exists a prime
$\eta|\mk{N}$ which is $\ca{O}$-good for $\mathcal{D}$. We set
\begin{align*}
u_{D,\wt{\delta}}(\mk{b},\mk{f}):=u_{D,\wt{\delta}}(\mk{b},\mathcal{D}).
\end{align*}
We say that $u_{D,\wt{\delta}}(\mk{b},\mk{f})$ is primitive of conductor $[\ca{O},\mk{g}]$ 
if the pair $[\mk{b},\mk{f}]$ is primitive of conductor $[\ca{O},\mk{g}]$ (resp. non-primitive) 
if the pair $[\mk{b},\mk{f}]$ is primitive of conductor $[\ca{O},\mk{g}]$ (resp. non-primitive).
\end{Def}

Now let us define an ad\'elic action of $C_K$ on $u_D$.
\begin{Def}\label{ad_action}
Let $\mk{b}$ be an invertible $\ca{O}$-ideal which is $\mk{f}$-int and let
$c\in C_K$. Note that $[\mk{b},\mk{b}\mk{N}]\in\mathcal{L}^{\ca{O}}(\mk{N})$. Let us 
denote the class of $[\mk{b},\mk{b}\mk{N}]$ in $\mathcal{L}^{\ca{O}}(\mk{N})/\sim_{\mk{f}}$ by 
$[[\mk{b},\mk{b}\mk{N}]]$ and let $c\star[[\mk{b},\mk{b}\mk{N}]]=[[\mk{a},\mk{a}\mk{N}]]$ where $\star$ is the
action which appears in Definition \ref{atyran}. We define
\begin{align*}
c\star u_{D}(\mk{b},\mk{f}):=u_{D}(\mk{a},\mk{f}).
\end{align*}
\end{Def} 

We may now state the main algebraicity conjecture on the $p$-adic invariant $u_D$ and the its
corresponding Shimura reciprocity formula.
\begin{Conj}\label{alg_conj}
Let $\mk{b}$ be an invertible $\ca{O}$-ideal which is $\mk{f}$-int.
Let $\mk{n}$ be an invertible $\ca{O}$-ideal such that
\begin{align*}
\Gamma_{\mk{b}}(\mk{f})\supseteq\ca{O}(\mk{n}\infty)^{\times}.
\end{align*} 
Then $u_{D,\wt{\delta}}(\mk{b},\mk{f})$ is a strong $p$-unit in 
$K(\mk{n}\infty)^{\laa\sigma_{\wp}\raa}$. Let
$rec:C_K\rightarrow G_{K^{ab}/K}$ be the reciprocity map given by class field
theory. Let $c\in C_K$ be 
an id\`ele of $K$ and let $rec(c)=\sigma\in \Gal(K^{ab}/K)$. Then 
\begin{align*}
u_{D,\wt{\delta}}(\mk{b},\mk{f})^{\sigma^{-1}}=c\star u_{D,\wt{\delta}}(\mk{b},\mk{f}).
\end{align*}
\end{Conj}
\begin{Rem}
If $\mk{b}$ is an integral invertible $\ca{O}$-ideal then one has
that $\Gamma_{\mk{b}}(\mk{f})\supseteq \ca{O}(\mk{f}\infty)^{\times}$
and therefore, Conjecture \ref{alg_conj} predicts that $u_{D,\wt{\delta}}
(\mk{b},\mk{f})\in K(\mk{f}\infty)^{\laa\sigma_{\wp}\raa}$. Now let us assume that 
the pair $[\mk{b},\mk{f}]$ is primitive of conductor $[\ca{O},\mk{g}]$. In this case,
there exists a $\lambda\in K^{\times}$ such that $(\lambda\mk{b},\lambda\mk{f})=\ca{O}$
and $\lambda\mk{f}=\mk{g}$ so that
\begin{align*}
\Gamma_{\mk{b}}(\mk{f})=\Gamma_{\lambda\mk{b}}(\mk{g})\supseteq\ca{O}(\mk{g}\infty)^{\times}.
\end{align*}
Therefore, Conjecture \ref{alg_conj} predicts that
\begin{align*}
u_{D,\wt{\delta}}(\mk{b},\mk{f})\in K(\mk{g}\infty)^{\laa\sigma_{\wp}\raa}.
\end{align*}
\end{Rem}
\begin{Rem}\label{axa}
In general, the author thinks that the following strengthening on the field of definition of 
$u_D(\mk{b},\mk{f})$ should hold true. So let $\mk{b}$ be an invertible $\ca{O}$-ideal which is $\mk{f}$-int and 
suppose that
\begin{align*}
\Gamma_{\mk{b}}(\mk{f})\supseteq\ca{O}'(\mk{n}'\infty)^{\times},
\end{align*}
where $\ca{O}'$ is some order and $\mk{n}'$ is an invertible $\ca{O}'$-ideal. Then the author expects that
$u_D(\mk{b},\mk{f})\in K(\mk{n}')$.
\end{Rem}
\subsection{Relation between $u_C$ and $u_D$}
The next proposition gives a relationship between the invariant $u_C$ to the invariant $u_D$.
\begin{Prop}\label{klasse}
Let $(r,\tau)\in\ZZ/f\ZZ\times H^{\ca{O}}(\mk{N})$ and let  $\delta=\Phi(\wt{\delta})$ be a good divisor for the
triple $(p,f,N)$. Then we have
\begin{align}\label{talle}
u_{C,\delta}(r,\tau)^{\nu(r,\tau)}=u_{D,\wt{\delta}}(\wt{r}I_{\tau^*},\mk{f}^*)^{12},
\end{align}
where $\mk{f}^*=f\ca{O}_{\tau^*}$ and 
\begin{align}\label{creux}
\nu(r,\tau):=\left([\epsilon(\eta_{\tau}):\epsilon(\eta_{\tau^*})]\mu(r,\tau^*)\right)^{-1},
\end{align}
where $\mu(r,\tau^*)$ is the rational number which appears in \eqref{sang}.
Moreover, the Shimura reciprocity law for $u_C$ agrees with the one for $u_D$.
\end{Prop}
\begin{proof}
From Lemma \ref{tired} it is enough to prove \eqref{talle} for 
\begin{align*}
\wt{\delta}=(d[\mk{d}]-d'[\mk{d}'])\s\s\s\mbox{where}\s\s\s d=\Norm(\mk{d})\s\s\s\mbox{and}\s\s\s d'=\Norm(\mk{d}'),
\end{align*}
where $\mk{d}'$ consecutive to $\mk{d}$. We have 
\begin{align*}
\zeta^*(\delta,(r,\tau),0)=[\epsilon(\eta_{\tau}):\epsilon(\eta_{\tau^*})]
\zeta(\delta^*,(r,\tau^*),0)=4\mu(r,\tau^*)
\zeta_{\wt{\delta}}(\wt{r}A_{\tau^*}\Lambda_{\tau^*},\mk{f}\infty,0),
\end{align*}
where the first equality follows from Lemma \ref{efkaristo} and the second equality follows 
from Lemma \ref{lievre}. Using \eqref{chico}
with the equality above we find that
\begin{align*}
\nu(r,\tau)\cdot v_p(u_{C,\delta}(r,\tau))=12\cdot v_p(u_{D,\wt{\delta}}(\wt{r}A_{\tau^*}
\Lambda_{\tau^*},\mathcal{D})).
\end{align*} 
This proves that the $p$-adic valuation on both sides of \eqref{talle} are equal. Therefore it
remains to show that the multiplicative integrals on both sides of \eqref{talle} agree.

Let $\mk{b}=\wt{r}I_{\tau}$ and $\gamma=\M{-f^2N}{0}{0}{-fN}$. For $C'=C(1)$ or $C'=C(\epsilon)$, we have
\begin{align}\label{fleur2}
\nu_{\wt{\delta}}^{\gamma}(\mk{b},fN\tau,fN\tau,C')(-U)=-
\nu_{\wt{\delta}}^{\gamma}(\mk{b},fN\tau,fN\tau,C')(U).
\end{align}
From \eqref{fleur2}, we get that
\begin{align*}
\mint{\XX}{}(x-fN\tau y)d\nu_{\wt{\delta}}^{\gamma}(\mk{b},fN\tau,fN\tau,C(1))=
\mint{\XX}{}(x-\tau y)d\nu_{\wt{\delta}}^{\gamma}(\mk{b},fN\tau,fN\tau,C(\epsilon))=1.
\end{align*}
Thus one has
\begin{align}\label{fleur}
\mint{\XX}{}(x-fN\tau y)d\nu_{\wt{\delta}}^{\gamma}(\mk{b},fN\tau,fN\tau,C(\epsilon,1))=\mint{\XX}{}(x-fN\tau y)
d\nu_{\wt{\delta}}^{\gamma}(\mk{b},fN\tau,fN\tau,\mathcal{D}).
\end{align}
Let $\xi$ (which depends of $\wt{\delta}$) 
be the measure which appears in Proposition \ref{tricky2}. From \eqref{terreur} we deduce that
\begin{align}\label{land1}
\mint{\XX}{}\xi(x,y)=1.
\end{align} 
Moreover, because $\nu_{\wt{\delta}}^{\gamma}(\mk{b},fN\tau,fN\tau,C)$ has total measure $0$ we have
\begin{align}\label{land2}
\mint{\XX}{}(x-fN\tau y)d\nu_{\wt{\delta}}(\mk{b},fN\tau,fN\tau,C(1,\epsilon))(fx,y)=
\mint{\XX}{}(x-fN\tau y)d\nu_{\wt{\delta}}^{\gamma}(\mk{b},fN\tau,fN\tau,C(1,\epsilon))(x,y).
\end{align}

Combining \eqref{land1} and \eqref{land2} and \eqref{bouton} we find that
\begin{align*}
\left(\mint{\XX}{}(x-fN\tau y)d\nu_{\wt{\delta}}(\mk{b},fN\tau,fN\tau,C(\epsilon,1))(x,y)\right)^{12}=
\left(\mint{\XX}{}(x-\tau y)d\wt{\mu_r}\{\infty\rightarrow\eta_{\tau}\infty\}(x,y)\right)^{\nu(r,\tau)}.
\end{align*}
Thus in order to show the multiplicative parts on both sides of \eqref{talle} agree,
it is enough to show that
\begin{align}\label{pays}
\mint{\XX}{}(x-fN\tau y)d\nu_{\wt{\delta}}(\mk{b},fN\tau,fN\tau,C(\epsilon,1))(x,y)=
\mint{\XX}{}(x-\tau^* y)d\nu_{\wt{\delta}}(\mk{b}^*,\tau^*,1,C(\epsilon,1))(x,y),
\end{align}
where $\mk{b}^{*}=I_{\tau^*}$ and $\tau^*=\frac{-1}{fN\tau}$.

Applying Lemma \ref{Sonne} with the matrix $\gamma=\M{0}{1}{1}{0}$ we find that
\begin{align*}
\nu_{\wt{\delta}}^{\gamma}\left(\mk{b}^*,\tau^*,1,\frac{1}{fN_0\tau}C(\epsilon,1)\right)
=\nu_{\wt{\delta}}(\mk{b},fN\tau,fN\tau,C(\epsilon,1)).
\end{align*}
It thus follows that
\begin{align*}
\mint{\XX}{}(x-fN\tau y)d\nu_{\wt{\delta}}(\mk{b},fN\tau,fN\tau,C(\epsilon,1))(x,y)&=
\mint{\XX}{}(x-\tau^* y)d\nu_{\wt{\delta}}\left(\mk{b}^*,\tau^*,1,\frac{1}{fN_0}C(\epsilon,1)\right)(y,fNx)\\
&=\mint{\XX}{}(x-\tau^* y)d\nu_{\wt{\delta}}(\mk{b}^*,\tau^*,1,C(\epsilon,1))(x,y),
\end{align*}
where the last equality follows from the independence on the choice of the fundamental domain 
and the fact that $\nu_{\wt{\delta}}\left(\mk{b}^*,\tau^*,1,\frac{1}{fN_0}C(\epsilon,1)\right)(y,fNx)$
has total measure zero. This shows \eqref{pays}.

Finally, the fact that the two Shimura reciprocity laws agree follows directly from their 
definitions. This concludes the proof. \fin
\end{proof}
\subsection{From $K(\mk{f}^*\infty)$ to $K(\mk{f}\infty)$}\label{down}
In this subsection we would like to show the equivalence between 
Conjecture \ref{u_H_conj} and Conjecture \ref{alg_conj} when restricted to the
$p$-adic invariant $u_C(r,\tau)$ for pairs $(r,\tau)\in\ZZ/f\ZZ\times H^{\ca{O}}(\mk{N},f)$
(we only impose the restriction $(A_{\tau},f)=1$ in order to simplify the presentation).
It follows from Proposition 9.3 that one has
\begin{align}\label{tenemos}
u_{C}(r,\tau)^{\nu(r,\tau)}=u_{D}(\wt{r}I_{\tau^*},\mk{f}^*)^{12},
\end{align}
and that the Shimura reciprocity law of $u_C$ is compatible with the Shimura reciprocity
law of $u_D$. Therefore the equivalence between the two conjectures will follow if we can
show that the field of definition for $u_C(r,\tau)$ which is predicted by Conjecture \ref{u_H_conj}
agrees with the one predicted by Conjecture \ref{alg_conj}.

Let $(r,\tau)\in\ZZ/f\ZZ\times H^{\ca{O}}(\mk{N})$. 
Then Conjecture \ref{alg_conj} predicts that $u_{D}(\wt{r}I_{\tau^*},\mk{f}^*)$ is a strong
$p$-unit in $K(\mk{f}^*\infty)^{\laa\sigma_{\wp}\raa}$ where $\mk{f}^*=f\ca{O}_{\tau^*}$. Now assume that
$\tau\in H^{\ca{O}}(\mk{N},f)$ so that $\cond(\ca{O}_{\tau^*})=f\cdot \cond(\ca{O})$ and that
$r\in(\ZZ/f\ZZ)^{\times}$. Under these assumptions we have that $\nu(r,\tau)=1$ and that 
$K(\mk{f}\infty)\subseteq K(\mk{f}^*\infty)\subseteq K(\mk{f}^2\infty)$. 
Conjecture \ref{u_H_conj}, when applied to the left hand side of \eqref{tenemos}, predicts that
\begin{align}\label{trompe}
u_{C}(r,\tau)\in K(\mk{f}\infty)^{\laa\sigma_{\wp}\raa}.
\end{align} 
On the other hand, Conjecture \ref{alg_conj}, when applied to the right hand side of \eqref{tenemos}, only 
predicts, a priori, that
\begin{align}\label{trompe2}
u_C(r,\tau)=u_{D}(\wt{r}I_{\tau^*},\mk{f}^*)^{12}\in K(\mk{f}^*\infty)^{\laa\sigma_{\wp}\raa}.
\end{align}
In general $K(\mk{f}\infty)\subsetneqq K(\mk{f}^*\infty)$. Therefore
it seems that Conjecture \ref{u_H_conj}
gives a finer statement than Conjecture \ref{alg_conj}. In the remainder of the
section we prove the following:
\begin{Prop}\label{klam}
Conjecture \ref{alg_conj} $\Longrightarrow$ \eqref{trompe}.
\end{Prop}

Let us explain our strategy of proof.
The inclusions $K(\mk{f}\infty)\subseteq K(\mk{f}^*\infty)\subseteq K(\mk{f}^2\infty)$ give 
rise to the two restriction maps 
\begin{align}\label{tet_1}
res_1&:\Gal(K(\mk{f}^2\infty)/K)
\rightarrow \Gal(K(\mk{f}^*\infty)/K),\\ \notag
res_2&:\Gal(K(\mk{f}^2\infty)/K)
\rightarrow \Gal(K(\mk{f}\infty)/K),
\end{align}
which correspond to the two projections
\begin{align}\label{tet_2}
\wt{\Theta}_1&:I_{\ca{O}}(f)/P_{\ca{O},1}(f^2\infty)\rightarrow I_{\ca{O}^*}(f)/P_{\ca{O}^*,1}(f\infty),\\ \notag
\wt{\Theta}_2&:I_{\ca{O}}(f)/P_{\ca{O},1}(f^2\infty)\rightarrow I_{\ca{O}}(f)/P_{\ca{O},1}(f\infty).
\end{align}
When $\tau\in H^{\ca{O}}(\mk{N},f)$, 
we will show that 
$\left(P_{\ca{O},1}(f\infty)\cap I_{\ca{O}}(f)\pmod{\sim_f}\right)=ker_2(\wt{\Theta}_2)$ 
acts trivially on the class $[\wt{r}I_{\tau^*}]$ and therefore, 
from Conjecture \ref{alg_conj}, this will mean that the action of 
$\Gal(K(\mk{f}^2\infty)/K)$ on $u_D(\wt{r}I_{\tau^*},\mk{f}^*)$ factors through 
$\Gal(K(\mk{f}\infty)/K)$. Having this additional information, we thus see a postiori that  
Conjecture \ref{alg_conj} implies the finer statement \eqref{trompe}.

We want to prove the following:
\begin{Prop}\label{carotte}
Let $\ca{O}$ be an order of conductor coprime to $f$ and let 
$(r,\tau)\in(\ZZ/f\ZZ)^{\times}\times H^{\ca{O}}(\mk{N},f)$. Denote the class
of $(r,\tau)$ modulo $\sim_f$ by $[(r,\tau)]$. 
Assume that $\{(r_i,\tau_i)\}_{i\in I}$ 
is a complete set of representatives for the action of $C_{\ca{O}}(f)$ on $[(r,\tau)]$. 
Let $\ca{O}^*=\ca{O}_{\tau^*}$. Then
$\{(r_i,\tau_i^*)\}$ is a complete set of representatives for the action of $C_{\ca{O}^*}(f)$ on
$[(r,\tau^*)]$. Moreover, the action of $C_{\ca{O}^*}(f)$ on $[(r,\tau^*)]$ factors through 
$C_{\ca{O}}(f)$.
\end{Prop}

In order to prove the previous proposition we need to prove a lemma.
\begin{Lemma}\label{tricky}
Let $(r,\tau),(r',\tau')\in\ZZ/f\ZZ\times H^{\ca{O}}(\mk{N})$. Then we
have that $(r,\tau)\sim_f(r',\tau')$ if and only if $(r',\tau^*)\sim_f(r,\tau'^*)$.
\end{Lemma}
\begin{proof}
First note that since $(\wt{r},f)=(\wt{r}',f)$ and that $1\leq\wt{r},\wt{r}'\leq f$ one has
that $\frac{\wt{r}}{\wt{r}'}$ is a unit modulo $f$. 
From the Proposition \ref{hard_prop} we know that there exists a matrix
$\gamma=\M{a}{b}{c}{d}\in\Gamma_0(fN)$ such that $\gamma\tau=\tau'$ and 
$d^{-1}r\equiv r'\pmod{f}$. Therefore we have
\begin{align}\label{flirt}
\M{a}{b}{c}{d}\V{\tau}{1}=(c\tau+d)\V{\tau'}{1}.
\end{align}
We can rewrite \eqref{flirt} in the following way
\begin{align}\label{bomb}
\M{a}{-bfN}{-\frac{c}{fN}}{d}\V{1}{\tau^*}=(c\tau+d)\frac{\tau'}{\tau}\V{1}{{\tau'}^*}.
\end{align}
It thus follows that $a+bf\tau^*=(c\tau+d)\frac{\tau'}{\tau}$ and we can thus rewrite
\eqref{bomb} as 
\begin{align}\label{anchois}
\M{d}{-\frac{c}{fN}}{-bfN}{a}\V{\tau^*}{1}=(-bfN\tau^*+a)\V{\tau'^*}{1}.
\end{align}

From \eqref{anchois} we may deduce that
\begin{align*}
\frac{\wt{r}}{\wt{r}'}(bfN(\tau^*)^{\sigma}+a)\wt{r}'A_{\tau^*}\Lambda_{\tau^*}=\wt{r}A_{\tau'^*}\Lambda_{\tau'^*}.
\end{align*}
Since $\M{d}{-\frac{c}{fN}}{-bfN}{a}\in\Gamma_0(fN)$ and $\frac{dr'}{r}\equiv 1\pmod{f}$ 
we have
\begin{align*}
\frac{\wt{r}}{\wt{r}'}(-bfN(\tau^*)^{\sigma}+a)\in(\wt{r}'A_{\tau^*}\Lambda_{\tau^*})^{-1}f+1.
\end{align*}
It thus follows that $\wt{r}'A_{\tau^*}\Lambda_{\tau^*}\sim_f\wt{r}A_{\tau'^*}\Lambda_{\tau'^*}$ 
and therefore from Proposition \ref{hard_prop} we find that $(r',\tau^*)\sim_f(r,\tau'^*)$. \fin
\end{proof}

Let us record the following useful corollary
\begin{Cor}\label{fusee}
Let $[(r,\tau)]\in(\ZZ/f\ZZ)^{\times}\times H^{\ca{O}}(\mk{N})/\sim_{f}$ and let 
\begin{align*}
C_{\ca{O}}(f)\cdot[(r,\tau)]=\{[(r_i,\tau_i)]\}_{i\in I}.
\end{align*}
be the orbit of $[(r,\tau)]$ under the action of $C_{\ca{O}}(f)$. Then
\begin{align*}
C_{\ca{O}^*}(f)\cdot[(r,\tau)]=\{[(r_i,\tau_i^*)]\}_{i\in I}.
\end{align*}
\end{Cor}
\begin{proof}
Without lost of generality, we may assume that the orbit $C_{\ca{O}}(f)\cdot[(r,\tau)]$ is written as
\begin{align*}
\{[(r_{ij},\tau_i)]\}_{i\in I,j\in J},
\end{align*}
so that if $i,i'\in I$ are such that $(*,\tau_i)\sim_f(*,\tau_{i'})$ then $i=i'$.
Assume that for $i\neq i'$ and some $j,j'\in J$ one has that 
$(r_{ij},\tau_i^*)\sim_f(r_{i'j'},\tau_{i'}^*)$. First note that
$\frac{r_{ij}}{r_{i'j'}}:=s$ is a unit modulo $f$. From Lemma \ref{tricky} we find that
$(r_{i'j'},\tau_i)\sim_f(r_{ij},\tau_{i'})$. Multiplying the first coordinate by $s$ we find that 
$(r_{ij},\tau_i)\sim_f(s r_{ij},\tau_{i'})$. In particular, this implies that $i=i'$ and that there
exists $\epsilon\in\ca{O}(\infty)^{\times}$ such that $\epsilon\equiv s\pmod{f}$. Now note that
$(r_{ij},\tau_i)\sim_f(s r_{ij},\tau_{i'})\sim_f(s^2 r_{i'j'},\tau_{i'})\sim_f(r_{i'j'},\tau_{i'})$.
Therefore we must have $r_{ij}=r_{i'j'}$. We thus conclude that the elements in the set 
$\{(r_{ij},\tau_i^{*})\}_{i\in I,j\in J}$ are inequivalent modulo $\sim_f$. 
Reversing this argument we see that this set gives a complete list of representatives
modulo $\sim_f$. This concludes the proof. \fin
\end{proof}

\noindent {\bf Proof of Proposition \ref{carotte}} Let $n=\cond(\ca{O})$. Note that if $\tau\in H^{\ca{O}}(\mk{N},f)$ then
$Q_{\tau}(x,y)=Ax^2+Bxy+Cy^2$ with $(A,f)=1$ and $B^2-4AC=n^2d_K$. Let $[(r,\tau)]$ be the equivalence class
of $(r,\tau)$ modulo $\sim_f$. Since $(A,f)=1$ and $N|A$ we have
$Q_{\tau^*}(x,y)=\sign(C)(Cf^2Nx^2-Bfxy+\frac{A}{N}y^2)$ so that $\cond(\ca{O}_{\tau^*})=fn$.
Let $\ca{O}^{*}=\ca{O}_{\tau^*}$ so that $\tau^*\in H^{\ca{O}^*}(\mk{N})$.
We let $\Theta_1,\Theta_2$ be as in \eqref{tet_1} and 
$\wt{\Theta}_1, \wt{\Theta}_2$ be as in \eqref{tet_2}.

We note that for 
$\lambda\ca{O}\in P_{\ca{O},1}(f\infty)$ one has that $\Theta_1(\lambda\ca{O})=\lambda\ca{O}^*$ 
(we leave the proof of this fact as an exercise).
We claim that $\wt{\Theta}_1(ker_2(\wt{\Theta}_2))$ \lq\lq acts trivially\rq\rq\s on the class $[(r,\tau^*)]$.
By acting trivially we mean that if 
$\lambda\ca{O}\in P_{\ca{O},1}(f\infty)$ is such that $\lambda\wt{r}I_{\tau^*}\in I_{\ca{O}^*}(f)$ then
$[\lambda\wt{r}I_{\tau^*}]=[\wt{r}I_{\tau^*}]$ where the brackets mean the class modulo
$\sim_f$. Let $\lambda\ca{O}\in P_{\ca{O},1}(f\infty)$ and assume that
$\lambda\in(\wt{r}I_{\tau^*})^{-1}f+1$. In particular one has that
$\lambda\in Q_{\ca{O},1}(f\infty)\cap(\wt{r}I_{\tau^*})^{-1}$. We have that
\begin{align}\label{huile1}
(\wt{r}I_{\tau^*})^{-1}f+1=\frac{1}{\wt{r}}\Lambda_{(\tau^*)^{\sigma}}f+1=
\frac{1}{\wt{r}}\left(f\ZZ+\left(\frac{-B-n\sqrt{D}}{NC}\right)\ZZ\right)+1,
\end{align}
and a direct computation shows that 
\begin{align}\label{huile2}
Q_{\ca{O},1}(f\infty)\cap(\wt{r}I_{\tau^*})^{-1}\subseteq\frac{1}{\wt{r}}
\left(f\ZZ+\left(\frac{-B-n\sqrt{D}}{NC}\right)\ZZ\right)+1.
\end{align}
Combining \eqref{huile1} and \eqref{huile2} we find that 
$\lambda\in Q_{\ca{O},1}(f\infty)\cap(\wt{r}I_{\tau^*})^{-1}\subseteq(\wt{r}I_{\tau^*})^{-1}f+1$.
Therefore $\wt{\Theta}_1(ker(\wt{\Theta}_2))$ acts trivially on $[(r,\tau^*)]$ as claimed.

Let $\{\mk{a}_i\}_{i\in I}$ be a complete set of 
representatives of $I_{\ca{O}}(f)$ modulo $P_{\ca{O},1}(f\infty)$ and let
$\psi([\mk{a}_i])=[(r_i,\tau_i)]$. Note that 
$(r_i,\tau_i)\in(\ZZ/f\ZZ)^{\times}\times H^{\ca{O}}(\mk{N},f)$. 
By assumption, for $i\neq j$ we have $[(r_i,\tau_i)]\not\sim_f [(r_j,\tau_j)]$ and therefore
$[(r_j^{-1},\tau_i)]\not\sim_f [(r_i^{-1},\tau_j)]$.
Thus, using Lemma \ref{tricky} we deduce that for $i\neq j$ one has
\begin{align}\label{gaul}
[(r_i^{-1},\tau_i^*)]\not\sim_f[(r_j^{-1},\tau_j^*)].
\end{align} 
Now let $Q_{\tau_i}(x,y)=A_ix^2+B_ixy+C_iy^2$, so that 
$$
Q_{\tau_i^*}(x,y)=\sign(C_i)\left(C_iNf^2x^2-B_ifxy+\frac{A_i}{N}y^2\right).
$$ 
Now set
\begin{align*}
\mk{b}_i=(\wt{r} A_{\tau^*}\Lambda_{\tau^*})^{-1}
(\wt{r}_i A_{\tau_i^*}\Lambda_{\tau_i^*}).
\end{align*} 
From Lemma \ref{tricky} one has that $\{\mk{b}_i\}_{i\in I}$ are inequivalent invertible $\ca{O}^*$-ideals 
modulo $\sim_f$. Moreover, as shown previously, for each $i$ one has that $\wt{\Theta}_1(ker(\wt{\Theta}_2))$ 
acts trivially on 
$[\mk{b}_i]$. Now it is not so obvious but it turns out that in each class 
$[\mk{b}_i]$ one may find a representative which is
integral and coprime to 
$\mk{f}^*$. Thanks to Corollary 1 of \cite{erratum_1}, this is indeed the case, therefore 
one may as well assume that each 
$\mk{b}_i\in I_{\ca{O}^*}(f)$ and that $\mk{b}_i\wt{r} A_{\tau^*}\Lambda_{\tau^*}\sim_f
\wt{r}_i A_{\tau_i^*}\Lambda_{\tau_i^*}$. It thus 
follows that the collection $\{\mk{b}_i\}_{i\in I}$ gives a complete set of
representatives of $I_{\ca{O}^*,1}(f)$ modulo $\wt{\Theta}_1(ker(\wt{\Theta}_2))$. Therefore
$\{[(r_i,\tau_i^*)]\}_{i\in I}$ gives a complete set of representatives of $[(r,\tau^*)]$
under the action $C_{\ca{O}^*}(f)$. Moreover, since $\wt{\Theta}_1(ker(\wt{\Theta}_2))$ acts trivially
on each class $[(r_i,\tau_i^*)]$ we see that the action of $C_{\ca{O}^*}(f)$ factors through 
$C_{\ca{O}}(f)$. This concludes the proof. \fin
\section{Norm formulas for $u_D$}
In this section we first prove a product formula for the $p$-adic invariant $u_D$ which involve different
orders of $K$. Second, we prove a non-trivial relationship between the $p$-adic 
invariants $u_D$ and $u_D^*$ where $u_D^*$ is the \lq\lq $p$-adic invariant $u_D$ twisted by the 
the involution $*$\rq\rq\s where $*$ is the involution given by $\tau\mapsto\frac{-1}{fN\tau}$.
Thirdly, we prove the main result of this paper (Theorem \ref{prec_th}) which may be viewed as a precise version of
Theorem \ref{vag_th}. Finally, we end this section by proving Proposition \ref{sam_err} which corrects
a mistake that appeared in \cite{Das3}.

For the rest of this subsection we fix two orders $\ca{O}$ and $\ca{O}'$ of conductor 
coprime to $N$, a splitting of $N$ namely a factorization of
$N\ca{O}_K$ as $\mk{N}\mk{N}^{\sigma}$, a good divisor 
$\wt{\delta}\in\wt{\Div}_f(\mk{N})$. For the whole section, we suppress the divisor $\wt{\delta}$ from the notation.
\begin{Prop}\label{pointe}
Let $\mk{b}\subseteq\ca{O}$ (resp. $\mk{b}'\subseteq\ca{O}'$) be an invertible $\ca{O}$-ideal
(resp. an invertible $\ca{O}'$-ideal). Let $\mk{g}\subseteq\ca{O}$ (resp. $\mk{g}\subseteq\ca{O}'$)
be an $\ca{O}$-ideal (resp. an $\ca{O}'$ ideal) which is not necessarily $\ca{O}$-invertible
(resp. $\ca{O}'$-invertible). Assume that $\mk{b}^{-1}\mk{g}\supseteq \mk{b}'^{-1}\mk{g}'$, 
$([\mk{b}^{-1}\mk{g}:\mk{b}'^{-1}\mk{g}'],pN)=1$, that 
$\Gamma_{\mk{b}}(\mk{g})\supseteq\Gamma_{\mk{b}'}(\mk{g}')$ and set 
$e=[\Gamma_{\mk{b}}(\mk{g}):\Gamma_{\mk{b}'}(\mk{g}')]$. Assume furthermore
that the following condition is satisfied:
\begin{align}\label{dag}\tag{$\dagger$}
\mbox{For all $x\in\mk{b}^{-1}\mk{g}$ ($x\neq -1$), $(x+1)\mk{b}'$ is $\mk{g}'$-int and 
$\Gamma_{(x+1)\mk{b}'}(\mk{g}')\supseteq\Gamma_{\mk{b}'}(\mk{g}')$}.
\end{align}
Let $\{\lambda\}$ be a complete set of representatives 
of $\mk{b}^{-1}\mk{g}/\mk{b}'^{-1}\mk{g}'$ where each
$\lambda$ is chosen to be totally positive and let $\mu_{\lambda}=\lambda+1$. Then
\begin{align}\label{zoller}
u_{D}(\mk{b},\mk{g})^e=\prod_{\lambda} u_{D}(\mu_\lambda\mk{b}',\mk{g}')^{h_{\lambda}},
\end{align}
where the $h_{\lambda}$'s are integers such that
\begin{align*}
h_{\lambda}=[\Gamma_{\mu_{\lambda}\mk{b}'}(\mk{g}'):\Gamma_{\mk{b}'}(\mk{g}')].
\end{align*}
Moreover, \eqref{zoller} is compatible with the ad\'elic action of $C_K$ on $u_D$ which appears
in Definition \ref{ad_action}.
\end{Prop}

We note that by assumption $\mk{b}$ is $\mk{g}$-int so that the left hand side of \eqref{zoller}
makes sense.  Moreover, using \eqref{dag} we find that
$\mu_{\lambda}\mk{b}'$ is $\mk{g}'$-int so that the $p$-adic invariant appearing in the product of 
the right hand side of \eqref{zoller} makes sense. We would like to emphasize here that the left hand side of 
\eqref{zoller} is a $p$-adic invariant relative to the order $\ca{O}$ and that the right hand side of 
\eqref{zoller} is a product of $p$-adic invariants relative to the order $\ca{O}'$. So when
$\ca{O}\neq\ca{O}'$, \eqref{zoller} gives a non-trivial relationship between $p$-adic invariants 
associated to different orders.
\begin{Rem}
We note that since by assumption $\Gamma_{\mk{b}}(\mk{g})\supseteq\Gamma_{\mk{b}'}(\mk{g}')$
it is expected (see Remark \ref{axa}) that the abelian extension generated by the 
$u_D(\mu_{\lambda}\mk{b}',\mk{g}')$'s over $K$ contains the abelian extension
generated by $u_D(\mk{b},\mk{g})$ over $K$. 
\end{Rem}
\begin{proof}
Using Lemma \ref{tired} it is enough to show the proposition for 
\begin{align*}
\wt{\delta}=\sum_{r\in\laa p\pmod{f}\raa}(d'[\mk{d}',r]-d[\mk{d},r]),
\end{align*}
where $\mk{d}'$ is consecutive to $\mk{d}$,
$\eta=\frac{\mk{d}'}{\mk{d}}$, $d=\Norm(\mk{d})$ and $d'=\Norm(\mk{d}')$. 
First note that since $\mk{b}$ (resp. $\mk{b}'$) is an $\ca{O}$-integral ideal (resp. $\ca{O}'$-integral) 
we have
\begin{align}\label{dieu}
\Gamma_{\mk{b}}(\mk{g})=(\mk{b}^{-1}\mk{g}+1)\cap\ca{O}(\infty)^{\times}
\s\s\s\mbox{and}\s\s\s\Gamma_{\mk{b}'}(\mk{g}')=(\mk{b}'^{-1}g'+1)\cap\ca{O}'(\infty)^{\times}.
\end{align} 
Let $\wt{\mathcal{D}}=\mathcal{D}_{\ca{O}}^{can}$ be the canonical fundamental domain for the action of 
$\ca{O}(\infty)^{\times}$ on the positive quadrant $Q$. Using the assumption that $(\cond(\ca{O}),N)=1$,
one may check that if $\eta|\mk{N}$ then $\eta$ is $\ca{O}$-good for $\wt{\mathcal{D}}$. Let
\begin{align*}
\mathcal{D}=\bigcup_{\epsilon}\epsilon\wt{\mathcal{D}},
\end{align*} 
where the set $\{\epsilon\}$ is a complete set of representatives of 
$\ca{O}(\infty)^{\times}/\Gamma_{\mk{b}}(\mk{f})$. Note that $\mathcal{D}$ is a
fundamental domain for the action of $\Gamma_{\mk{b}}(\mk{f})$ on the positive quadrant
$Q$. One may also check that for every prime 
$\eta|\mk{N}$, $\eta$ is $\ca{O}$-good for $\mathcal{D}$. Finally, set
\begin{align}\label{bald}
\mathcal{D}'=\bigcup_{\epsilon'}\epsilon'\mathcal{D},
\end{align} 
where $\{\epsilon'\}$ is a complete set of representatives of 
$\Gamma_{\mk{b}}(\mk{g})/\Gamma_{\mk{b}'}(\mk{g}')$.
In a similar way, $\mathcal{D}'$ is a fundamental domain for the
action of $\Gamma_{\mk{b}'}(\mk{g}')$ on $Q$ and for every prime $\eta|\mk{N}$,
$\eta$ is $\ca{O}'$-good for $\mathcal{D}'$. By assumption we have
\begin{align*}
\mk{b}^{-1}\mk{g}+1=\bigcup_{\lambda}(\mk{b}'^{-1}\mk{g}'+\lambda+1),
\end{align*}
where $\{\lambda\}$ is a complete set of totally positive representatives of 
$\mk{b}^{-1}\mk{g}/\mk{b}'^{-1}\mk{g}'$.

For every $\epsilon'$ we have
\begin{align}\label{tripatif}\notag
\epsilon'\mathcal{D}\cap (\mk{b}^{-1}\mk{g}+1)&=
\bigcup_{\lambda}\epsilon'\mathcal{D}\cap(\mk{b}'^{-1}\mk{g}'+\lambda+1)\\
&=\bigcup_{\lambda}\mu_{\lambda}
\left(\mu_{\lambda}^{-1}\epsilon'\mathcal{D}\cap((\mu_{\lambda}\mk{b}')^{-1}\mk{g}'+1)\right),
\end{align}
where $\mu_{\lambda}=1+\lambda$ (note that $\mu_{\lambda}\neq 0$ since $\lambda\gg 0$).
Let $U\subseteq\ca{O}_{K_p}^{\times}$ be an arbitrary compact-open set. Using \eqref{tripatif}, 
\eqref{bald} and the homogeneous property of the measure $\nu_{\wt{\delta}}$, we may deduce that
\begin{align}\label{sabre1}
\sum_{\epsilon'}\nu_{\wt{\delta}}(\mk{b},\mk{g},1,\epsilon'\mathcal{D})(U)
&=\sum_{\lambda}\nu_{\wt{\delta}}(\mu_{\lambda}\mk{b}',\mk{g}',1,
\mu_{\lambda}^{-1}\mathcal{D}')(\mu_{\lambda}^{-1}U)\\ \label{sabre11}
&=\sum_{\lambda}\nu_{\wt{\delta}}h_{\lambda}(\mu_{\lambda}\mk{b}',\mk{g}',1,
\mu_{\lambda}^{-1}\mathcal{D}_{\lambda}')(\mu_{\lambda}^{-1}U),
\end{align}
where $\mathcal{D}_{\lambda}'$ is a fundamental domain for the action of 
$\Gamma_{\mu_{\lambda}\mk{b}'}(\mk{g}')$ on $Q$ chosen so that
\begin{align*}
\mathcal{D}'=\bigcup_{\{\epsilon_i(\lambda)\}_i}\epsilon_i(\lambda)\mathcal{D}_{\lambda}',
\end{align*}  
where $\{\epsilon_i(\lambda)\}_i$ is a complete set of representatives of 
$\Gamma_{\mu_{\lambda}\mk{b}'}(\mk{g}')/\Gamma_{\mk{b}'}(\mk{g}')$. Moreover, we note that
$\eta$ is $\ca{O}$-good for the domain $\mathcal{D}_{\lambda}'$. Here 
the $h_{\lambda}$'s are positive integers such that
\begin{align}\label{cool}
h_{\lambda}=[\Gamma_{\mu_{\lambda}\mk{b}'}(\mk{g}'):\Gamma_{\mk{b}'}(\mk{g}')].
\end{align} 
If we take $U=\ca{O}_{K_p}$ in \eqref{sabre1} and we unfold the definition of $\nu_{\wt{\delta}}$ then we find
\begin{align}\label{tricot}
e\zeta_{\wt{\delta}}(\mk{b},\mk{g}\infty,0)=\sum_{\lambda}h_{\lambda}\zeta_{\wt{\delta}}
(\mu_{\lambda}\mk{b}',\mk{g}'\infty,0),
\end{align}
where $e=[\Gamma_{\mk{b}}(\mk{g}):\Gamma_{\mk{b}'}(\mk{g}')]$.
Using \eqref{tricot} we find that
\begin{align*}
ev_p(u_D(\mk{b},\mk{g}))=\sum_{\lambda}h_{\lambda}v_p(u_D(\mu_{\lambda}\mk{b}',\mk{g}')).
\end{align*}
We have an isomorphism $K_p^{\times}\simeq p^{\ZZ}\times\ca{O}_{K_p}^{\times}$. Let us denote the projection
on the $i$-th coordinate by $\pi_i$. From \eqref{tricot} we get that 
$\pi_1(u_D(\mk{b},\mk{g})^e)=\pi_1(\prod_{\lambda}u_D(\mu_{\lambda}\mk{b}',\mk{g}')^{h_{\lambda}})$. 
Thus in order to show \eqref{zoller} it remains to show that $\pi_2(u_D(\mk{b},\mk{g})^e)$
coincides with $\pi_2(\prod_{\lambda}u_D(\mu_{\lambda}\mk{b}',\mk{g}')^{h_{\lambda}})$. 
Using the fact that $p\star\wt{\delta}=\wt{\delta}$ one finds that
\begin{align*}
\nu_{\wt{\delta}}(\mk{b},\mk{g},1,\mathcal{D})(\ca{O}_{K_p}^{\times})
=\zeta_{\{p\}}(\mk{b},\mk{g}\infty,0)=0,
\end{align*}
and
\begin{align}\label{pilier2}
\nu_{\wt{\delta}}(\mu_{\lambda}\mk{b}',\mk{g}',1,\mu_{\lambda}^{-1}\mathcal{D}')(\ca{O}_{K_p}^{\times})
=\zeta_{\{p\}}(\mu_{\lambda}\mk{b}',\mk{g}'\infty,0)=0.
\end{align}
Now using the assumption that $([\mk{b}^{-1}\mk{g}:\mk{b}'^{-1}\mk{g}'],pN)=1$ 
we may assume that the $\lambda$'s are chosen in such a way that the $\mu_{\lambda}$'s 
are coprime to $pN$. So in particular, one has that
$\mu_{\lambda}\ca{O}_{K_p}^{\times}=\ca{O}_{K_p}^{\times}$. Now make the crucial observation
that $\eta$ is again $\ca{O}$-good for the cone $\mu_{\lambda}^{-1}\mathcal{D}'$ (it uses $(\mu_{\lambda},N)=1$). 
We have
\begin{align}\label{tansut}\notag
\mint{\ca{O}_{K_p}^{\times}}{}x\cdot d\nu_{\wt{\delta}}(\mu_{\lambda}\mk{b}',\mk{g},1,
\mu_{\lambda}^{-1}\mathcal{D}')(\mu_{\lambda}^{-1} x)&=
\mint{\ca{O}_{K_p}^{\times}}{}x\cdot d\nu_{\wt{\delta}}(\mu_{\lambda}\mk{b}',\mk{g}',1,
\mu_{\lambda}^{-1}\mathcal{D}')(x)\\
&=\mint{\ca{O}_{K_p}^{\times}}{}x\cdot d\nu_{\wt{\delta}}(\mu_{\lambda}\mk{b}',\mk{g}',1,
\mathcal{D}')(x),
\end{align}
where the first equality follows from \eqref{pilier2} and the second follows from 
the independence on the choice of the fundamental domain. From \eqref{tansut} and \eqref{sabre11} 
we deduce that
\begin{align}\label{tansut2}
\prod_{\lambda}\mint{\ca{O}_{K_p}^{\times}}{}x\cdot d\nu_{\wt{\delta}}(\mu_{\lambda}\mk{b}',\mk{g},1,
\mu_{\lambda}^{-1}\mathcal{D}')(\mu_{\lambda}^{-1} x)=\prod_{\lambda}
\pi_2(u_D(\mu_{\lambda}\mk{b}',\mk{g}')^{h_{\lambda}}).
\end{align} 
On the other hand, we also have
\begin{align}\label{tansut3} \notag
\prod_{\lambda}\mint{\ca{O}_{K_p}^{\times}}{}x\cdot d\nu_{\wt{\delta}}(\mu_{\lambda}\mk{b}',\mk{g}',1,
\mu_{\lambda}^{-1}\mathcal{D}')(\mu_{\lambda}^{-1}x)&=
\prod_{\lambda,\epsilon'}\mint{\ca{O}_{K_p}^{\times}}{}x\cdot d\nu_{\wt{\delta}}
(\mu_{\lambda}\mk{b}',\mk{g}',1,
\mu_{\lambda}^{-1}\epsilon'\mathcal{D})(\mu_{\lambda}^{-1}x)\\ \notag
&=\prod_{\epsilon'}\mint{\ca{O}_{K_p}^{\times}}{}x\cdot d\nu_{\wt{\delta}}
(\mk{b},\mk{g},1,\epsilon'\mathcal{D})(x)\\
&=\pi_2(u_D(\mk{b},\mk{g})^e).
\end{align}
The first equality follows from \eqref{bald}. The second equality follows from \eqref{sabre1}. 
The third equality follows from the independence on the choice of the fundamental domain. 
Finally, combining \eqref{tansut2} with \eqref{tansut3} we find that
\begin{align*}
\pi_2(u_D(\mk{b},\mk{g})^e)=\pi_2\left(\prod_{\lambda}u_D(\mu_{\lambda}\mk{b}',\mk{g}')^{h_{\lambda}}\right).
\end{align*} 
This concludes the proof of \eqref{zoller}.

Finally, it remains to show that \eqref{zoller} is compatible to 
the ad\'elic action of $C_K$. To fix the idea, assume that $\ca{O}'\subseteq\ca{O}$
(a similar argument works when this inclusion is reversed). Let $n=\cond(\ca{O})$
and let $g=[\ca{O}':\mk{g}']$. Let 
\begin{align*}
\pi:C_K\rightarrow C_{\ca{O}}(ng),
\end{align*}
be the natural map given by class field theory. We have natural surjective maps
\begin{align*}
p:C_{\ca{O}}(gn)\rightarrow C_{\ca{O}}(\mk{g})\s\s\s\mbox{and}\s\s\s
p':C_{\ca{O}}(gn)\rightarrow C_{\ca{O}'}(\mk{g}').
\end{align*} 
Let $c\in C_K$, $p(c)=[\mk{c}]$ and 
$p'(c)=[\mk{c}']$. Without loss
of generality we may assume that $\mk{c}$ and $\mk{c}'$ are chosen so that
$(\mk{c},\mk{b})=\ca{O}$,
$(\mk{c}',\mk{b}')=\ca{O}'$, $\mk{c}\cap\ca{O}'=\mk{c}'$,
$(\mk{c},\mk{g})=\ca{O}$ and $(\mk{c}',\mk{g}')=\ca{O}'$.
Let $\{\lambda\}$ be a complete set of representatives
of $\mk{b}^{-1}\mk{g}/{\mk{b}'}^{-1}\mk{g}'$ such that $\lambda\in\mk{c}^{-1}$ so that 
\begin{align}\label{tramp2}
\mk{b}^{-1}\mk{g}=\bigcup_{\lambda}({\mk{b}'}^{-1}\mk{g}'+\lambda).
\end{align}
Since every $\lambda\in\mk{b}^{-1}\mk{g}$, we find, after intersecting both sides of 
the equality \eqref{tramp2} with $\mk{c}^{-1}$, that
\begin{align}\label{tramp3}\notag
(\mk{b}\mk{c})^{-1}\mk{g}=\mk{c}^{-1}\cap(\mk{b}^{-1}\mk{g})=\bigcup_{\lambda}
\mk{c}^{-1}\cap(\mk{b}'^{-1}\mk{g}+\lambda)&=
\bigcup_{\lambda}(\mk{c}^{-1}\cap(\mk{b}'^{-1}\mk{g}')+\lambda)\\
&=\bigcup_{\lambda}((\mk{c}'\mk{b}')^{-1}\mk{g}'+\lambda).
\end{align}
The first equality follows from $\mk{c}\cap\mk{b}=\mk{c}\mk{b}$ and the second equality
follows from \eqref{tramp2}. The third equality
follows from the assumption that $\lambda\in\mk{c}^{-1}$ and the last equality
follows from the assumption that $\mk{c}\cap\mk{\ca{O}'}=\mk{c}'$ and $(\mk{c}',\mk{b}')=\ca{O}'$. 
Now a direct computation shows that
\begin{align*}
c\star u_D(\mk{b},\mk{g})=u_D(\mk{c}\mk{b},\mk{g})
=\prod_{\lambda}u_D(\mu_{\lambda}\mk{c}'\mk{b}',\mk{g}')=
\prod_{\lambda}c\star u_D(\mu_{\lambda}\mk{b}',\mk{g}').
\end{align*}
The first and the third equalities follow from the definition of $\star$ 
(see Definition \ref{ad_action}).
The second equality follows from \eqref{tramp3} and the definition of $u_D$.
This shows that the Shimura reciprocity law is compatible with the product \eqref{zoller}.
This concludes the proof. \fin
\end{proof}

Armed with Proposition \ref{pointe} we may now find a non-trivial relationship between the $p$-adic 
invariants $u_D$ and $u_D^*$ where $u_D^*$ is the \lq\lq $p$-adic invariant $u_D$ twisted by the 
the involution $*$\rq\rq\s where $*$ is the involution given by $\tau\mapsto\frac{-1}{fN\tau}$.
\begin{Prop}\label{prec_prop}
Let $\tau\in H^{\ca{O}}(\mk{N})$, $\mk{f}=f\ca{O}_{\tau}$ and $\mk{f}^*=f\ca{O}_{\tau^*}$ 
where $\tau^*=\frac{-1}{fN\tau}$. Set 
\begin{align}\label{solfege}
(\mk{b}^*)^{-1}=\Lambda_{\tau^*}\s\s\s\mbox{and}\s\s\s
\mk{b}^{-1}=\Lambda_{f\tau^*}.
\end{align}
Then
\begin{align}\label{voile1}
u_D(\mk{b}^*,\mk{f}^*)^e=\prod_{\lambda}u_D(\mu_{\lambda}\mk{b},\mk{f})^{h_{\lambda}},
\end{align}
where $\{\lambda\}$ is a complete set of totally positive representatives of $(\mk{b}^*)^{-1}/\mk{b}^{-1}$, 
$\mu_{\lambda}=\lambda f+1$ and the $h_{\lambda}$'s and $e$ are as in Proposition \ref{pointe}. Similarly,
if one sets 
\begin{align}\label{solfege2}
\mk{b}^{-1}=\Lambda_{\tau}\s\s\s\mbox{and}\s\s\s
(\mk{b}^*)^{-1}=\Lambda_{f\tau}.
\end{align}
one has that
\begin{align}\label{voile2}
u_D(\mk{b},\mk{f})^{e'}=\prod_{\lambda'}u_D(\mu_{\lambda'}\mk{b}^*,\mk{f}^*)^{h_{\lambda'}},
\end{align}
where $\{\lambda'\}$ is a complete set of totally positive representatives of $\mk{b}^{-1}/(\mk{b}^*)^{-1}$, 
$\mu_{\lambda'}=\lambda' f+1$ and the $h_{\lambda'}$'s and $e'$ are as in Proposition \ref{pointe}.
\end{Prop}
\begin{proof}
First note that $\cond(\ca{O}^*)=f\cdot \cond(\ca{O})$, 
$\tau^*=\frac{-1}{fN\tau}\in H^{\ca{O}^*}(\mk{N})$, 
$\frac{1}{N\tau}\in H^{\ca{O}}(\mk{N})$,
$\mk{b}$ is $\mk{f}$-int and $\mk{b}^*$ is $\mk{f}^*$-integral. 
We have $(\mk{b}^*)^{-1}\supseteq\mk{b}^{-1}$ and that 
$(\mk{b}^*)^{-1}\mk{f}^*/\mk{b}^{-1}\mk{f}\simeq(\mk{b}^*)^{-1}/\mk{b}^{-1}\simeq\ZZ/f\ZZ$. 
Now make the following crucial observation:
\begin{enumerate}
\item The condition \eqref{dag} is satisfied: namely, for all $x\in(\mk{b}^{*})^{-1}\mk{f}^*$ ($x\neq -1$) one has
that $(1+x)\mk{b}$ is $\mk{f}$-int and that $\Gamma_{(1+x)\mk{b}}(\mk{f})\supseteq\Gamma_{\mk{b}}(\mk{f})$.
\end{enumerate}
Now applying Proposition \ref{pointe} we find that
\begin{align}\label{flicken}
u_D(\mk{b}^*,\mk{f}^*)^{e}=\prod_{\lambda}u_D(\mu_{\lambda}\mk{b},\mk{f})^{h_{\lambda}},
\end{align}
where $\mu_{\lambda}=1+f\lambda$ where 
$\{\lambda\}$ is a complete set of totally positive representatives of $(\mk{b}^*)^{-1}/\mk{b}^{-1}$,
$e$ and $h_{\lambda}$'s are chosen as in Proposition \ref{pointe}. This
proves \eqref{voile1}. The proof of \eqref{voile2} follows from \eqref{voile1} by 
replacing $\tau$ by $\tau^*$. \fin
\end{proof}

We may now give a precise version of Theorem \ref{vag_th} that was stated at the end of the introduction.
\begin{Th}\label{prec_th}
Let $(r,\tau)\in\ZZ/f\ZZ\times H^{\ca{O}}(\mk{N})$ and set 
\begin{align*}
(\mk{b}^*)^{-1}=\frac{1}{\wt{r}}\Lambda_{(\tau^*)^{\sigma}}
\s\s\s\mbox{and}\s\s\s\mk{b}^{-1}=\frac{1}{\wt{r}}\Lambda_{f(\tau^*)^{\sigma}},
\end{align*}
where $\tau^*=\frac{-1}{fN\tau}$ and $\sigma$ is the non-trivial automorphism of $\Gal(K/\QQ)$. 
Let $\ca{O}^*=\ca{O}_{\tau^*}$, $\mk{f}=f\ca{O}$ and $\mk{f}^*=f\ca{O}^*$. Then we have
\begin{align}\label{klos}
u_C(r,\tau)^{\nu(r,\tau)e}=u_D(\mk{b}^*,\mk{f}^*)^{12e}=\prod_{\lambda}u_D(\mu_{\lambda}\mk{b},\mk{f})^{12h_{\lambda}},
\end{align}
where $\mu_{\lambda}=1+f\lambda$, $\{\lambda\}$ being a complete set of totally positive 
representatives of $(\mk{b}^*)^{-1}/\mk{b}^{-1}$ and $\nu(r,\tau)$ is the integer which appears
in Proposition \ref{klasse} and the $h_{\lambda}$'s and $e$ are as in Proposition \ref{pointe}. 
Moreover, we have that
\begin{align}\label{tickle}
u_D(\mk{b},\mk{f})^{12e'}=\prod_{\lambda'}u_D(\mu_{\lambda'}\mk{b}^*,\mk{f}^*)^{12e'}
&=\prod_{\lambda'}c_{\lambda'}\star\left(u_D(\mk{b}^*,\mk{f}^*)^{12 h_{\lambda'}}\right)\\ \label{tickle2}
&=\prod_{\lambda'}c_{\lambda'}\star(u_C(r,\tau^*)^{\nu(r,\tau^*) h_{\lambda'}}),
\end{align}
where $\mu_{\lambda'}=1+f\lambda'$ and $\{\lambda'\}$ is a complete set of representatives
of $(f\mk{b}^{-1})/(\mk{b}^*)^{-1}$.Here
$c_{\lambda'}\in C_K$ is chosen so that $\pi(c_{\lambda'})\sim_{\mk{f}^*}\lambda'\ca{O}^*$,
where $\pi:C_K\rightarrow I_{\ca{O}^*}/P_{\ca{O}^*,1}(\mk{f}^*)$ is the natural projection map.
\end{Th}
We would like to emphasize here that the $p$-adic invariants $u_C(r,\tau)$ and 
$u_D(\mu_{\lambda}\mk{b},\mk{f})$'s which appear in \eqref{klos} are relative to the order $\ca{O}$. Similarly,
the $p$-adic invariants
$u_D(\mu_{\lambda}\mk{b},\mk{f})$ and $u_C(r,\tau^*)$'s which appear in \eqref{tickle2} are relative
to the order $\ca{O}^*$.
\begin{proof}
The first equality in \eqref{klos} follows from Proposition \ref{klasse}. The second equality
follows from \eqref{klos} follows from \eqref{voile1}. The first equality in \eqref{tickle} follows from 
\eqref{voile1} and the second
equality follows from the Shimura reciprocity law applied to $u(\mk{b}^*,\mk{f}^*)$. Finally \eqref{tickle2}
follows from Proposition \ref{klasse} and the compatibility of the Shimura reciprocity laws of 
$u_C$ and $u_D$. This concludes the proof.\fin
\end{proof}

We have the following corollary which may be viewed as a precise version of Theorem \ref{vag_th}.
\begin{Cor}\label{fatigue}
Let $\rho\in H^{\ca{O}_K}(\mk{N},f)$. If one sets $\tau=\rho$ then one has that
\begin{align}\label{klar}
u_C(r,\tau)^{\nu(r,\tau)}=\prod_{\lambda}u_D(\mu_{\lambda}\mk{b},\mk{f})^{12h_{\lambda}}.
\end{align}
where
the quantities $\mk{b},\mk{f},\mu_{\lambda}$ and $h_{\lambda}$ are chosen as in \eqref{klos} 
of Theorem \ref{prec_th}.
In a similar way, if one sets $\tau=\rho^*:=\frac{-1}{fN\rho}$, one has that
\begin{align}\label{klar2}
u_D(\mk{b},\mk{f})^{12}=
\prod_{\lambda'}c_{\lambda'}\star u_C(r,\tau^*),
\end{align}
where the quantities $\mk{b},\mk{f}$, $e'$ and $c_{\lambda'}$ 
are chosen as in \eqref{tickle2} of Theorem \ref{prec_th}.
\end{Cor}

\begin{proof}
First let $Q_{\rho}(x,y)=Ax^2+Bxy+Cy^2$ (note that $(A,f)=1$). Let us treat the first 
case, i.e., when $\tau=\rho$. We have
\begin{align}\label{fati}
Q_{\tau^*}(x,y)=\sign(C)\left(Cf^2Nx^2-Bfxy+\frac{A}{N}y^2\right),
\end{align}
so that
\begin{align}\label{fatig}
\tau^*=\frac{B+f\sqrt{d_K}}{fCN}
\end{align}
From \eqref{klos} of Theorem \ref{prec_th} we get
\begin{align*}
u_C(r,\tau)^{\nu(r,\tau)e}=\prod_{\lambda}u_D(\mu_{\lambda}\mk{b},\mk{f})^{12h_{\lambda}}.
\end{align*}
Using \eqref{fatig} one may check that $e=[\Gamma_{\mk{b}^*}(\mk{f}^*):\Gamma_{\mk{b}}(\mk{f})]=1$. 
This proves \eqref{klar}. 

Now let us treat the second case, i.e., when $\tau=\frac{-1}{fN\rho}=:\rho^*$. Then in this case we have 
that $Q_{\tau}(x,y)=\sign(C)\left(Cf^2Nx^2-Bfxy+\frac{A}{N}y^2\right)$. Therefore,
\begin{align}\label{fati2}
Q_{\tau^*}(x,y)=Ax^2+Bxy+Cy^2.
\end{align}
so that
\begin{align}\label{fatig2}
\tau^*=\frac{-B+\sqrt{d_K}}{\frac{A}{N}}.
\end{align}
Applying \eqref{tickle2} of Theorem \ref{prec_th} with the previous parameters we find that
\begin{align*}
u_D(\mk{b},\mk{f})^{12e'}=
\prod_{\lambda'}c_{\lambda'}\star\left(u_C(r,\tau^*)^{\nu(r,\tau^*)h_{\lambda'}}\right)
\end{align*}
Using \eqref{fatig2}, a direct computation shows that for all $\lambda'$ one has that 
$\nu(r,\tau^*)h_{\lambda'}=1$. Moreover, one also have $e'=1$. This concludes the proof of \eqref{klar2}. \fin
\end{proof}

\begin{Rem}\label{implication1}
In the special case where $f=1$ and $\tau\in H^{\ca{O}}(\mk{N})$ we may deduce 
from \eqref{klos} that
\begin{align}\label{star1}
u_C(1,\tau)=u_D(I_{\tau^*},\ca{O})^{12},
\end{align}
where $\tau^*=\frac{-1}{N\tau}$. Here, note that $I_{\tau^*}^{-1}=\Lambda_{(\tau^*)^{\sigma}}$ (for
the definition of $I_{\tau}$ see Definition \ref{ini_def}). Moreover, it follows from Remark 
\ref{hausen} (or the computation carried in Appendix \ref{goeland}) that
\begin{align}\label{star2}
u_{C}(1,\tau)=u_{DD}(\tau).
\end{align}
Combining \eqref{star2} with \eqref{star1} we deduce that
\begin{align}\label{star3}
u_D(I_{\tau^*},\ca{O})^{12}=u_{DD}(\tau).
\end{align}
It thus follows from \eqref{star3} that the $p$-adic invariant $u_D$ (once his definition
is extended to arbitrary orders as was done in Definition \ref{crocus}) 
may be viewed as a natural generalization of the Darmon-Dasgupta $p$-adic invariant $u_{DD}$
with the subtlety that the appearance of $\tau$ in the argument of $u_D$ is twisted by the Atkin-Lehner involution
of level $N$ namely $\tau\mapsto\frac{-1}{N\tau}=\tau^*$
\end{Rem}

For the rest of the section we assume that $f=1$ and we
let $\cond{\ca{O}}=n$ so that $\ca{O}=\ca{O}_n$ where $(n,N)=1$. For every 
$\tau\in H^{\ca{O}}(\mk{N})$ we may consider the $p$-adic invariant 
$$
u_D(I_{\tau^*},\ca{O})^{12}=u_{DD}(\tau)=u(\alpha_{\delta},\tau),
$$ 
where the first equality follows from Remark \ref{implication1}. The $p$-adic invariant
$u(\alpha_{\delta},\tau)$ is the the Darmon-Dasgupta $p$-adic invariant
constructed in \cite{Dar-Das}. Here $\delta\in\Div(N)$ is the fixed ambient divisor and
$\alpha_{\delta}$ is the associated modular unit to $\delta$ (see Appendix \ref{goeland}).
Conjecture \ref{u_H_conj}  implies that 
$u_D(I_{\tau^*},\ca{O})^{12}\in K(\ca{O}_n\infty)$. Moreover, for every $\ca{O}_K$-ideal 
$\mk{a}\subseteq\ca{O}_K$ we have a $p$-adic invariant $u_D(\mk{a},\mk{n})$
where $\mk{n}=n\ca{O}_K$. In a similar way, Conjecture \ref{alg_conj} predicts that 
$u_D(\mk{a},\mk{n})\subseteq K(\mk{n}\infty)$. Since 
$K(\ca{O}\infty)\subseteq K(\mk{n}\infty)$ it is natural to ask if a suitable 
product of $p$-adic invariants $u_D(\mk{a},\mk{n})$'s (where $\mk{a}$ here varies
and $\mk{n}$ is fixed) is equal to $u_D(I_{\tau^*},\ca{O})^{12}$. The proposition below 
gives a positive answer. Before stating it, let us make a few observations.

Let $\tau\in H^{\ca{O}_K}(\mk{N},n)$. We thus have that $Q_{\tau}(x,y)=Ax^2+Bxy+Cy^2$ with
$(A,n)=1$ and $d_K=B^2-4AC$. Set 
\begin{align}\label{tract}
\mk{b}^{-1}=\Lambda_{n(\tau^*)^{\sigma}}\s\s\s\mbox{and}\s\s\s(\mk{b}')^{-1}=\Lambda_{(\tau^*)^{\sigma}},
\end{align}
where $\tau^*=\frac{-1}{N\tau}$. Note that
\begin{align*}
Q_{\tau^*}(x,y)=\sign(C)\left(CNx^2-Bxy+\frac{A}{N}y^2\right),
\end{align*}
and that
\begin{align*}
Q_{n\tau^*}(x,y)=\sign(C)\left(CNn^2x^2-Bnxy+\frac{A}{N}y^2\right),
\end{align*}
so that $\End(\Lambda_{(\tau^*)^{\sigma}})=\ca{O}_K$ and $\End(\Lambda_{n(\tau^*)^{\sigma}})=\ca{O}_n$. 
We may now state the proposition.
\begin{Prop}\label{sam_err}
We have the following identity
\begin{align}\label{ticle}
u_D(I_{\tau^*},\ca{O}_n)^e=u_D(\mk{b},\ca{O}_n)^e=\prod_{\lambda}u_D(\mu_{\lambda}\mk{b}',n\ca{O}_K)^{h_{\lambda}},
\end{align}
where $\mu_{\lambda}=1+\lambda$ and $\{\lambda\}$ is a complete set of representatives 
$\mk{b}^{-1}/n(\mk{b}')^{-1}$, 
$$
h_{\lambda}=[\Gamma_{\mu_{\lambda}\mk{b}'}(n\ca{O}_K):\ca{O}_K(n\infty)^{\times}],
$$ 
and $e=[\ca{O}_n(\infty)^{\times}:\ca{O}_K(n\infty)^{\times}]$, i.e., $e$ 
is the order of the image of a generator $\epsilon_n$ of 
$\ca{O}_n(\infty)^{\times}$ inside the group $(\ZZ/n\ZZ)^{\times}$.
\end{Prop}
\begin{proof}
Apply Proposition \ref{pointe} with the parameters $\ca{O}=\ca{O}_n$, $\ca{O}'=\ca{O}_K$, 
$\mk{g}=\ca{O}_n$ and $\mk{g}'=n\ca{O}_K$. \fin
\end{proof}
\begin{Rem}
We keep the same notation as in \eqref{tract}. 
We would like to point out a mistake which appears in Theorem 8.3 of \cite{Das3} which claims that
\begin{align}\label{lang}
u_{DD}(\mk{b},\ca{O}_n)^{\frac{1}{6}}=\prod_{a}u_D((a)\mk{b}',n\ca{O}_K)^{2},
\end{align}
where in the notation of \cite{Das3} the $p$-adic invariant 
$u_{DD}(\mk{b},\ca{O}_n)^{\frac{1}{6}}$ corresponds to the $p$-adic invariant
$u(\alpha,\tau)$. For an explanation concerning the appearance of this $6$-th root see Remark \ref{hausen}. Here the product
on the right hand side of \eqref{lang}
goes over a complete set of representatives of positive integers $\{a\}$ of $(\ZZ/n\ZZ)^{\times}/\laa\epsilon_n\raa$
which are coprime to $pnN$.
When $n=1$, the formula \eqref{lang} holds true and is (up to a $6$-th root extraction) 
equivalent to \eqref{ticle}. However
when $n>1$ the formula \eqref{lang} differs from the formula \eqref{ticle} since the set
$\{a\}$ is strictly smaller than the set $\{\lambda\}$. A complete set of totally positive representatives of 
$\mk{b}^{-1}/\mk{n}(\mk{b}')^{-1}$ is given for example by $\{r\}_{r=1}^{n}$. Note that
$u_D(r\mk{b}',\mk{n})=u_D(\frac{r}{\mk{d}}\mk{b}',\frac{\mk{n}}{\mk{d}})$ 
where $\mk{n}=n\ca{O}_K$ and $(\mk{n},r\mk{b}')=\mk{d}$. Since $\mk{d}$ is
an $\ca{O}_K$-ideal ($\ca{O}_K$ being the maximal order) it is automatically invertible. Therefore
$u_D(r\mk{b}',\mk{n})$ is a primitive $p$-adic invariant of conductor $\frac{\mk{n}}{\mk{d}}$.
Here the word primitive is used 
in the sense of Definition \ref{crocus}. It is thus essential to consider $p$-adic 
invariants which are primitive of conductor $\mk{d}$ for various divisors $\mk{d}|\mk{n}$. 
Let us note that the right hand side of equation $(85)$ in \cite{Das3} is indeed 
the (conjectural) Gross-Stark $p$-unit of $K(\ca{O}_n\infty)$. However, the 
left hand side of equation $(85)$ in \cite{Das3}  is a product of powers of Gross-Stark 
$p$-units for ring class fields of various conductors $\mk{d}|\mk{n}$.
The error is due to the discrepancy between the 
partial zeta-functions associated to ideal classes of the order $\ca{O}_n$, namely
\begin{align}\label{flame}
\zeta(\mk{b}',\ca{O}_n,w_1,s)=\Norm(\mk{b}')^{-s}
\sum_{\ca{O}_n(\infty)^{\times}\bs\{0\neq\mu\in\mk{b}'^{-1}\}}
\frac{w_1(\mu)}{|\Norm_{K/\QQ}(\mu)|^s},\hspace{1.5cm} \Re(s)>1,
\end{align}
and the partial zeta-functions associated to the extension $K(\ca{O}_n\infty)/K$, namely 
\begin{align}\label{flame2}
\zeta_R(\mk{b}',\ca{O}_n,w_1,s)=\Norm(\mk{b}')^{-s}
\sum_{\ca{O}_n(\infty)^{\times}\bs\{0\neq\mu\in\mk{b}'^{-1},(\mu,R)=1\}}
\frac{w_1(\mu)}{|\Norm_{K/\QQ}(\mu)|^s},\hspace{1.5cm} \Re(s)>1,
\end{align}
where $R=\{\mbox{$\nu$ is a place of $K$: $\nu|n$}\}$. When $n>1$ the existence of invertible ideals 
of $\ca{O}_n$ that are not relatively prime to $n$, implies that \eqref{flame2} differs from \eqref{flame}.  
Since the Darmon-Dasgupta invariant $u(\alpha,\tau)$ was constructed from special values of \eqref{flame} rather
than \eqref{flame2} it explains the discrepancy between $u(\alpha,\tau)=u_{DD}(\mk{b},\ca{O}_n)^{\frac{1}{6}}$ 
and the $p$-adic Gross-Sark $p$-unit of $K(\ca{O}_n\infty)$.
\end{Rem}

\appendix
%%\makeatletter
%%\def\@seccntformat#1{Appendix\ \csname the#1\endcsname\quad}
%%\makeatother

%%\renewcommand\thesection{Appendix \Alph{section}}

\input appendix1.tex

\input appendix0.tex

\input appendix2.tex

\bibliographystyle{alpha}
\bibliography{biblio}

\vspace{1cm}
\begin{small}
\noindent {\sc Hugo Chapdelaine, D\'epartement de math\'ematiques et de statistique, Universit\'e Laval,
Qu\'ebec, Canada G1K 7P4} \\
{\sf hugo.chapdelaine@mat.ulaval.ca}
\end{small}

\end{document}

%% file: appendix1.tex
\section{Special values of partial zeta functions and Dasgupta's refinement}\label{appendix1}
Let $F$ be a totally real number field and let $S$ be a finite set of places of $F$
which contains all the infinite places of $F$. Let $\mk{f}$ be an integral ideal of $F$
and let $M=F(\mk{f}\infty)$ be the narrow ray class field of $F$ of conductor $\mk{f}$.
For an ideal $\mk{a}$ of $F$ we denote by $\sigma_{\mk{a}}\in \Gal(M/F)$ the 
Frobenius at $\mk{a}$. For any $\sigma\in \Gal(M/F)$ let
\begin{align}\label{rabbit_a}
\zeta_S(M/F,\sigma,s):=\sum_{\substack{(\mk{a},S)=1\\ \sigma_{\mk{a}}=\sigma}}\frac{1}{\Norm(\mk{a})^s}.
\end{align}
It was proved by Siegel and Klingen that
the special values at negative integers of 
$\zeta_S(\sigma,s)$ are \textit{rational numbers} (see \cite{Kli} and \cite{Sie3}). 
One of the main tool which is used to study these rational values
is the so-called \textit{$q$-expansion principle}
(see \cite{D-R}) which says that for all integral ideals $\mk{b}$ coprime to $\mk{f}$ and all integers $k\geq 1$ one has
\begin{align}\label{bille}
\zeta_R(M/F,\sigma,1-k)-\Norm(\mk{b})^k\zeta_R(M/F,\sigma\sigma_{\mk{b}},1-k)\in\ZZ
\left[\frac{1}{\Norm(\mk{b})}\right].
\end{align}
It follows from Theorem 2.4 of \cite{Co} that the relation \eqref{bille}, when 
applied to all prime ideals $\mk{b}$ of $M$ coprime to $\mk{f}$, implies that
\begin{align}\label{mur}
w_k(M/F)\zeta_R(M/F,\sigma,1-k)\in\ZZ,
\end{align}
where $w_k(M/F)$ is the largest integer, say $e$, for which the abelian group $\Gal(M(\mu_e)/F)$ 
has an exponent which divides $k$. Here $\mu_e$ denotes the group of roots of unity of order $e$.
In particular, $w_1(M/F)$ is just the group of roots of unity of $M$. Thus
the denominator of the rational number $\zeta_R(\sigma,1-k)$ is a divisor of $w_k(M/F)$.
It follows from \eqref{bille} that if $\eta\in T$ then for all integers
$k\geq 1$ one has
\begin{align*}
\zeta_{S,T}(M/F,\sigma,1-k)\in\ZZ\left[\frac{1}{\Norm\eta}\right].
\end{align*} 
In particular, if $T$ contains two primes of different residue characteristics then 
the special value $\zeta_{S,T}(M/F,\sigma,1-k)$ is an integer. Another sufficient condition which guarantees 
the integrality of $\zeta_{S,T}(M/F,\sigma,1-k)$ is to assume the existence of a prime $\eta\in T$ such that
$(l,w_k(M/F))=1$ where $l=\Norm(\eta)$. This explains the raison d'\^etre of Assumption \ref{ass_inter}.

There is another key property of the special values $\zeta_R(M/F,\sigma,1-k)$ which plays 
a key role in the setting of $p$-adic Gross-Stark conjecture. Let $\tau$ be a complex conjugation of $M/F$ and
let $k\geq 1$ be an odd integer. Then
\begin{align}\label{rouet}
\zeta_R(M/F,\sigma\tau,1-k)=-\zeta_R(M/F,\sigma,1-k).
\end{align}
For a proof of \eqref{rouet} see for example p. 12 of \cite{Ch8}. The identity \eqref{rouet} forces
additional restrictions on the type of global $\mk{p}$-units which are predicted by the Strong Gross conjecture.
Let us explain it. If $\sigma:M\hookrightarrow\CC$ is a complex embedding and $u\in U_{\mk{p}}$ one
readily sees from the definition of $U_{\mk{p}}$ that
\begin{align}\label{strong}
u^{\sigma\tau_{\infty}\sigma^{-1}}=u^{-1},
\end{align}  
where $\tau_{\infty}$ is the complex conjugation of $\CC$. In particular, we see from
\eqref{strong} that every complex conjugation of $M$ acts by $-1$ on the multiplicative
group $U_{\mk{p}}$. It thus follows that $U_{\mk{p}}\subseteq L_{CM}$ where $L_{CM}$ stands for 
the largest CM subfield contained in $M$. Therefore if $\tau$ is
a complex conjugation of $M/F$ then $(u_{T}^{\sigma})^{\tau}=(u_{T}^{\sigma})^{-1}$
which is in harmony with the 
fact that 
\begin{align}\label{mouche}
\zeta_{R,T}(M/F,\sigma\tau,0)=-\zeta_{R,T}(M/F,\sigma,0),
\end{align}
where \eqref{mouche} follows from \eqref{rouet} and the definition of $\zeta_{R,T}(M/F,\sigma,s)$.

It is possible to use \eqref{pierre2} of Conjecture \ref{Gross_strong} to get $\mk{p}$-adic information about $u_T$. 
For example, if
the set of finite places of $S$ consists \textit{exactly} of the primes which divide 
$\mk{f}\mk{p}$ then one may consider the tower of fields $L_{n}=F(\mk{f}\mk{p}^n\infty)$ with $n$ increasing. 
In this case, class field theory provides an isomorphism
\begin{align}\label{bouteille}
rec_{\mk{p}}:F_{\mk{p}}^{\times}/\wh{E}_{\mk{p}}(\mk{f})\stackrel{\simeq}{\rightarrow} \Gal(L_{\infty}/M),
\end{align}
where $L_{\infty}=\bigcup_{n} L_n$, $E_{\mk{p}}(\mk{f})$ denotes the group 
of totally positive $\mk{p}$-units of $F$ which are
congruent to $1$ modulo $\mk{f}$ and $\wh{E}_{\mk{p}}(\mk{f})$ denotes its closure 
in $F_{\mk{p}}^{\times}$. Let 
\begin{align*}
g_T:=\displaystyle\lim_{\substack{\longleftarrow\\ n}} 
rec_{\mk{p}}^{L_n}(u_T)\in \Gal(L_{\infty}/M),
\end{align*} 
where the transition maps are given by the restrictions. Then \eqref{pierre2} 
predicts that 
\begin{align}\label{price}
rec_{\mk{p}}^{-1}(g_T)\equiv u_T\pmod{\wh{E}_{\mk{p}}(\mk{f})}.
\end{align}
In general the $\ZZ_p$-rank
of $\wh{E}_{\mk{p}}(\mk{f})$ will be larger than one and therefore \eqref{price} does not
provide an exact formula for $u_T$. 
It is explained in \cite{Das3} that  by expanding the set $S$ in an appropriate way one
may gain more $\mk{p}$-adic information about $u_T$. In Section 5.4 of \cite{Das3}, Dasgupta shows that by 
repeating this process indefinitely one can specify $u_T$ in $F_{\mk{p}}^{\times}$ to any specified 
degree of $\mk{p}$-adic accuracy. However there is a lack of explicitness in this process which makes it not very suitable for numerical
computations. Also, even from a theoretical point of view, this process
is not completely satisfactory since it is rather indirect. 

We would like to mention one more key result which follows from Lemma 5.17 of \cite{Das3} and
\eqref{mouche} which is not stated explicitly in \cite{Das3}.
\begin{Prop}\label{candle}(Dasgupta)
Assume that $M$ is linearly disjoint from $F(\zeta_{p^m})$ for all $m\geq 1$. Let $\tau_{\infty}$ be a complex conjugation of
$M/F$ and suppose that $\mk{c}$ is an $\ca{O}_K$-ideal of $F$ such that $\sigma_{\mk{c}}=\tau_{\infty}$. Then 
\begin{align}\label{roulant}
u_{D}(\mk{a}\mk{c},\mk{f})=u_{D}(\mk{a},\mk{f})^{-1}.
\end{align} 
\end{Prop}

Even when $M$ is not linearly disjoint from $F(\zeta_{p^m})$
it is expected that \eqref{roulant} holds. 
The author does not know of a proof of \eqref{roulant} which avoids the artificial
process of enlarging the set of places $S$. It would be quite interesting to find a more
direct proof of \eqref{roulant}.

%% file: appendix0.tex
\section{A norm formula from $u_C$ to $u_{DD}$}\label{goeland}
We now discuss the compatibility of our $p$-adic invariant with the one constructed by Darmon and Dasgupta.
One can relate the modular unit considered in \cite{Ch1} with the modular
units used in \cite{Dar-Das}.
We have for any positive integer $N$ the identity
\begin{align}\label{norm}
\prod_{j=1}^{N-1}g_{(\frac{j}{N},0)}(N\tau)^{12}=
\zeta\frac{\Delta(\tau)}{\Delta(N\tau)},
\end{align}
for some $\zeta\in\mu_N$ where 
$$
g_{(\frac{j}{N},0)}(\tau)=q_{\tau}^{\frac{1}{2}\wt{B}_2(\frac{j}{N})}
(1-q_{\frac{j\tau}{N}})\prod_{n \geq 1} (1-q_{\tau}^nq_{\frac{j\tau}{N}})(1-q_{\tau}^nq_{\frac{-j\tau}{N}}),
$$ 
and 
$$
\Delta(\tau)=q\prod_{n\geq 1} (1-q^n)^{24},
$$ 
where $q=e^{2\pi i\tau}$, $\tau\in\ca{H}=\{x+iy\in\CC:y>0\}$. Here 
$\wt{B}_2(x)=B_2(\{x\})$ where $B_2(x)=x^2-x+\frac{1}{6}$ and $0\leq\{x\}<1$ denotes the fractional part of $x$. 
We note that
\begin{align*}
g_{0,0)}(\tau)^{12}=\Delta(\tau).
\end{align*} 
We can thus rewrite \eqref{norm} as
\begin{align}
\prod_{j=0}^{N-1}g_{(\frac{j}{N},0)}(N\tau)^{12}=
\zeta\Delta(\tau),
\end{align}
for some $\zeta\in\mu_N$.

As in \cite{Dar-Das}, choose a divisor $\delta=\sum_{d|N} n_{d}[d]\in Div(N)$ such that
$$
\sum_{d|N}n_{d}d=0 \s\s \textrm{and}\s\s \sum_{d|N} n_{d}=0.
$$ 
To such a divisor Darmon and Dasgupta associate the modular unit
\begin{align}\label{beate}
\alpha_{\delta}(\tau):=\prod_{d|N}\Delta(d\tau)^{n_{d}},
\end{align}
which is $\Gamma_0(N)$-invariant. More generally, to a divisor 
$\delta=\sum_{d|N,r\in\ZZ/f\ZZ}n(d,r)[d,r]\in Div_f(N)$ we associate the
modular unit
\begin{align}\label{siegel_units}
\beta_{\delta}(z):=\prod_{d|N,r\in\ZZ/f\ZZ}g_{(\frac{r}{f},0)}(fd\tau)^{12n(d,r)}.
\end{align} 
If we set $\delta'=\sum_{d|N,r\in\ZZ/f\ZZ}n(d,r)[d,r]\in Div_f(N)$
with $n(d,r)=n_{d}$ for all $r\in\ZZ/f\ZZ$, then we readily see that
$\delta'$ is a good divisor with respect to any prime number $p$.
Using equation (\ref{norm}) with $N=f$ we find
\begin{align}\label{hermana}
\beta_{\delta'}(\tau)&=\prod_{d|f,r\in\ZZ/f\ZZ}g_{(\frac{r}{f},0)}(fd\tau)^{12n(d,r)} \notag \\  \notag
&=\zeta\prod_{d}\Delta(d\tau)^{n_{d}}\\
&=\zeta\alpha_{\delta}(\tau),
\end{align}
for some $\zeta\in\mu_f$. 

Let $p$ be a prime number inert in $K$. 
Assume that $(p,f,N)$ satisfy the usual assumptions and let $\tau\in H^{\ca{O}_f}(\mk{N},f)$
where $\ca{O}_f$ stands for the unique order of conductor $f$ of $K$. By definition of 
$\delta'$ and $u_{C,\delta'}(1,\tau)$ we find that
\begin{align*}
u_{C,\delta'}(1,\tau)=\prod_{r=0}^{f-1} u_{C,\delta}(r,\tau).
\end{align*}
Using \eqref{hermana} we deduce directly from the definitions of the $p$-adic invariants $u_C$ 
and $u_{DD}$ (see \cite{Dar-Das} for the definition of $u_{DD}$) that
\begin{align*}
u_{C,\delta'}(1,\tau)=u_{DD,\delta}(\tau).
\end{align*}
Therefore it follows that
\begin{align}\label{trio}
\prod_{r=0}^{f-1} u_{C,\delta}(r,\tau)=u_{DD,\delta}(\tau).
\end{align}
Conjecture 2.14 of \cite{Dar-Das} predicts that $u_{DD,\delta}(\tau)\in K(\ca{O}_f\infty)$. 
Conjecture \ref{u_H_conj} predicts that for each $r\in\ZZ/f\ZZ$, $u_{C,\delta}(r,\tau)$ is a strong $p$-unit 
of $K(f\ca{O}_f\infty)$. Now let us make the following additional assumption:
\begin{Assumption}\label{leo}
Let $\ca{O}_K=\ZZ+\omega\ZZ$ and
let $\epsilon$ be the generator of $\ca{O}_{K}(f\infty)^{\times}$ such that $\epsilon>1$. Assume that
$\epsilon=u+vf\omega$ with $u,v\in\ZZ$ and $(v,f)=1$.
\end{Assumption}
Note that $u\equiv 1\pmod{f}$ and that $\epsilon\in\ca{O}_f^{\times}$. Then we have the following proposition:
\begin{Prop}
Conjecture \ref{u_H_conj} and Assumption \ref{leo} imply that
for each $d|f$ one has
\begin{align}\label{gogo1}
\prod_{\substack{0\leq r\leq f-1\\ (r,f)=d}} u_{C,\delta}(r,\tau)\in K(\ca{O}_f\infty),
\end{align}
where $K(\ca{O}_f\infty)$ stands for the narrow ring class field of $K$ of conductor $f$. 
In particular, one has
\begin{align}\label{gogo2}
\prod_{r=0}^{f-1} u_{C,\delta}(r,\tau)\in K(\ca{O}_f\infty).
\end{align}
\end{Prop}
\begin{proof}
The generalized ideal class group $I_{\ca{O}_f}(1)/P_{\ca{O}_f}(\infty)$ corresponds from
class field theory to the narrow ring class field extension $K(\ca{O}_f\infty)$. Here
\begin{align*}
P_{\ca{O}_f}(\infty)=
\left\{\frac{\alpha}{\beta}\ca{O}_f:\alpha,\beta\in\ca{O}_f,\frac{\alpha}{\beta}\gg 0\right\},
\end{align*}
and two ideals $\mk{a},\mk{b}\in I_{\ca{O}_f}(1)$ are equivalent modulo $P_{\ca{O}_f}(\infty)$ if and
on if there exists $\lambda\ca{O}_f\in P_{\ca{O}_f}(\infty)$ such that $\lambda\mk{a}=\mk{b}$. 
Similarly, one may define $I_{\ca{O}_f}(f)/P_{\ca{O}_f}(f\infty)$ where 
\begin{align*}
P_{\ca{O}_f}(f\infty)=
\left\{\frac{\alpha}{\beta}\ca{O}_f:\alpha,\beta\in\ca{O}_f,(\alpha\ca{O}_f,f\ca{O}_f)=
(\beta\ca{O}_f,f\ca{O}_f)=\ca{O}_f,\frac{\alpha}{\beta}\gg 0\right\}.
\end{align*}
One can show that the natural map
\begin{align}\label{typhoon}
I_{\ca{O}_f}(f)/P_{\ca{O}_f}(f\infty)&\stackrel{\simeq}{\rightarrow} I_{\ca{O}_f}(1)/P_{\ca{O}_f}(\infty)\\
                              [\mk{a}]&\mapsto[\mk{a}], \notag
\end{align}
is an isomorphism. In order to show that \eqref{gogo1} lies in $K(\ca{O}_f\infty)$ 
it is enough to show that
$ker(res)$ acts trivially on the left hand side of \eqref{gogo1} where 
$$
res: \Gal(K(f\ca{O}_{f}\infty)/K)\rightarrow \Gal(K(\ca{O}_f\infty)/K).
$$
Using class field theory and \eqref{typhoon} one has that the restriction map above corresponds to the map
\begin{align*}
\pi:I_{\ca{O}_{f}}(f)/P_{\ca{O}_{f},1}(f\infty)&\rightarrow I_{\ca{O}_{f}}(f)/P_{\ca{O}_f}(f\infty)\\
                  [\mk{a}]&\mapsto[\mk{a}].
\end{align*}
Therefore the subgroup $ker(res)$ corresponds to 
$\left(P_{\ca{O}_f}(f\infty)\cap I_{\ca{O}_f}(f\infty)\right)/P_{\ca{O}_{f},1}(f\infty)$. 
Let $\Theta:I_{\ca{O}_f}(f)\rightarrow I_{\ca{O}_{f^2}}(f)$ be the map given 
by $\mk{a}\mapsto\mk{a}\cap\ca{O}_{f^2}$. Using formula \eqref{oparleur} we see 
that in order to show that \eqref{gogo1} is in $K(\ca{O}_f\infty)$ it is enough to show that 
the collection of ideal classes
\begin{align}\label{feynman}
\{\wt{r}I_{\tau^*}\pmod{\sim_f}:1\leq r\leq f\s\mbox{and}\s (r,f)=d\}
\end{align} 
is stable as a set under the action of $\Theta(P_{\ca{O}_f}(f\infty)\cap I_{\ca{O}_f}(f\infty))$. 
A direct computation shows that
\begin{align*}
P_{\ca{O}_f}(f\infty)\cap I_{\ca{O}_f}(f\infty)=
\left\{\alpha\ca{O}_{f}:\alpha=a+bf\omega,
a,b\in\ZZ,(a,f)=1\s\mbox{and}\s\alpha\gg 0\right\}.
\end{align*}
We note that if $\alpha=a+bf\omega$ and $b\equiv 0\mod{f}$ then 
$a+bf\omega\in\ca{O}_{f^2}$. Moreover, one may check that
$\Theta(\alpha\ca{O}_f)=\alpha\ca{O}_{f^2}$ and therefore $\Theta(\alpha\ca{O}_f)$ 
is again a principal ideal. Now let $\alpha\ca{O}_f\in P_{\ca{O}_f}(f\infty)\cap I_{\ca{O}_f}(f\infty)$.
Then using Assumption \ref{leo} one sees that it is always possible to multiply $\alpha$ by a suitable
power of $\epsilon$ so that $\epsilon^m\alpha:=\alpha'=a'+b'f\omega$ with $b'\equiv 0\pmod{f}$. 
It thus follows that
\begin{align*}
\Theta(P_{\ca{O}_f}(f\infty)\cap I_{\ca{O}_f}(f\infty))=
\left\{\alpha\ca{O}_{f^2}:\alpha=a+bf^2\omega,
a,b\in\ZZ,(a,f)=1\s\mbox{and}\s\alpha\gg 0\right\}.
\end{align*}

Now let $\epsilon$ be as in Assumption \ref{leo} and note that
$\epsilon'=\epsilon^{f}=u'+v'f^2\omega$ for $u',v'\in\ZZ$, $u'\equiv 1\pmod{f}$
and $(v',f)=1$. Now let $\alpha\ca{O}_{f^2}\in\Theta(P_{\ca{O}_f}(f\infty)\cap I_{\ca{O}_f}(f\infty))$
where $\alpha=a+bf^2\omega$. In a similar way to what we did previously we see that we may multiply
$\alpha$ by a suitable
power of $\epsilon'$ so that $(\epsilon')^{m'}\alpha:=\alpha'=a'+b'f^2\omega$ with $b'\equiv 0\pmod{f}$.
Now note that the multiplication by $\alpha'$ only reshuffles the ideal classes in \eqref{feynman}. Therefore 
the Shimura reciprocity law of Conjecture \ref{u_H_conj} predicts that 
$\prod_{\substack{1\leq r\leq f\\ (r,f)=d}}u_{C,\delta}(r,\tau)$ lies in $K(\ca{O}_f\infty)$. This concludes the
proof. \fin
\end{proof}
\begin{Rem}
The author thinks that Conjecture \ref{u_H_conj} alone (i.e. without the help of Assumption \ref{leo}) should
imply relation \eqref{gogo2}. 
\end{Rem} 
\begin{Prop}\label{ano_way}
The identity \eqref{zoller} in Proposition \ref{pointe} 
implies \ref{trio} with the following parameters: $\tau\in H^{\ca{O}_f}(\mk{N},f)$,
$\mk{b}^{-1}=\Lambda_{\frac{1}{N\tau^{\sigma}}}$, $\mk{g}=\ca{O}_f$, $(\mk{b}')^{-1}=
\Lambda_{(\tau^*)^{\sigma}}$, 
$\mk{g}'=f\ca{O}_{(\tau^*)^{\sigma}}$ where $\tau^*=\frac{-1}{fN\tau}$. 
\end{Prop}
\begin{proof}
We suppress the divisor $\delta$ for the notation. First note that
$\End_K(\mk{b})=\ca{O}_f$ and that $\End_K(\mk{b}')=\ca{O}_{f^2}$. Moreover, if 
$Q_{\tau}(x,y)=Ax^2+Bxy+Cy^2$ ($(A,f)=1$) then we have $B^2-4AC=f^2d_K$. A direct computation shows that 
\begin{align*}
Q_{\tau^*}(x,y)=\sign(C)\left(CNf^2x^2-Bfxy+\frac{A}{N}y^2\right),
\end{align*}
so that
\begin{align}\label{rousse}
\tau^*=\frac{B+f\sqrt{d_K}}{fCN}.
\end{align}
On one hand we have
\begin{align}\label{rock_1}
u_{D}(\mk{b},\ca{O}_f)^{12}=u_{D}(I_{\frac{1}{N\tau}},\ca{O}_f)^{12}=u_{DD}(\tau),
\end{align}
where the first equality follows from the identity $\mk{b}=I_{\frac{1}{N\tau}}$. 
For the second equality see the discussion in Remark \ref{implication1}. 

Now for every $r\in\ZZ/f\ZZ$ we have that
\begin{align}\label{rock_2}
u_{D}(\wt{r}\mk{b}',\mk{g}')^{12}=u_{D}(\wt{r}I_{\tau^*},f\ca{O}_{\tau^*})^{12}.
\end{align}
Using \eqref{rock_2} in \eqref{zoller} of Proposition \ref{pointe} we find that
\begin{align}\label{rock_3}
\prod_{r=1}^{f} u_{D}(\wt{r}\mk{b}',\mk{g}')^{12h_r}
=u_{D}(\mk{b},\ca{O}_f)^{12e},
\end{align}
where
\begin{align*}
e=[\Gamma_{\mk{b}}(\mk{g}):\Gamma_{\mk{b}'}(\mk{g}')]\s\s\s\mbox{and}\s\s\s 
h_r=[\Gamma_{\wt{r}\mk{b}'}(\mk{g}'):\Gamma_{\mk{b}'}(\mk{g}')].
\end{align*}
Now note that $\{\wt{r}\}_{\wt{r}=1}^f$
is a complete set of totally positive representatives of $\mk{b}^{-1}\mk{g}/(\mk{b}')^{-1}\mk{g}'$. 
Using Proposition \ref{klasse} we find that 
\begin{align}\label{rock_5}
u_{D}(\wt{r}I_{\tau^*},f\ca{O}_{\tau^*})^{12\nu(r,\tau)^{-1}}=u_{C}(r,\tau),
\end{align}
\begin{align*}
\nu(r,\tau)={\left([\epsilon(\eta_{\tau}):\epsilon(\eta_{\tau^*})]
[\Gamma_{\wt{r}\mk{b}'}(\mk{g}'):\epsilon(\laa\eta_{\tau^*}\raa)]\right)}^{-1}\in\QQ_{>0}.
\end{align*}
For the definitions of $\epsilon(\eta_{\tau})$ and
$\epsilon(\eta_{\tau^*})$ (see Definition \ref{matrix_def}). 
Using \eqref{rousse}, direct computations show that 
\begin{align*}
\laa\epsilon(\eta_{\tau})\raa=\laa\epsilon(\eta_{\tau^*})\raa=
\ca{O}_{f^2}(\infty)^{\times}\cap\ca{O}_K(f\infty)^{\times}.
\end{align*}
and that 
\begin{align*}
\ca{O}_{f^2}(\infty)^{\times}\cap\ca{O}_K(f\infty)^{\times}=
\Gamma_{\mk{b}'}(\mk{g'})=\Gamma_{\mk{b}}(\mk{g}).
\end{align*}
In particular, we have
\begin{align}\label{model}
\nu(r,\tau)=[\Gamma_{\mk{b}'}(\mk{g}'):\Gamma_{\wt{r}\mk{b}'}(\mk{g}')]=h_r^{-1}\s\s\s\mbox{and}\s\s\s
e=1.
\end{align}
Using \eqref{model} in \eqref{rock_3} we find that
\begin{align}\label{rock_4}
\prod_{r=1}^{f} u_{D}(\wt{r}\mk{b}',\mk{g}')^{12\nu(r,\tau)^{-1}}=u_D(\mk{b},\mk{f})^{12}.
\end{align}
Finally, combining \eqref{rock_4} with \eqref{rock_5} and \eqref{rock_1} we find that
\begin{align*}
\prod_{r=1}^{f} u_{C}(r,\tau)=u_{DD}(\tau).
\end{align*}
This concludes the proof. \fin
\end{proof}

%% file: appendix2.tex
\section{Additional proofs}\label{appendix2}
\subsection{\bf Proof of Proposition \ref{hard_prop}} 
Let $(r,\tau),(r',\tau')\in(\ZZ/f\ZZ\times H^{\ca{O}}(\mk{N}))/\sim_f$ and
let $I_{\tau}=A\Lambda_{\tau}$, $I_{\tau'}=A'\Lambda_{\tau'}$ where
$Q_{\tau}(x,y)=Ax^2+Bxy+Cy^2$ and $Q_{\tau'}(x,y)=A'x^2+B'xy+C'y^2$. We have
$\End_K(\Lambda_{\tau})=\End_K(\Lambda_{\tau'})=\ca{O}$ where 
$\disc(\ca{O})=D=B^2-4AC=B'^2-4A'C'$. 

We first show that the map $\wt{\psi}$ is well defined, i.e., 
if $(r,\tau)\sim_f(r',\tau')$ then $\psi(r,\tau)\sim_f\psi(r',\tau')$. 
So assume that there exists a matrix 
$\gamma=\M{a}{b}{c}{d}\in\Gamma_0(fN)$ such that $\gamma\tau=\tau'$ and $d^{-1}r\equiv r'\pmod{f}$. 

We want to show that $[\wt{r}A\Lambda_{\tau},\wt{r}A\Lambda_{N\tau}]
\sim_f [\wt{r}'A'\Lambda_{\tau}',\wt{r}'A'\Lambda_{N\tau}']$. We have
\begin{align}\label{priere1}
\Lambda_{\tau}=(c\tau+d)\Lambda_{\tau'},
\end{align}
and
\begin{align}\label{priere2}
\M{a}{b}{c}{d}\M{\tau}{\tau^{\sigma}}{1}{1}=
\M{c\tau+d}{0}{0}{c\tau^{\sigma}+d}\M{\tau'}{(\tau')^{\sigma}}{1}{1}.
\end{align}
Taking the determinant of \eqref{priere2} we find
\begin{align}\label{girafe}
\frac{(\tau-\tau^{\sigma})}{(\tau'-\tau'^{\sigma})}=(c\tau+d)(c\tau^{\sigma}+d).
\end{align}
Since $\tau=\frac{-B+\sqrt{D}}{2A}$ and $\tau'=\frac{-B'+\sqrt{D}}{2A'}$ we find
$(\tau-\tau^{\sigma})=\frac{\sqrt{D}}{A}$ 
and $(\tau'-\tau'^{\sigma})=\frac{\sqrt{D}}{A'}$. Substituting in \eqref{girafe}
we find
\begin{align*}
\frac{A'}{A}(c\tau+d)^{-1}=(c\tau^{\sigma}+d).
\end{align*}
From \eqref{priere1} we may deduce that 
$\frac{A'}{A}(c\tau+d)^{-1}(A\Lambda_{\tau})=A'\Lambda_{\tau'}$ and thus
$\frac{\wt{r}'}{\wt{r}}(c\tau^{\sigma}+d)(\wt{r}A\Lambda_{\tau})=\wt{r}'A'\Lambda_{\tau'}$.
Since $c\equiv 0\pmod{f}$ and $\frac{\wt{r}'}{\wt{r}}d\equiv 1\pmod{f}$ we see
that 
$$
\frac{\wt{r}'}{\wt{r}}(c\tau^{\sigma}+d)\in(\wt{r}A\Lambda_{\tau})^{-1}f+1=\frac{1}{\wt{r}}\Lambda_{\tau^{\sigma}}f+1.
$$ 
Moreover, since $N|c$ we also have
\begin{align*}
\frac{\wt{r}'}{\wt{r}}(c\tau^{\sigma}+d)(\wt{r}A\Lambda_{N\tau})=\wt{r}'A'\Lambda_{N\tau'}.
\end{align*}
It thus follows that $\psi(r,\tau)\sim_f\psi(r',\tau')$.

Let us now show now that the map $\psi$ is injective. 
Assume that $\psi(r,\tau)\sim_f\psi(r',\tau')$, i.e., 
$$
[\wt{r}A\Lambda_{\tau},\wt{r}A\Lambda_{N\tau}]
\sim_f[\wt{r}'A'\Lambda_{\tau'},\wt{r}'A'\Lambda_{N\tau'}].
$$ 
Then
there exists a $\lambda\in (\wt{r}A\Lambda_{\tau})^{-1}f+1=\frac{1}{\wt{r}}\Lambda_{\tau^{\sigma}}f+1$ such that 
\begin{align}\label{fraise1}
\lambda\wt{r}A\Lambda_{\tau}=\wt{r}'A'\Lambda_{\tau'}\s\s\s\mbox{and}\s\s\s
\lambda\wt{r}A\Lambda_{N\tau}=\wt{r}'A'\Lambda_{N\tau'}.
\end{align}
We note that $A'$ is the smallest integer such that $A'\Lambda_{\tau'}\subseteq\ca{O}$ where
$\ca{O}=\ca{O}_{\tau}=\ca{O}_{\tau'}$. Write $\lambda=\frac{1}{\wt{r}}(c\tau^{\sigma}+d)$ with
$c,d\in\ZZ$, $c\equiv 0\pmod{f}$ and $d\equiv\frac{\wt{r}}{\wt{r}'}\pmod{f}$ and let $(c,d)=N$. We claim that $N=1$.
If not then
\begin{align*}
\frac{1}{N}(c\tau^{\sigma}+d)A\Lambda_{\tau}\subseteq\ca{O}.
\end{align*}  
Therefore we would have $\frac{A'}{N}\Lambda_{\tau'}\subseteq\ca{O}$ which contradicts the minimality
of $A'$. Since $(c,d)=1$ there exists $a,b\in\ZZ$
such that $ad-bc=1$. Note that $\gamma=\M{a}{b}{c}{d}\in\Gamma_0(f)$. Let $\tau''=\gamma\tau$. Then we have 
\begin{align}\label{poisson}
(c\tau+d)\Lambda_{\tau''}=\Lambda_{\tau}.
\end{align}
We let $Q_{\tau''}(x,y)=A''x^2+B''xy+C''y^2$. We have
\begin{align*}
\frac{A''}{A}=(c\tau^{\sigma}+d)(c\tau+d).
\end{align*}
Therefore from \eqref{poisson} we deduce that
\begin{align}\label{fraise2}
A''\Lambda_{\tau''}=(c\tau^{\sigma}+d)A\Lambda_{\tau}.
\end{align}
Combining \eqref{fraise1} with \eqref{fraise2} we get
\begin{align*}
A''\Lambda_{\tau''}=A'\Lambda_{\tau'}.
\end{align*}
From this we deduce that $A''=A'$ and $\tau''=(\tau'+n)$ for some $n\in\ZZ$. From this
it follows that
\begin{align*}
\M{1}{-n}{0}{1}\M{a}{b}{c}{d}\V{\tau}{1}=(c\tau+d)\V{\tau'}{1}.
\end{align*}
Since $\M{1}{-n}{0}{1}\M{a}{b}{c}{d}\in\Gamma_0(f)$. We thus have shown that
$(r,\tau)\sim_f (r',\tau')$ if and only if $\psi(r,\tau)\sim_f\psi(r',\tau')$.
From this it follows that the map $\psi$ is well defined and injective. It remains to 
show that $\psi$ is surjective.

Let $\mk{a}$ be an arbitrary integral invertible $\ca{O}_n$-ideal. For a vector
$\V{v_1}{v_2}$ we let 
$$
\left\laa\V{v_1}{v_2}\right\raa=\ZZ v_1+\ZZ v_2,
$$ 
be the $\ZZ$-lattice generated by $v_1$ and $v_2$. We can always find 
$\gamma=\M{a}{b}{c}{d}\in M_2(\ZZ)$ such that
\begin{align*}
\left\laa\M{a}{b}{c}{d}\V{n\omega_1}{1}\right\raa=\mk{a}.
\end{align*}
Let $(a,c)=m$ and $a'=a/m$, $c'=c/m$. We thus have $(a',c')=1$. 
Let $u,v\in\ZZ$ be such that $-ua'-vc'=1$. Then
\begin{align*}
\left\laa\M{A}{B}{0}{D}\V{n\omega_1}{1}\right\raa=\mk{a}.
\end{align*}
where 
\begin{align*}
\M{A}{B}{0}{D}=\M{u}{v}{c'}{-a'}\M{a}{b}{c}{d}.
\end{align*}
We thus have $\mk{a}=\ZZ(An\omega+B)+\ZZ D$. So $\mk{a}=D(\ZZ\frac{An\omega+B}{D}+\ZZ)$. 
Let $(A,B,D)=r$ and let $\frac{A}{r}=A'$, $\frac{B}{r}=B'$ and $\frac{C}{r}=C'$.
Without loss of generality we may assume that and $D',A'>0$. 
Now set $\tau=\frac{A'n\omega+B'}{D'}$. 
Note that $\tau>\tau^{\sigma}$ so that $(r,\tau)\in\ZZ/f\ZZ\times H^{\ca{O}_{\tau}}(\mk{N})$.
By construction we have $\wt{\psi}(r,\tau)=[\mk{a}]$. Since $\mk{a}$ was 
arbitrary it follows that $\wt{\psi}$ is surjective. \fin 

\subsection{\bf Proof of Proposition \ref{lilas}} 
Let $\mk{a},\mk{a}'\in I_{\ca{O}}(\mk{f}n)$ and
write $\mk{a}=rA_{\tau}\Lambda_{\tau}$
with $r\in\ZZ_{\geq 1}$ and $\tau\in K\bs\QQ$. By assumption we have
that $(rr'A_{\tau}A_{\tau'},fn)=1$ and that $\ca{O}_{\tau}=\ca{O}_{\tau'}=\ca{O}_m$. 
Since $(rA_{\tau},n)=1$ 
we have $\Theta(\mk{a})=\mk{a}\cap\ca{O}'=rA_{\tau}\Lambda_{n\tau}$. Note that $\End_K(\Lambda_{n\tau})=\ca{O}'$, so that 
indeed, the image of $\Theta$ lies in $I_{\ca{O}'}(\mk{f}')$. Let us show that $\Theta$ is multiplicative. Since
$\Theta(\mk{a}\mk{a}')\supseteq\Theta(\mk{a})\Theta(\mk{a}')$ the multiplicativity of $\Theta$ is equivalent to 
\begin{align}\label{tablis1}
[\ca{O}':\Theta(\mk{a})\Theta(\mk{a}')]=[\ca{O}':\Theta(\mk{a}\mk{a}')].
\end{align}
Since $\mk{a}$ is $\ca{O}$-invertible we have 
\begin{align}\label{tablis2}
[\ca{O}:\mk{a}\mk{a}']=[\ca{O}:\mk{a}][\ca{O}:\mk{a}'].
\end{align}
Now for an arbitrary ideal $\mk{b}\in I_{\ca{O}}(\mk{f}n)$ one has
\begin{align}\label{tablis3}
[\ca{O}:\mk{b}]=[\ca{O}':\Theta(\mk{b})].
\end{align}
Now \eqref{tablis1} follows from \eqref{tablis3} and \eqref{tablis2}. This shows $(i)$.
Let us show $(ii)$. Let us assume that $\mk{a}\sim_{fn}\mk{a}'$, i.e., 
$\lambda\mk{a}=\mk{a}'$ where $\lambda\in\mk{a}^{-1}fn+1$. Since $\mk{a}^{-1}=
\frac{1}{r}\Lambda_{\tau^{\sigma}}$ this implies that 
$\lambda\in\frac{fn}{r}\Lambda_{\tau^\sigma}+1$. Since $(rA_{\tau},fn)=1$ we can always
find a positive integer $b$ such that $b\lambda\in\ca{O}$ and $b\equiv 1\pmod{fn}$ so that
$\mk{a}'\sim_{fn}b\mk{a}'$. Therefore, without loss of generality we may assume that 
$\lambda\in\ca{O}\cap\left(\frac{fn}{r}\Lambda_{\tau^\sigma}+1\right)=(1+fn\ca{O})$. 
Since $\Theta(\lambda\ca{O})=\lambda\ca{O}'$ and $\Theta(\mk{a}')=
\Theta(\lambda\ca{O})\Theta(\mk{a})$ we find that
$\lambda rA_{\tau}\Lambda_{n\tau}=\Theta(\mk{a}')$. Since 
$$
\lambda\in 1+fn\ca{O}\subseteq\frac{f}{r}\Lambda_{n\tau}+1,
$$
we find that $\Theta(\mk{a})\sim_{f}\Theta(\mk{a}')$. We thus have proved the implication 
$\mk{a}\sim_f\mk{a}'\Longrightarrow\Theta(\mk{a})\sim_f\Theta(\mk{a}')$. Therefore the
map $\Theta$ gives rise to a  
map $\wt{\Theta}:C_{\ca{O}}(n\mk{f})\rightarrow C_{\ca{O}'}(\mk{f}')$. It remains to
show that the map $\wt{\Theta}$ is onto. So let $[\mk{b}]\in C_{\ca{O}'}(\mk{f}')$. Then
we can write $\mk{b}$ as $\mk{b}=rA_{\tau}\Lambda_{\tau}$ with $r,A_{\tau}\in\ZZ_{\geq 1}$,
$(rA_{\tau},f)=1$ and $\tau\in K\bs\QQ$ is such that $\ca{O}_{\tau}=\ca{O}'$. We claim that we can always find a matrix
$\gamma\in\Gamma_1(f)$ such that $\gamma\tau=\tau'$ and $(A_{\tau'},fn)=1$. 
Let $Q_{\tau'}(x,y)=A'x^2+B'xy+C'y^2$ so that $\tau'=\frac{-B'+mn\sqrt{d}_K}{A'}$. 
Set $\rho=\frac{-B'+m\sqrt{d}_K}{A'}$ and note that
$\ca{O}_{\rho}=\ca{O}_m$ and that $rA_{\tau'}\Lambda_{\rho}$ is integral. Since
$rA_{\tau}\Lambda_{\tau}\sim_f rA_{\tau'}\Lambda_{\tau'}$ and 
$\Theta\left(rA_{\tau'}\Lambda_{\rho}\right)=rA_{\tau'}\Lambda_{\tau'}$,
we see that $\Theta$ is onto. We leave the proof of the existence of the matrix $\gamma$ to the reader. \fin
\subsection{\bf Proof of Corollary \ref{tricky2}} First note that for $x\in\RR$ one has
\begin{align*}
\wt{B}_1(x)-B_1^*(x)=
\left\{
\begin{array}{rcc}
0 & \mbox{if} & 0<x<1 \\
-\frac{1}{2} & \mbox{if} & x\in\ZZ.
\end{array}
\right.
\end{align*}
Moreover, for $x\in\RR$ one has that
\begin{align}\label{tonton}
\wt{B}_1(-x)=-\wt{B}_1(x)
\end{align}
Let
\begin{align*}
y_1(h)=\frac{a}{\frac{c}{fd}}\left(h+\frac{j}{f}+\frac{v}{p^n}\right)
-\frac{fdu}{p^n},\s\s\s
\s\s\s y_2(h)=\frac{1}{\frac{c}{fd}}\left(h+\frac{j}{f}+\frac{v}{p^n}\right),\s\s\s\s
y_2'(h)=ly_2(h),
\end{align*}
be the arguments appearing the $\wt{B}_1$'s of \eqref{libellule2}.

Comparing \eqref{libellule2} with \eqref{tempou} we readily see that the two expressions
are the same except that in \eqref{libellule2} one evaluates $\wt{B}_1$ instead of $B_1^*$.
First note that if $f>1$ and $j\neq 0$ then for all $1\leq h\leq\frac{c}{fd}$, one has
$y_1(h),y_2(h),y_2'(h)\in\QQ\bs\ZZ$ and therefore $\xi(U_{u,v,n})=0$. So it is enough to prove \eqref{terreur} when $j=0$. 
Note that since $U_{u,v,n}$ is a ball contained in $\XX$ one has that $n\geq 1$ and $(u,v)\not\equiv(0,0)\pmod{p}$.
Note in particular that for a fixed value of $h$ one cannot have both $y_1(h)$ and $y_2(h)$ in $\ZZ$. We also note
that if $y_2(h)\in\ZZ$ then $y_2'(h)\in\ZZ$. Let
\begin{align*}
R^*(y_1(h),y_2(h),y_2'(h),u,v):=
B_1^*\left(y_1(h)\right)\left[\left(\frac{N}{d}\right)B_1^*(y_2(h))
-\left(\frac{N}{d'}\right)B_1^*(y_2'(h))\right],
\end{align*} 
and similarly define $\wt{R}(y_1(h),y_2(h),y_2'(h),u,v)$ by replacing $B_1^*$ by $\wt{B}_1$
in the expression above. In order to show that $\xi(U_{u,v,n})+\xi(U_{-u,-v,n})=0$ it is enough to 
show that 
\begin{align}\label{tile}
\sum_{\substack{1\leq h\leq\frac{c}{fd}\\ y_1(h)\in\ZZ\vee y_2(h)\in\ZZ\vee y_2'(h)\in\ZZ}}
[R^*(y_1(h),y_2(h),y_2'(h),u,v)+R^*(y_1(h),y_2(h),y_2'(h),-u,-v)],
\end{align}
is equal to zero and similarly if one replaces $R^*$ by $\wt{R}$. We will treat the following two cases: 
\begin{enumerate}
\item (1st case $v\equiv 0\pmod{p^n}$). Note that in this case for all $h$ one has $y_1(h)\notin\ZZ$. 
Note that the two following $5$-tuples can be paired:
\begin{align*}
(y_1(h),y_2(h),y_2'(h),u,0)\longleftrightarrow(-y_1(h),y_2(h),y_2'(h),-u,0).
\end{align*}
Therefore using \eqref{tonton} we see that \eqref{tile} equals zero.
\item (2nd case $v\not\equiv 0\pmod{p^n}$) In this case $y_2(h)\notin\ZZ$ and
$y_2'(h)\notin\ZZ$. Therefore we may assume that $y_1(h)\in\ZZ$. Note that the following two $5$-tuples 
can be paired:
\begin{align*}
(y_1(h),y_2(h),y_2'(h),u,v)\longleftrightarrow(-y_1(h),-y_2(h),-y_2'(h),-u,-v).
\end{align*}
Therefore using \eqref{tonton} one obtains that \eqref{tile} is equal to zero.
\end{enumerate}
A similar argument applies if one replaces $R^*$ by $\wt{R}$.
This concludes the proof of \eqref{terreur}. Finally, the fact that $\frac{1}{3}\xi(U_{u,v,n})$ is an integer 
comes from the observation that 
$$
4R^*(y_1(h),y_2(h),y_2'(h),u,v)\in\ZZ
$$
and similarly if one replaces $R^*$ by $\wt{R}$. \fin

%% file: comp_p_units_arxiv.bbl
\begin{thebibliography}{Cha09d}

\bibitem[Cha]{Ch_T}
H.~Chapdelaine.
\newblock Elliptic units in ray class fields of real quadratic number fields,
  version with a few corrections and supplements.
\newblock available at
  \texttt{http://www.mat.ulaval.ca/fileadmin/Pages{\_}personnelles{\_}des{\_}p%
rofs/hchapd/thesis{\_}final.pdf}.

\bibitem[Cha07]{Ch}
H.~Chapdelaine.
\newblock {\em Elliptic units in ray class fields of real quadratic number
  fields}.
\newblock PhD thesis, McGill University, 2007.

\bibitem[Cha09a]{Ch3}
H.~Chapdelaine.
\newblock Computation of $p$-units in ray class fields or real quadratic number
  fields.
\newblock {\em {M}ath. {C}omp., 34 pages}, 78:2307--2345, 2009.

\bibitem[Cha09b]{erratum_1}
H.~Chapdelaine.
\newblock Erratum to {$p$}-units in ray class fields of real quadratic fields.
\newblock {\em submitted to Compositio Math., 5 pages}, 1:6, 2009.

\bibitem[Cha09c]{Ch4}
H.~Chapdelaine.
\newblock Functional equation for partial zeta functions twisted by additive
  characters.
\newblock {\em {A}cta {A}rith., 16 pages}, 136:213--228, 2009.

\bibitem[Cha09d]{Ch1}
H.~Chapdelaine.
\newblock $p$-units in ray class fields of real quadratic number fields.
\newblock {\em {C}ompositio {M}ath.}, 145:364--392, 2009.

\bibitem[Cha10]{Ch8}
H.~Chapdelaine.
\newblock Some arithmetic properties of partial zeta functions weighted by sign
  characters,.
\newblock {\em J. {N}umber {T}heory}, 130:803--814, 2010.

\bibitem[Coa77]{Co}
J.~Coates.
\newblock {\em On $p$-adic {L}-functions and {I}wasawa's theory}.
\newblock Academic Press, 1977.

\bibitem[Dar04]{Dar2}
H.~Darmon.
\newblock {\em Rational points on modular elliptic curves}.
\newblock AMS Publication, 2004.

\bibitem[Das07]{Das2}
S.~Dasgupta.
\newblock Computations of elliptic units for real quadratic number fields.
\newblock {\em Canadian Journal of Mathematics}, pages 553--574, 2007.

\bibitem[Das08]{Das3}
S.~Dasgupta.
\newblock Shintani zeta functions and {G}ross-{S}tark units for totally real
  fields.
\newblock {\em Duke Mathematical J.}, 143:225--279, 2008.

\bibitem[DD06]{Dar-Das}
H.~Darmon and S.~Dasgupta.
\newblock Elliptic units for real quadratic fields.
\newblock {\em Annals of Mathematics (2)}, 163:301--346, 2006.

\bibitem[DR80]{D-R}
P.~Deligne and K.A. Ribet.
\newblock Values of abelian {$L$}-functions at negative integers over totally
  real fields.
\newblock {\em Inventiones Math.}, 59:227--286, 1980.

\bibitem[Gro81]{Gr1}
B.~Gross.
\newblock $p$-adic {$L$}-series at s=0.
\newblock {\em J. Fac. Sci. Univ. Tokyo}, 28:979--994, 1981.

\bibitem[Gro88]{Gr2}
B.~Gross.
\newblock On the values of abelian {L}-functions at {$s=0$}.
\newblock {\em J. Fac. Sci. Univ. Tokyo}, 35:177--197, 1988.

\bibitem[Hal85]{Hal}
U.~Halbritter.
\newblock Some new reciprocity formulas for generalized {Dedekind} sums.
\newblock {\em Results Math.}, 8:21--46, 1985.

\bibitem[Kli62]{Kli}
H.~Klingen.
\newblock Uber die {W}erte der {D}edekindschen {Z}eta funktionen.
\newblock {\em Math. Ann.}, 145:265--272, 1962.

\bibitem[Neu99]{Neu}
J.~Neukirch.
\newblock {\em Algebraic number theory}.
\newblock Springer-Verlag Berlin Heidelberg, 1999.

\bibitem[Shi76]{Shin}
T.~Shintani.
\newblock On evaluation of zeta functions of totally real algebraic number
  fields at non positive integers.
\newblock {\em J. Fac. Sci. Univ. Tokyo Sect. 1A}, 23:393--417, 1976.

\bibitem[Sie68]{Sie2}
C.L. Siegel.
\newblock {Bernoullische Polynome und quadratische Zahlk\"{o}rper}.
\newblock {\em {Nach. Akad. Wiss. G\"{o}ttingen Math.-Phys. K1. II}}, pages
  7--38, 1968.

\bibitem[Sie69]{Sie3}
C.L. Siegel.
\newblock Berechnung von {Z}etafunktionen an ganzzahligen {S}tellen.
\newblock {\em {Nach. Akad. Wiss. G\"{o}ttingen Math-Phys. Kl. II}}, pages
  87--102, 1969.

\end{thebibliography}
